\documentclass[10pt]{amsart}
\usepackage{hyperref}
\hypersetup{bookmarksdepth=2} 
\usepackage[dvipsnames]{xcolor}
\usepackage[
    backend=bibtex,
    style=numeric,
    natbib=false,
    url=false, 
    doi=false,
    eprint=false,
    maxnames=50
]{biblatex}
\usepackage[capitalize]{cleveref}

\hypersetup{
colorlinks=true, breaklinks=true,bookmarksnumbered,
urlcolor=webbrown, linkcolor=RoyalBlue, citecolor=OliveGreen, 
pdftitle={}, 
pdfauthor={\textcopyright}, 
pdfsubject={}, 
pdfkeywords={}, 
pdfcreator={pdfLaTeX}, 
pdfproducer={LaTeX with hyperref and ClassicThesis} 
}

\makeatletter
\def\@tocline#1#2#3#4#5#6#7{\relax
  \ifnum #1>\c@tocdepth 
  \else
    \par \addpenalty\@secpenalty\addvspace{#2}%
    \begingroup \hyphenpenalty\@M
    \@ifempty{#4}{%
      \@tempdima\csname r@tocindent\number#1\endcsname\relax
    }{%
      \@tempdima#4\relax
    }%
    \parindent\z@ \leftskip#3\relax \advance\leftskip\@tempdima\relax
    \rightskip\@pnumwidth plus4em \parfillskip-\@pnumwidth
    #5\leavevmode\hskip-\@tempdima
      \ifcase #1
       \or\or \hskip 1em \or \hskip 2em \else \hskip 3em \fi%
      #6\nobreak\relax
    \dotfill\hbox to\@pnumwidth{\@tocpagenum{#7}}\par
    \nobreak
    \endgroup
  \fi}
\makeatother

\usepackage{tikz,amsthm,amsmath,amstext,amssymb,amscd,epsfig,euscript, mathrsfs, dsfont,pspicture,multicol, mathtools,graphpap,graphics,graphicx,times,enumerate,subfig,sidecap,wrapfig,color}

\usepackage{ hyperref}
\usepackage{enumerate}
\usepackage{esint}

 \numberwithin{equation}{section}

\DeclareMathOperator{\Gr}{Gr}
\DeclareMathOperator{\gr}{gr}

\DeclareMathOperator*{\lip}{Lip_1^+}

\def\N{{\mathbb{N}}}

\def\dd{{\mathfrak{d}}}

\renewcommand{\emptyset}{\varnothing}

\def\ve{\varepsilon}

\DeclareMathOperator{\diam}{diam}

\def\Tan{\mathrm{Tan}} 					
\def\eu{\mathrm{eu}} 	

\newcounter{eps}
\newcommand{\newep}{\refstepcounter{eps}\ensuremath{\varepsilon_{\theeps}}}
\newcommand{\oldep}[1]{\ensuremath{\varepsilon_{\ref{#1}}}}

\def\BMO{\mathop\mathrm{BMO}} 					
\def\Lip{\mathop\mathrm{Lip}} 						

\def\dim{\mathrm{dim}} 					
\def\dist{\textup{dist}} 						
\def\ker{\mathop\mathrm{ker }}						
\def\supp{\mathop\mathrm{supp}}					

\def\XXint#1#2#3{{\setbox0=\hbox{$#1{#2#3}{\int}$ }
\vcenter{\hbox{$#2#3$ }}\kern-.58\wd0}}

\def\grad{\nabla}

\theoremstyle{plain}
\newtheorem{theorem}{Theorem}

\newtheorem{corollary}[theorem]{Corollary}

\newtheorem{lemma}[theorem]{Lemma}

\newtheorem{proposition}[theorem]{Proposition}

\theoremstyle{definition}

\newtheorem{definition}[theorem]{Definition}
\newtheorem{remark}[theorem]{Remark}

\numberwithin{equation}{section}
\numberwithin{theorem}{section}
\newtheorem{thmx}{Theorem}

  \DeclareFontFamily{U}{mathb}{\hyphenchar\font45} 
\DeclareFontShape{U}{mathb}{m}{n}{
      <5> <6> <7> <8> <9> <10> gen * mathb
      <10.95> mathb10 <12> <14.4> <17.28> <20.74> <24.88> mathb12
      }{}
\DeclareSymbolFont{mathb}{U}{mathb}{m}{n}

\DeclareMathSymbol{\toitself}      {3}{mathb}{"FD}  

\newcommand{\vv}{\vspace{2mm}}
\newcommand{\vvv}{\vspace{4mm}}

\def\R{\mathbb{R}}
\def\vphi{\varphi}

\begin{document}

\title[On the density problem in the parabolic space]{On the density problem in the parabolic space}

\author{Andrea Merlo}
\address{Departamento de Matem\'aticas, Universidad del Pa\' is Vasco, Barrio Sarriena s/n 48940 Leioa, Spain. }
\email{andrea.merlo@ehu.eus}

\author[Mihalis Mourgoglou]{Mihalis Mourgoglou}
\address{Departamento de Matem\'aticas, Universidad del Pa\' is Vasco, Barrio Sarriena s/n 48940 Leioa, Spain and\\
Ikerbasque, Basque Foundation for Science, Bilbao, Spain.}
\email{michail.mourgoglou@ehu.eus}

\author{Carmelo Puliatti}
\address{Departamento de Matem\'aticas, Universidad del Pa\' is Vasco, Barrio Sarriena s/n 48940 Leioa, Spain. }
\email{carmelo.puliatti@ehu.eus}

\thanks{During the writing of this work	  A.M.~was supported by the Simons Foundation grant 601941, GD., by the Swiss National Science Foundation
(grant 200021-204501 `\emph{Regularity of sub-Riemannian geodesics and
applications}')
and by the European Research Council (ERC Starting Grant 713998 GeoMeG and by the European Union’s Horizon Europe research and innovation programme under the Marie Sk\l odowska-Curie grant agreement no 101065346.
M.M. was supported by IKERBASQUE and partially supported by the grant PID2020-118986GB-I00 of the Ministerio de
Economía y Competitividad (Spain), and by IT-1247-19 and IT-1615-22 (Basque Government). C.P. was supported by the grants IT-1247-19 and IT-1615-22 (Basque Government), PGC2018-094522-B-I00 (MICINN, Spain), and partially supported by PID2020-118986GB-I00 (Ministerio de Economía y Competitividad, Spain).}

\newcommand{\mih}[1]{\marginpar{\color{red} \scriptsize \textbf{Mih:} #1}}
\newcommand{\car}[1]{\marginpar{\color{blue} \scriptsize \textbf{Carmelo:} #1}}
\newcommand{\andrea}[1]{\marginpar{\color{blue} \scriptsize \textbf{Andrea:} #1}}
\maketitle

\begin{abstract}
In this work we extend many classical results concerning the relationship between densities, tangents and rectifiability to the parabolic spaces, namely $\R^{n+1}$ equipped with parabolic dilations. In particular we prove a Marstrand-Mattila rectifiability criterion for measures of general dimension, we provide a characterisation through densities of intrinsic rectifiable measures, and we study the structure of $1$-codimensional uniform measures.
Finally, we apply some of our results to the study of a quantitative version of parabolic rectifiability: we prove that the weak constant density condition for a $1$-codimensional Ahlfors-regular measure implies the bilateral weak geometric lemma.
\end{abstract}

\tableofcontents

\section{Introduction}

Rectifiablity is the central concept of Geometric Measure Theory. Although the modern formal definition was given by H. Federer in the 60s (see \cite{Federer1996GeometricTheory}) the idea of relaxing the notion of smooth surface to that of countable union of Lipschitz images was already present in the seminal works \cite{Besicovitch1,Besicovitch2,Besicovitch3} of A. Besicovitch. Henceforth, the characterization of rectifiability via both analytic and geometric properties has produced an extensive amount of research. We refer for instance to the books \cite{Federer1996GeometricTheory}, \cite{Mattila1995GeometrySpaces}, the recent survey \cite{Mattila_survey}, and the references therein.




Since Besicovitch's foundational works, one of the key problems of Geometric Measure Theory was to determine the geometric structure of measures with density, namely Radon measures $\phi$ on $\R^n$ such that the limit 
\begin{equation}
    \lim_{r\to 0}\frac{\phi(B(x,r))}{r^\alpha}
    \label{BP}
\end{equation}
exists positive and finite for some $\alpha>0$ and for $\phi$-almost every $x\in \R^n$. This query, usually referred to as \emph{the density problem}, has driven a lot of research during the 60s and 70s (see for instance \cite{marstrand, mattila1975}), and was finally solved in the celebrated D. Preiss's work \cite{Preiss1987GeometryDensities}.
Its relevance is even broader, since it led the main contributors to its solution in $\mathbb R^n$, namely A. Besicovitch, J. Marstrand, P. Mattila, and D. Preiss, to develop many of the tools, blowup analysis arguments, and results that became fundamental techniques in the study of local properties of measures.

The solution of the Euclidean density problem is summarised in the following theorem, which is commonly known as \textit{Preiss's rectifiability criterion.}

\begin{thmx}\label{preiss}
Let $\phi$ be a Radon measure on $\R^n$ and let $\alpha\geq 0$. The following statements are equivalent:
\begin{itemize}
    \item[(i)]$\alpha\in\{0,1,\ldots,n\}$, $\phi$ is absolutely continuous with respect to the Hausdorff measure $\mathcal{H}^{\alpha}$, and $\phi$-almost all of $\mathbb{R}^n$ can be covered with countably many Lipschitz images of $\mathbb R^\alpha.$
    \item[(ii)]The limit in \eqref{BP} exists positive and finite $\phi$-almost everywhere, where $B(x,r)$ is the Euclidean ball of centre $x$ and radius $r>0$.
\end{itemize}
\end{thmx}

As mentioned above, the proof of Theorem \ref{preiss} was a community effort. The case $n=2$ was completely solved by Besicovitch in \cite{Besicovitch1,Besicovitch2,Besicovitch3}. Marstrand proved in \cite{marstrand, marstranddensity} that if \eqref{BP} exists positive and finite in an arbitrary dimension $n$, then $\alpha$ has to be an integer. He further showed that Theorem \ref{preiss} holds for $n=3$ and $\alpha=2$. Ten years later, Mattila proved in \cite{mattila1975} that if $E\subset \mathbb R^n$ is measurable, $\phi=\mathcal{H}^\alpha\llcorner E$, and the limit in \eqref{BP} equals $1$ $\mathcal{H}^\alpha\llcorner E$-almost everywhere, then $E$ must be rectifiable. In order to achieve this result, Mattila obtained in full generality another keystone of Geometric Measure Theory which is often referred to as \textit{Marstrand-Mattila rectifiability criterion} and allows one to infer the global property of rectifiability of a measure from a pointwise information which involves proper flatness conditions. We refer the reader to  \cite[Section 5]{DeLellis2008RectifiableMeasures} for more details and a proof. Finally, Preiss's contribution to Theorem \ref{preiss} was to find a way to link the information given by the existence of \eqref{BP} to the local regularity and flatness of the measure $\phi$.
The importance of \cite{Preiss1987GeometryDensities} further stems from the fundamental techniques which were introduced there, such as the theory of \textit{tangent measures}. These methods played a crucial role also in many modern applications such as the study of harmonic measure in \cite{KPT, Badger, AMT_CPAM, AMTV, Azzam_Mourgoglou}.

Because of the wide range of possible applications in the theory of optimal transport and  PDEs, the last twenty years have witnessed a remarkable surge in interest for an extension of the results of Geometric Measure Theory in the context of metric spaces. For some relevant examples we refer to \cite{Ambrosio2000RectifiableSpaces, metriccurrents, kirchharea, cheeger}. A family of metric spaces which has received a special attention are \textit{Carnot groups}, namely simply connected nilpotent Lie group whose Lie algebra is stratified and generated by its first layer (see for instance \cite{SerraCassano_notes_2016}). Their main relevance is twofold: first, they are infinitesimal models of geodesic metric spaces (see \cite{carnottangenti}) and, secondly, all non-commutative Carnot groups are $1$ and $2$-codimensional purely unrectifiable metric spaces, in the sense that they do not contain Lipschitz images of codimension $1$ or $2$. For the latter reason it is apparent that if one wants to extend the geometric techniques to these settings, at least in codimension $1$, the classical theory is of no help.

A significant step toward a better understanding of rectifiability in Carnot groups was made in the early 2000's by B. Franchi, R. Serapioni, and F. Serra Cassano, who introduced in \cite{FSSC01} an intrinsic notion of rectifiability in the Heisenberg group and achieved an analogue of De Giorgi's rectifiability theorem.
In particular, they formulated it in terms of so-called $\mathcal C^1_{\mathbb H}$-regular surface (see \cite[Definition 6.1]{FSSC01}), which admit an associated Implicit Function Theorem (see \cite[Theorem 6.5]{FSSC01}).
The first named author has recently proved in \cite{MerloG1cod} and \cite{MerloMM} that $\mathcal C^1_{\mathbb H}$-rectifiability is indeed the correct framework to provide the first non-Euclidean analogue of Preiss's theorem.
His result was attained for 1-codimensional measures in $\mathbb H^n$.
Motivated by \cite{MerloMM}, he has also formulated a notion of $\mathcal P^*$-rectifiability (resp. $\mathcal P$-rectifiability) in general Carnot groups which is based on the almost everywhere flatness (resp. flatness \text{and uniqueness}) of tangent measures, which he further investigated in collaboration with G. Antonelli in \cite{antonelli2022rectifiable}, \cite{MR4342997}, and \cite{MerloAntonelli}.
As a consequence of the works \cite{antonelli2022rectifiable}, \cite{Chousionis2015MarstrandsGroup}, \cite{ChousionisONGROUP}, \cite{MerloG1cod}, \cite{MerloMM} of G. Antonelli, V. Chousionis, V. Magnani, J. Tyson, and the first named author, Preiss's density theorem is also known to hold in $\mathbb H^1$ for higher codimensional measures. For a statement and a more detailed discussion we refer to the introduction of  \cite{antonelli2022rectifiable}.

\medskip

The purpose of the present work is to investigate the relationship between \textit{densities} and \textit{regularity} of measures in the non-Euclidean setting of \emph{parabolic spaces} $\mathbb{P}^n$ (see \S\ref{parabgroup} for the definition), to solve the density problem for $1$-codimensional measures on $\mathbb{P}^n$. Along the way, we provide the parabolic analogues of several fundamental tools in Geometric Measure Theory.

Before delving into the explanation of our contributions to the matter, we give a short account of what has been done so far in $\mathbb{P}^n$. Motivated by the study the heat equation on time-varying domains, S. Hofmann, J. Lewis, and K. Nystr\"om introduced in \cite{HLN_AASFM} and  \cite{HLN_Duke} a definition of \emph{quantitative (uniform) $(n+1)$-rectifiability} in the parabolic space, which is inspired by the works of David and Semmes' in the Euclidean case (see \cite{DavidSemmes_asterisque} and \cite{DavidSemmes}). Its investigation constitutes an active field of research (see \cite{BHHLN1, BHHLN2, BHHLN3}). We also mention that the need of a better understanding of qualitative parabolic rectifiability arises also from other recent studies beyond geometry.
Parabolic tangent measures and some of the techniques of \cite{Preiss1987GeometryDensities} have been implemented in the study of non-variational two-phase problems for caloric measure by the second and third named author in \cite{Mourgoglou_Puliatti}. Moreover, J. Mateu, L. Prat, and X. Tolsa investigated in \cite{Mateu_Prat_Tolsa} a parabolic Lipschitz-harmonic capacity in the context of removability for Lipschitz caloric functions. 

In order to bridge clearer connections between parabolic uniform rectifiability and the geometry of $\mathbb{P}^n$, P. Mattila formulated in \cite{MatParRect} the notion of \textit{LG-rectifiabile set} in the spirit of Federer's definition (see Definition \ref{def:LGrect}).
This notion, which is extremely geometric and natural, allowed Mattila to extend to $\mathbb P^n$ some of the classical Euclidean characterizations of rectifiable sets, in particular those in terms of approximate tangent planes and tangent measures, see Theorem \ref{theorem_Mattila_parabolic_rect} or \cite{MatParRect} for more details. In fact, Theorem \ref{theorem_Mattila_parabolic_rect} tells us that the notion of $LG$-rectifiability coincides with that of 
$\mathcal P_m$-rectifiability, introduced by the third named author in the context of Carnot groups in \cite{MerloMM} (see Definition \ref{def:PhRectifiableMeasure}).

\medskip

We now proceed with a description of the content of the paper and its connection with the existing literature. As mentioned above, the main subject is the density problem in $\mathbb{P}^n$, which we solve in codimension $1$.

\begin{theorem}\label{theorem:main_theorem_Preiss}
    Suppose that $\phi$ is a Radon measure on $\mathbb{P}^n$. Then, the following are equivalent:
    \begin{itemize}
        \item[(i)] $\phi$ is absolutely continuous with respect to $\mathcal{H}^{n+1}$ and it is $\mathscr{P}_{n+1}$-rectifiable in the sense of Definition \ref{def:PhRectifiableMeasure}. In other words, there are countably many compact sets $K_i\subseteq \mathcal{V}$ and Lipschitz maps $g_i\colon K_i\to \R$ with the Rademacher property (see Definition \ref{def:Rademacher_property}) such that $$\phi\Bigl(\mathbb{P}^n\setminus \bigcup_{i\in\N} g_i(K_i)\Bigr)=0.$$
        \item[(ii)]$\phi$ is absolutely continuous with respect to $\mathcal{H}^{n+1}$ and it is supported on a $1$-codimensional LG-rectifiable set.
        \item[(iii)] For $\phi$-almost every $x\in\mathbb{P}^n$ we have
        \[
        0<\lim_{r\to 0}\frac{\phi(B(x,r))}{r^{n+1}}<\infty,
    \]
where $B(x,r)$ is the ball relative to the Koranyi metric (see \eqref{eq:Korany_metric}).
    \end{itemize}
\end{theorem}

\vv

Theorem \ref{theorem:main_theorem_Preiss} is a parabolic analogue of Theorem \ref{preiss} and of \cite[Theorem 1.3]{MerloG1cod}; we show that Mattila's notion of LG-rectifiability is the one which solves the $1$-codimensional density problem in $\mathbb{P}^n$ for the Koranyi metric. This is quite remarkable both from the theoretical point of view, as it provides a \emph{metric} reason why the notion of LG-rectifiability is natural and, on the other hand, together with results like \cite[Example 8.2]{MatParRect}, it highlights a contrast between the notion of parabolic uniform rectifiability and the geometry of the parabolic space. We shall discuss the latter issue in detail below.

It is legitimate to ask whether Theorem \ref{theorem:main_theorem_Preiss} keeps being true if we replace the Koranyi norm with a bi-Lipschitz equivalent one. The next result shows that the solution to the density problem in the parabolic space is extremely sensitive to the change of norms. Indeed, in Appendix \ref{section:counterexample_PDT} we provide a \textit{counterexample to the parabolic Preiss's density theorem} with an alternative metric.

\begin{theorem}\label{th:failurepreiss}
Endow the parabolic group $\mathbb{P}^1$ with the metric induced by the norm \footnote{Note that the metrics induced by the Koranyi norm (see \eqref{eq:Korany_metric}) and that induced by $\lVert\cdot\rVert_\infty$ are bi-Lipschitz equivalent.} 
\begin{equation}
       \|x\|_\infty \coloneqq \max\{|x_H|, |x_T|^{1/2}\} \qquad \text{ for }\quad x=(x_H,x_T)\in \R\times \R
    \label{infinitynorm}
\end{equation}
and let $\mathfrak B_\infty(x,r)$ denote the ball relative to \eqref{infinitynorm}.
Then, there exists a Radon measure $\nu$ on $\mathbb{P}^1$ which satisfies
\[
\nu(\mathfrak B_\infty(x,r))=r^2\qquad \text{for any $x\in\supp(\nu)$ and any $r>0$}
\]
and such that the set of Preiss's tangent measures $\Tan_2(\nu,x)$ (see Definition \ref{def:TangentMeasure}) for $\nu$-almost every $x\in\mathbb{P}^1$ {never contains} the Haar measure of the vertical line $\mathcal V\coloneqq \{(0,s):s\in\R\}$.
\end{theorem}
\vv

It is not hard to show that $\mathcal V$ is the only $2$-dimensional homogeneous subgroup of $\mathbb{P}^1$ (see Corollary \ref{corollary:grn1}).
Thus, Theorem \ref{th:failurepreiss} proves the existence of a Radon measure having $2$-dimensional density with respect to the standard parabolic cylinder $\mathfrak B_\infty(x,r)$ such that $\nu$ never resembles a flat parabolic surface at any scale and at $\nu$-almost every point of $\mathbb{P}^1$. Together with Theorem \ref{theorem:main_theorem_Preiss}, the above results also prove that $\mathbb{P}^n$ is a metric space for which two bi-Lipschitz equivalent metrics give different solutions to the density problem, and, to our knowledge, this is the first instance where such a phenomenon has ever been observed. We also mention that Theorem \ref{th:failurepreiss} shows one of the reasons why it is so difficult to prove rectifiability criteria in the spirit of Theorem \ref{preiss}: the class of measures for which the density exists \emph{depends on the metric}, so all the arguments must heavily rely on the shape of the ball. 

\vv
Let us discuss the strategy of the proof of Theorem \ref{theorem:main_theorem_Preiss}, that can be divided into two main steps. The first is to prove the Marstrand-Mattila rectifiability criterion in $\mathbb{P}^n$, which holds for general homogeous and translation-invariant metrics on $\mathbb P^n$, and links the local structure of a measure to its global regularity properties.

\begin{theorem}\label{thm:MMconormale:intro}
Let $\mathfrak d$ be a homogeneous metric on $\mathbb P^n$ which is invariant under translations, see \eqref{eq:homogeous_metric} and \eqref{eq:transl_inv_d}. Let $\phi$ be a Radon measure on $(\mathbb{P}^{n}, \mathfrak d)$ such that:
\begin{itemize}
    \item[(i)] There exists $h\in\{0,\ldots,n+2\}$ such that for $\phi$-almost every $x\in \mathbb{P}^n$ we have $$0<
    \liminf_{r\to 0}\frac{\phi(B_{\mathfrak d}(x,r))}{r^{h}}\leq\limsup_{r\to 0}\frac{\phi(B_{\mathfrak d}(x,r))}{r^{h}},
    <\infty,$$
    where $B_{\mathfrak d}(x,r)\coloneqq \{y\in \mathbb P^n:\mathfrak d(x,y)\leq r\}$.
    \item[(ii)] For $\phi$-almost every $x\in \mathbb{P}^n$  the elements of $\Tan_h(\phi,x)$ contained in the family of $h$-dimensional flat measures $\mathfrak{M}(h)$, namely the family of Haar measures relative to $h$-dimensional homogeneous subgroups of $\mathbb{P}^n$.
\end{itemize}

Then there are countably many $h$-dimensional homogeneous subgroups $W_i$ of $\mathbb{P}^n$, compact sets $K_i\subset   W_i$, and differentiable Lipschitz maps $f_i\colon K_i\to \mathbb{P}^n$ such that
\[
    \phi\Bigl(\mathbb{P}^n\setminus \bigcup_{i\in\N} f_i(K_i)\Bigr)=0.
\]
\end{theorem}
\vv

With the above theorem in hand, in Proposition \ref{preiss1} we prove Theorem \ref{theorem:main_theorem_Preiss} in the case $n=1$.
 Theorem \ref{thm:MMconormale:intro} tells us that if a measure has dimension $h$, which is item (i), and it is well approximated by planes at small scales, which is item (ii), then it essentially is the surface measure of a differentiable Lipschitz graph of dimension $h$. We shall remark, however, that in the statement of Theorem \ref{thm:MMconormale:intro} we allow the plane which approximates the measure to possibly change at different scales, in a similar fashion to what happens at the centre of an exponential spiral. Hence, to prove rectifiability with such weak hypotheses is significantly harder than in presence of a unique blowup, which is commonly referred to as the characterisation in terms of approximate tangent plane (see \cite[Chapter 3]{Federer1996GeometricTheory}). 
 
 The first step to achieve Theorem \ref{thm:MMconormale:intro} is to show that items (i) and (ii) imply that $\mathbb{P}^n$ can be covered $\phi$-almost all by countably many Lipschitz images of compact subsets of an $h$-dimensional homogeneous subgroup of $\mathbb{P}^n$. To do this, we use techniques analogous to those that Federer developed to obtain his celebrated projection theorem (see \cite{Federer_TAMS_1947}). In order to prove that the Lipschitz graphs $\Gamma_i$ previously constructed are differentiable, we show that the locality of tangent measures given by Proposition \ref{tuttitg} yields that the tangents relative to the natural surface measures of $\Gamma_i$ are flat almost everywhere. The proof that such blowup is unique is a (much) more delicate version of the uniqueness of the differential (see Proposition \ref{mm0.0}). 
 
\textbf{If not specified, throughout the rest of this introductory section we always understand that $\mathbb{P}^n$ is endowed with the Koranyi metric.}
 
 The second step to prove Theorem \ref{theorem:main_theorem_Preiss} also in the case $n>1$ can be summarized with the following crucial result, which shows that a measure in the parabolic space with $1$-codimensional density has flat tangents almost everywhere.
 
 \begin{theorem}\label{preissprabolicointrovero}
     Suppose that $\phi$ is a Radon measure on $\mathbb P^n$ such that 
       \begin{equation}
          0<\lim_{r\to 0}\frac{\phi(B(x,r))}{r^{n+1}}<\infty,\qquad\text{for $\phi$-almost every $x\in\mathbb{P}^n$. }
 \label{eq:BPdensity}
       \end{equation}
            Then item (ii) in Theorem \ref{thm:MMconormale:intro} holds true. 
 \end{theorem}
 \vv

 To prove that (iii) implies (i) in Theorem \ref{theorem:main_theorem_Preiss} it is enough to combine Theorems \ref{thm:MMconormale:intro} and \ref{preissprabolicointrovero}, while the other implications follow almost immediately from Mattila's work \cite{MatParRect}. Theorem \ref{preissprabolicointrovero} is the technical core of the paper and the results employed in its proof have several by-products, as we detail below. The main idea of its proof is to exploit the weak* connectedness of $\Tan_{n+1}(\mu,x)$ and show the following result. 

Let us recall that a Radon measures $\nu$ on $\mathbb{P}^n$ is called $(n+1)$-\textit{uniform} if $\nu(B(x,r))=r^{n+1}$ for any $x\in\supp(\nu)$ and any $r>0$.
	\begin{theorem}\label{th:disconnection:intro}
    Let $\phi$ be a Radon measure on $\mathbb P^n$ such that
    $$0<\lim_{r\to 0}\frac{\phi(B(x,r))}{r^{n+1}}<\infty\qquad  \text{for $\phi$-almost every $x\in\mathbb{P}^n$}.$$
    Suppose further that
    \begin{itemize}
        \item[(\textbf{P})] there exist a weak* continuous functional $\mathscr{F}$ on the space of Radon measures and a constant $\varepsilon>0$ depending only on $n$ such that, if $\mu$ is an $(n+1)$-uniform measure on $\mathbb P^n$ and $\mathscr{F}(\nu)\leq\varepsilon$
    for some tangent at infinity $\nu$ of $\mu$ (see Definition \ref{def:blowdowns}), then $\mu$ is flat.
    \end{itemize}
    Then item (ii) in Theorem \ref{thm:MMconormale:intro} holds true. 
\end{theorem}
\vv

It is easy to see that if a measure satisfies  \eqref{eq:BPdensity}, then its tangent measures are uniform at $\phi$-almost every point. This information is not sufficient to conclude a priori that \eqref{eq:BPdensity} implies the flatness of blowups. Indeed, it is very well known that there are uniform measures in $\R^n$ which are not flat (see for instance \cite{Kowalski1986Besicovitch-typeSubmanifolds} and \cite{nimer}) so the proof of Theorem \ref{preissprabolicointrovero} requires a finer understanding of the geometric structure of uniform measures.

Condition (\textbf{P}) should be interpreted as the property that the class of $1$-codimensional uniform measures can be split into flat and non-flat measures, and those two families are weak* disconnected. Moreover, the fact that $\Tan_{n+1}(\phi,x)$ is weak* connected for $\phi$-almost every $x\in\mathbb P^n$ implies that either the measures in $\Tan_{n+1}(\phi,x)$ are flat or $\Tan_{n+1}(\phi,x)$ is contained in the set of non-flat uniform measures $\phi$-almost everywhere. However, since uniform measures are surface measures of $1$-codimensional analytic manifolds in $\R^{n+1}$ and thanks to the fact that the tangent measures of a tangent measure are themselves tangent measures by Proposition \ref{tuttitg}, we infer that $\Tan_{n+1}(\phi,x)$ contains a flat measure for $\phi$-almost every $x\in \mathbb P^n$. The above argument shows that item (ii) of Theorem \ref{thm:MMconormale:intro} holds.

Given the close relation of the spaces $\mathbb{H}^n$ and $\mathbb{P}^n$, one could expect that either the strategy to prove (\textbf{P}) is quite similar in the two cases or that the property in the parabolic space can be derived from the proof in Heisenberg groups. However, this is far from being true and the proof of (\textbf{P}) turned out to be extremely technical.

The first step in order to deduce (\textbf{P}) is to get some weak geometric information on the support of uniform measures. We follow the strategies of \cite{Preiss1987GeometryDensities} and \cite{MerloG1cod},  and study the so-called \textit{moments} (see Section \ref{section:moments_expansion}). 
The computations in \cite{MerloG1cod} which adapt Preiss's moments to the Heisenberg groups and analyze their Taylor expansion can be modified and adapted to get the following result.

\begin{theorem}\label{theorem:1_codim_unif_measures}
    Let $\nu$ be a $1$-codimensional uniform measure on $\mathbb{P}^n$. Then either there are $b\in\R^n$, $\tau\in\R$, and a symmetric non-zero matrix $\mathcal Q\in \R^{n\times n}$ such that
   \begin{equation}
        \supp(\mu)\subseteq \bigl\{x\in\mathbb{P}^n:\langle x_H,\mathcal Q[x_H]+b\rangle+\tau x_T=0\bigr\},
        \label{supportquadratiic}
   \end{equation}
   or there exists $C(n+1)>0$ such that
   \begin{equation}
  \begin{split}
   0=\lim_{s\to 0}&\frac{8 s^{\frac{3}{2}+\frac{n+1}{4}}}{C(n+1)}\int \lvert z_H\rvert^4 z_H\otimes z_H
					e^{-s\lVert z\rVert^4}\,d\mu(z)\\
			&\qquad-\frac{s^{\frac{1}{2}+\frac{n+1}{4}}}{C(n+1)}\int (4z_H\otimes z_H+2\lvert z_H\rvert^2\textrm{id}_{n}) e^{-s\lVert z\rVert^4}\, d\mu(z).	\label{eq:degeneratemeasures}
  \end{split}
   \end{equation}
   
\end{theorem}
\vv

Theorem \ref{theorem:1_codim_unif_measures} shows how far the geometry of the parabolic space is from that of $\mathbb R^n$ and $\mathbb H^n$. Indeed, the solutions of the Euclidean and Heisenberg density problems heavily rely on the rigidity which follows from the fact that, in those spaces, uniform measures are supported on the zero set of quadratic polynomials (see \cite{MerloG1cod} and \cite{Preiss1987GeometryDensities}). However, we have not been able to show any analogue of such result in $\mathbb{P}^n$.

This forced our approach to diverge significantly from those previously employed.  
As seen above, in order to prove Theorem \ref{preissprabolicointrovero} we just need to show that, inside the set of uniform measures, flat measures are weak* disconnected from the non-flat ones. Let $\mathscr{F}\colon \mathcal{M}\to\R^+$  be the functional 
 \begin{equation}\label{eq:functional_F_intro}
     \mathscr{F}(\nu)\coloneqq\inf_{u\in\mathbb{S}^{n-1}} 
        \int \lvert z_H\rvert^4\langle z_H,u\rangle^2
					e^{-\lVert z\rVert^4}\, d\nu(z).
 \end{equation}  
A relatively standard computation proves that  $\mathscr{F}$ disconnects $1$-codimensional uniform measures satisfying \eqref{eq:degeneratemeasures} from the flat ones (see Proposition \ref{disconnectdegenerated}). The most complex task is then to discuss the geometry of uniform measures supported on quadrics, which one can show are surface measures associated to the quadratic surfaces. Moreover, thanks to a hard-won Taylor expansion of the perimeter measure together with some delicate algebraic considerations (see Appendix \ref{TYLR} and Subsection \ref{sectneq0notexist} respectively) we also deduce that the  quadrics that can support a uniform measure must be vertically ruled. 
Because of this reduction and the coarea-type formula in Proposition \ref{prop:reprez}, we obtain the following result.
 
 \begin{theorem}\label{classificazionemisureunif} Let $\mu$ be an $(n+1)$-uniform measure on $\mathbb P^n$. If $\mu$ is supported on a quadric, then either $\nu$ is flat or up to isometries there exists $c=c(n)>0$ such that
$$\mu=c\mathcal{H}^3\llcorner \{x_1^2+x^2_2+x_3^2=x^4\}\otimes \mathcal{L}^{n-4}\llcorner\mathrm{span}\{e_4,\ldots,e_n\}\otimes \mathcal{L}^1\llcorner \mathrm{span}\{e_{n+1}\}.$$

If, on the other hand, $\mu$ is not supported on a quadratic surface, there exists a constant $\oldep{eps:1}=\oldep{eps:1}(n+1)>0$ such that
    \begin{equation}\nonumber
        \lambda^{\frac{3}{2}+\frac{n+1}{4}}\int \lvert z_H\rvert^4\langle z_H,u\rangle^2
					e^{-\lambda\lVert z\rVert^4}\, d\nu(z)>\oldep{eps:1}
    \end{equation}
    for any $\nu\in \Tan_{n+1}(\mu,\infty)$, any $\lambda>0$ and any $u\in\R^n$ with $\lvert u\rvert=1$.
 \end{theorem}
 \vv

Theorem \ref{classificazionemisureunif} implies property (\textbf{P}), concluding the proof of Theorem \ref{preissprabolicointrovero} and hence of Theorem \ref{theorem:main_theorem_Preiss}. We also remark that we actually prove an analogue of the second part of Theorem \ref{classificazionemisureunif} for uniform measures of arbitrary codimension.

\vv
The above characterisation of uniform measures is the key tool to obtain the results outlined in the rest of the introduction, where we present some applications to the study of parabolic uniform rectifiability.

A Radon measure $\mu$ on $\mathbb P^n$ is said $(n+1)$-Ahlfors regular, or simply $(n+1)$-Ahlfors-regular, if there exists $C>0$ such that
\[
    C^{-1}r^{n+1}\leq \mu(B(x,r))\leq C r^{n+1}
\]
for all $x\in \supp(\mu)$ and $0<r<\diam (\supp(\mu)).$ 
As previously mentioned, Hofmann, Lewis, and Nystr\"om proposed a parabolic equivalent of David and Semmes' notion of uniform rectifiability of a set $\Sigma \subset \mathbb P^n$. They formulated it in terms of a Carleson condition on the measure $$d\nu(x,r)=\beta_{2,\sigma_\Sigma}(x,r)\,d\sigma_\Sigma(x)\,r^{-1}dr $$
on $\mathbb P^n\times \mathbb R_{>0}$, where $\sigma_\Sigma$ is the surface measure on $\Sigma$ and $\beta_{2,\sigma_\Sigma}$ are the (parabolic) $\beta_{2}$-numbers which quantify the flatness of ${\sigma_\Sigma}$.
A crucial role in the study of parabolic uniform rectifiability is played by \textit{regular parabolic Lipschitz functions}, namely maps $f\colon \mathbb P^{n-1}\to \mathbb R$ which are Lipschitz with respect to the usual $\ell^2$-norm in the space variables
and such that $D^{1/2}_Tf\in \BMO(\mathbb P^{n-1})$, where $D_T^{1/2}$ stands for the $1/2$-order derivative in the time variable and $\BMO(\mathbb P^{n-1})$ for the parabolic version of the space of functions with bounded mean oscillation.
Furthermore, Hofmann, Lewis, and Nystr\"om proved that a parabolic uniformly rectifiable set which is also Reifenberg flat (in a parabolic sense) has \textit{big pieces of regular parabolic Lipschitz graphs}. More recently it has been shown in \cite{BHHLN2} and \cite{BHHLN3} that a parabolic $(n+1)$-Ahlfors regular  $\Sigma$ is parabolic uniformly rectifiabile if and only if it admits a (bilateral) \textit{corona decomposition} with respect to regular parabolic Lipschitz graphs or, equivalently, if $\Sigma$ has \textit{big pieces squared} of regular parabolic Lispchitz graphs. For the precise definitions and a more in-depth discussion of these results we refer to the aforementioned papers.

 Several different characterizations of Euclidean uniform rectifiability are available in the literature. For instance, it is well known that an $n$-Ahlfors regular measure on $\mathbb R^{n+1}$ is uniformly rectifiable if and only if it satisfies the so-called \textit{bilateral weak geometric lemma} (BWGL) (see the monograph \cite{DavidSemmes}). However, it has been shown in \cite[Observation 4.19]{BHHLN3} that the BWGL in the context of the parabolic space does not imply parabolic uniform rectifiability.
As an application of our analysis of parabolic uniform measures, we prove that the parabolic \textit{weak constant density condition} (WCD) implies the BWGL (see Definitions \ref{def-WCD} and  \ref{def:BWGL} respectively). The precise definitions involve Jones' $\beta$-numbers and the lattice of dyadic cubes adapted to a regular measure so, in order not to make the introduction too lengthy, we choose to defer them until Section \ref{section:BWGL}.

\begin{theorem}\label{theorem:WCD_implies_BWGL_2}
    Let $\mu$ be an $(n+1)$-Ahlfors regular measure on $\mathbb P^n$, and assume that it satisfies the weak constant density condition. Then $\mu$ satisfies the bilateral weak geometric lemma.
\end{theorem}
\vv

The Euclidean analogue of Theorem \ref{theorem:WCD_implies_BWGL_2} was first proved by David and Semmes in \cite[Chapter III.5]{DavidSemmes}. In the parabolic case, it follows from Theorem \ref{classificazionemisureunif}, or more precisely from Lemma \ref{lemmaF->beta}, together with a fine study of the flatness properties of uniform measures in the spirit of the arguments of Tolsa in \cite{Tolsa_uniform_measures}. We remark that the approach in \cite{Tolsa_uniform_measures} uses the Riesz transform to prove that if a uniform measure is far from being flat (respectively very flat) inside a cube $Q$ then it is far from being flat (respectively very flat) in most of the subcubes of $Q$. However, in $\mathbb{P}^n$ the Riesz transform is not the right operator (compare with Proposition \ref{prop:propaux1_s}) and in order to prove these strong propagation properties of flatness, we need to investigate more carefully the geometric structure of uniform measures (see Lemma \ref{lemmaangolitangenti}).

Theorem \ref{theorem:WCD_implies_BWGL_2} provides a direct link with the upcoming article \cite{MMP2} in which we exhibit further applications of the parabolic Preiss's theory to the study the parabolic analogue of a proper quantitative version of the parabolic density problem. Hence, we find it useful to briefly outline some of its contents to put the result in a wider context. In \cite{MMP2} we are interested in establishing whether the characterisation given by Chousionis, Garnett, Le, and Tolsa of uniform rectifiability in terms of the boundedness of a proper square function involving the density given in \cite{MR3461027} can be extended to the parabolic spaces. In particular, using the results of the present paper, we are able to prove the following result.
 
 \begin{thmx}\label{densitythBWGL}
 Suppose $\mu$ is an $(n+1)$-Ahlfors-regular measure on $\mathbb{P}^n$ such that 
 \begin{equation}
      \int_{B(y,R)}\int_0^R\Big\lvert\frac{\mu(B(x,r))}{r^{n+1}}-\frac{\mu(B(x,2r))}{2^{n+1}r^{n+1}}\Big\rvert^{2}\,\frac{dr}{r}d\mu(x)\leq CR^{n+1} ,\quad \text{    for $\mu$-almost every $y\in\mathbb{P}^n$.}
      \label{sfe}
 \end{equation}
 Then $\mu$ is $\mathscr{P}_{n+1}$-rectifiable in the sense of Definition \ref{def:PhRectifiableMeasure} and it satisfies the WCD. In particular, $\mu$ satisfies the BWGL by Theorem \ref{theorem:WCD_implies_BWGL_2}.
 \end{thmx}
 \vv
 
  Theorem \ref{densitythBWGL} is in line with what happens in Euclidean space although, as mentioned above, we cannot push the regularity given by the BWGL to parabolic uniform rectifiability. The following result, however, highlights profound differences in the parabolic setting.
 
\begin{thmx}\label{counterexamples:intro}
There exists a $2$-Ahlfors-regular measure $\mu$
on $\mathbb{P}^1$ such that for any $\varepsilon>0$  there exists a constant $C_\varepsilon>0$ such that
    \begin{equation}
        \int_0^1\Big\lvert\frac{\mu(B(x,r))}{r^{2}}-\frac{\mu(B(x,2r))}{4r^{2}}\Big\rvert^{\frac{1}{2}+\varepsilon}\,\frac{dr}{r}\leq C_\varepsilon,\qquad \text{    for $\mu$-almost every $x\in \mathbb{P}^n$,}
        \label{eq:density}
    \end{equation}
and for any regular parabolic Lipschitz graph $\Gamma$ we have $\mu( \Gamma)=0$.
\end{thmx}
\vv

The idea behind Theorem \ref{counterexamples:intro} is to construct a parabolic Lipschitz function $f\colon\R\to\R$ such that, if $\Gamma$ denotes its graph, then \eqref{eq:density} holds for $\mu=\mathcal{H}^2\llcorner \Gamma$ and \begin{equation}
    \int_\R\int_\R\frac{\lvert f(\tau)-f(\sigma)\rvert^2}{\lvert \tau-\sigma\rvert^2}\,d\tau d\sigma=\infty.
    \label{eq.parabolicpip}
\end{equation}
Our construction is inspired by the one of non-differentiable continuous Weierstrass function and the techniques used in \cite[Example 8.2]{MatParRect}, and it involves delicate Fourier estimates for the co-Lipschitz constant of the function $f$. In fact, it gives much more than \eqref{eq.parabolicpip}.
We finally mention that the proof of Theorem \ref{counterexamples:intro} also requires a computation of the Hausdorff measure of balls centred on $\Gamma$, which we achieve via a version of the area formula for parabolic differentiable Lipschitz functions which is of interest on it own right (see for instance \cite{MR4342997,kirchharea, VITTONE}).

\subsection*{Structure of the paper.} \textit{Section \ref{section:Preliminaries}} collects preliminary facts and  known results on general measure theory, uniform measures, and parabolic rectifiability. In \textit{Section \ref{section:Mastrand}} we prove some additional structural properties of uniform measures. V. Chousionis and J. Tyson showed in \cite{Chousionis2015MarstrandsGroup} that the support of a uniform measure on $\mathbb P^n$ is an analytic variety, but we actually need more information; an adaptation of the techniques of \cite{KirchheimPreiss02} shows that a uniform measure actually coincides, modulo a dimensional factor, with the surface measure on its support. We further prove that the tangent measures to a uniform measure of codimension $1$ are indeed dilation-invariant, and this analysis allows us to conclude the section with the classification of uniform measures on $\mathbb P^1$ (see Proposition \ref{uniformmeasuresinP1}).

\textit{Section \ref{section:MM_rect_theorem}} is devoted to the proof of Theorem \ref{thm:MMconormale:intro}. In \textit{Section \ref{section:moments_expansion}} we introduce the parabolic version of Preiss's moments associated with a uniform measure on $\mathbb P^n$. This is done in Definition \ref{defimom} and it involves a polarization of Koranyi norm (see Proposition \ref{prop6}). Although most of the calculations mimics those in \cite{Preiss1987GeometryDensities} and \cite{MerloG1cod}, this is a pivotal component of the proof since it allows to construct a candidate for an algebraic surface which contains the support of the uniform measure (see Section \ref{subsection:expansion_moments_candidate_quadric}).

\textit{Section \ref{section:unif_measure_quadratic_surface}} contains the proof of Theorem \ref{theorem:1_codim_unif_measures}: we split the class of $(n+1)$-uniform measures into ``non-degenerate'' and ``degenerate'' measures (see \eqref{eq:degenerate}), prove that the first ones are contained in quadratic surfaces (see Proposition \ref{prop:notdeg}), and finally show that \eqref{eq:degeneratemeasures} and \eqref{eq:prop47_statement} hold for the second sub-class.
The analysis of non-degenerate $1$-codimensional uniform measures is continued in \textit{Section \ref{eq:sec_non_deg_unif_meas}}, where we prove the first part of Theorem \ref{classificazionemisureunif}. Then, in \textit{Section \ref{section:proof_main_theorem}} we study the properties of the functional $\mathscr F$ in \eqref{eq:functional_F_intro} and conclude the proof of Theorem \ref{theorem:main_theorem_Preiss}.

In \textit{Section \ref{section:BWGL}} we apply the methods described above and prove Theorem \ref{theorem:WCD_implies_BWGL_2}. The argument is based on a touching-point argument in terms of an auxiliary operator defined in \eqref{eq:definition_Rrs} which replaces the Riesz transform and allows us to infer the flatness result Corollary \ref{corollary:flat_ball_big}. Moreover, we remark that the solution of the parabolic Preiss's density problem is the key tool used in the form of Lemma \ref{lemmaF->beta}, and the proof of the BWGL follows from a stopping-time argument along the lines of \cite{Tolsa_uniform_measures}.

\textit{Appendix \ref{section:counterexample_PDT}} contains the construction of the graph which proves Theorem \ref{th:failurepreiss} and, finally, in \textit{Appendix \ref{TYLR}} we adapt the calculations in \cite{MerloG1cod} in order to study the first three non-trivial coefficients of the expansion \eqref{eq:expansion_area_formula} of the Hausdorff measure on a quadratic surface.

\vv

\section{Preliminaries and notation} \label{section:Preliminaries}

We proceed to recall well-known facts and introduce some notations.
In case the proof of a proposition is not present in literature but the Euclidean (or Heisenberg group) argument applies verbatim, we just provide a reference.

\subsection{The parabolic group \texorpdfstring{$\mathbb{P}^n$}{Lg}}
\label{parabgroup}

Let $\pi_H\colon \R^{n+1}\to\R^{n}$ be the projection onto the first $n$ coordinates and $\pi_T\colon\R^{n+1}\to \R$ be the projection onto the last one. The Lie groups $\mathbb{P}^n$ are the smooth manifolds $\R^{n+1}$ endowed with the Euclidean sum.

Given $\lambda>0$ we define the \textit{anisotropic dilations} $\delta_\lambda\colon\mathbb{P}^n\to\mathbb{P}^n$ as \[\delta_\lambda(x)\coloneqq(\lambda x_H,\lambda^2  x_T),\]\label{dilatan}
where  $x_H\coloneqq \pi_H(x)$ and $x_T\coloneqq \pi_T(x)$, and metrize $(\mathbb P^n, +)$ with a distance $\dd\colon\mathbb{P}^n\times\mathbb{P}^n\to\mathbb{P}^n$ which is \textit{homogeneous} with respect to $\delta_\lambda$, namely 
\begin{equation}\label{eq:homogeous_metric}
   \dd(\delta_\lambda(x),\delta_\lambda(y))=\lambda \, \dd(x,y)\qquad \text{ for all }x,y\in \mathbb{P}^n \text{ and }\lambda>0.
\end{equation}

Moreover we let $\lVert x\rVert_{\dd}\coloneqq \dd(x,0)$\label{Koranyinorm}, $B_{\dd}(x,r)\coloneqq\{z\in\mathbb{P}^n:\dd(z,x)\leq r\}$ be the (closed) ball, and we define $U_{\dd}(x,r)\coloneqq\{z\in\mathbb{P}^n:\dd(z,x)<r\}$. If the metric is clear from the context, we generally drop the subscript ``$\dd$''.

The \emph{Koranyi distance} is the metric
\begin{equation}\label{eq:Korany_metric}
	d(x,y)\coloneqq \bigl(\lvert y_H-x_H\rvert^4+\lvert y_T- x_T\rvert^2\bigr)^{1/4}, \qquad \text{ for }x,y\in\mathbb P^n.
\end{equation}

It is also useful to denote as $\langle x, y\rangle \coloneqq x_1 y_1 +\cdots + x_n y_n$, where $x,y\in \R^n$, the standard Euclidean scalar product and by $|\cdot|$ its associated norm. Given $x\in \R^n$ and $r>0$, we write $B_n(x,r)\coloneqq \{y\in\R^n: |x-y|\leq r\}$ for the Euclidean balls. Moreover, for a linear subspace $V$ of $\R^n$, $V^\perp$ represents its orthogonal complement in $\R^n$ with respect to the Euclidean scalar product.

Finally, for a measure $\mu$ on $\mathbb P^n$ and a set $E\subset \mathbb P^n$ we denote by $\mu\llcorner E$ the restriction
\(
    \mu\llcorner E(A)\coloneqq \mu(A\cap E),
\)
 for $A\subseteq \mathbb P^n.$
 
 \vv

\subsection{Hausdorff and surface measures}
Let $\mathfrak d(\cdot, \cdot)$ be a homogeneous metric on $\mathbb P^n$.
\begin{definition}\label{Hausdro}
	For $h\in[0,\infty)$ we define the  $h$-dimensional {\em (parabolic) spherical Hausdorff measure} on $\mathbb P^n$ relative to $\mathfrak d$ as\label{sphericaldhausmeas}
	\[
		\mathcal{S}^{h}(A)\coloneqq\sup_{\delta>0}\,\inf\bigg\{\sum_{j=1}^\infty  r_j^h:A\subseteq \bigcup_{j=1}^\infty  B(x_j,r_j),~r_j\leq\delta\bigg\}
	\]
	and the $h$-dimensional {\em Hausdorff measure}\label{hausmeas} relative to $\mathfrak d$ as
	\[
		\mathcal{H}^h(A)\coloneqq\sup_{\delta>0}\,\inf \biggl\{\sum_{j=1}^{\infty} 2^{-h}(\diam E_j)^h:A \subseteq \bigcup_{j=1}^{\infty} E_j,\, \diam E\leq \delta\biggr\}.
	\]

	We also define the $h$-dimensional {\em centered (spherical) Hausdorff measure} relative to $\mathfrak d$ as\label{centredhausmeas}
	\[
		\mathcal{C}^{h}(A)\coloneqq\underset{E\subseteq A}{\sup}\,\,\mathcal{C}^h_0(E),
	\]
	where
	\[
		\mathcal{C}^{h}_0(E)\coloneqq\sup_{\delta>0}\,\inf\biggl\{\sum_{j=1}^\infty  r_j^h:E\subseteq \bigcup_{j=1}^\infty B(x_j,r_j),~ x_j\in E,~r_j\leq\delta\biggr\}.
	\]

\end{definition}
\vv
We stress that $\mathcal{C}^h$ is an outer measure, so it defines a Borel regular measure (see \cite[Proposition 4.1]{EdgarCentered}).
\begin{remark}
	\label{remarkone}
	The spherical and centered Hausdorff measures do not coincide. However, $\mathcal{S}^h $ and $\mathcal{C}^h$ are equivalent, see \cite[Section 2.10.2]{Federer1996GeometricTheory} and \cite[Proposition 4.2]{EdgarCentered}. More precisely, a combination of Remark 2.3 and Lemma 2.5 in \cite{FSSCArea} yields
	$$2^{-h}\mathcal{S}^h\leq\mathcal{C}^h\leq 2^{h}\mathcal{S}^h, \qquad \text{ for all } h>0.$$
\end{remark}

\begin{definition}
For any Borel set $E\subseteq \mathbb{P}^n$ we write
\begin{equation}\label{eq:def_sigma_measure}
	\sigma_E(A)\coloneqq \int \mathcal{H}^{n-1}_{\mathrm{eu}}\bigl(A\cap \{z\in E:\pi_T(z)=t\}\bigr)\, dt,
\end{equation}
where $\mathcal{H}^{n-1}_{\mathrm{eu}}$ stands for the Euclidean $(n-1)$-Hausdorff measure in $\mathbb R^{n+1}$.
\end{definition}

\begin{remark}\label{rksuperfici}
The co-area formula \cite[Theorem D]{MR3624413} implies that, if $E$ is a Euclidean $n$-rectifiable set in $\mathbb{P}^n$, then $\sigma_E$ and $\mathcal{H}^{n+1}\llcorner E$ are mutually absolutely continuous.
\end{remark}

Finally, we denote by $\mathcal L^{n+1}$ the Lebesgue measure on $\mathbb R^{n+1}$ and by $\dim_{\mathrm{eu}}(E)$ the Euclidean Hausdorff dimension of $E\subseteq \mathbb R^{n+1}.$
\vv

\subsection{Planes, Grassmanians, and flat measures}
Anisotropic dilations on $\mathbb P^n$ allow us to define a parabolic Grassmanian, which can also be readily characterized.

\begin{definition}
For $h=0,\ldots, n+2$, we let $\Gr(h)$ be the family of \textit{homogeneous subgroups} of $\mathbb P^n$ with Hausdorff dimension $h$. Moreover, $\Gr(\mathbb{P}^n)\coloneqq \bigcup_{h=0}^{n+2}\Gr(h)$.
\end{definition}

\begin{lemma}\label{homplanesstructure}
	Let $V\in \Gr (\mathbb{P}^n)$. Either $V=V_1 \oplus e_{n+1}$ or $V=V_1\times \{0\}$ for some $V_1\in \Gr(\R^n)$.
\end{lemma}
\begin{proof}
	Let $w\in \R^{n+1}$ be a vector orthogonal to $V$. Then, since $V$ is homogeneous, we have
	\begin{equation}\label{eq:hom_lem_1}
	\lambda\, \langle w_H, v_H\rangle + \lambda^2\,  w_T v_T= 0 \qquad \text{ for all }\lambda>0, \, v\in V.
	\end{equation}
	
	We divide both sides of \eqref{eq:hom_lem_1} by $\lambda^2$, take the limit as $\lambda\to \infty$, and obtain that $ w_T\,  v_T=0$. Hence, if there exists $w\in \mathbb{P}^n$ orthogonal to $V$ with  $w_T\neq 0$ we have that $v_T=0$ for all $v\in V$, namely $V$ is of the form $V_1\times \{0\}$ for some $V_1\in \Gr(\R^n)$. Otherwise it holds that $e_{n+1}\in V$, which proves the lemma.
	\end{proof}

\begin{corollary}\label{corollary:grn1}
	Let $V\in \Gr(n+1)$. Then $V=V_1\oplus e_{n+1}$ for some hyperplane $V_1$ in $\R^n$.
\end{corollary}

\begin{proof}
It is easy to see that, given $V_1\in \Gr(\R^n),$ we have $\dim (V_1\times\{0\})=\dim_\eu (V_1\times\{0\})\leq n.$ Then, Lemma \ref{homplanesstructure} concludes the proof.
\end{proof}

\vv

\begin{definition}[Flat measures]\label{flatmeasures}
For any $h\in\{1,\ldots,Q\}$ we let $\mathfrak{M}(h)$ be the {\em family of flat $h$-dimensional measures} on $\mathbb{P}^n$, i.e.
\begin{equation}\label{eq:def_M_h}
    \mathfrak{M}(h)\coloneqq\bigl\{\lambda \mathcal S^h\llcorner V:\text{ for some }\lambda> 0 \text{ and }V\in\Gr(h)\bigr\}.
\end{equation}

We point out that we can replace $\mathcal S^h\llcorner V$ in \eqref{eq:def_M_h} with any other Haar measure on $V$ since they all coincide up to a multiplicative constant.
\end{definition}

\begin{definition}[Stratification vector]\label{def:stratification}
Let $h\in\{1,\ldots,n+2\}$. For any $ V\in\Gr(h)$ we define
\[
    \mathfrak{s}( V)\coloneqq\bigl(\dim_{\eu}\bigl((\R^n\times \{0\})\cap V\bigr), \dim_{\eu}( \R e_{n+1}\cap  V)\bigr),
\]
that with abuse of language we call the {\em stratification}, or {\em stratification vector}, of $ V$. Furthermore, we define
\[
    \mathfrak{S}(h)\coloneqq \bigl\{\mathfrak{s}( V)\in\N^2: V\in \Gr(h)\bigr\},
\]
for any $\mathcal{T}\subseteq \mathfrak{M}(h)$ we set
\[
    \mathfrak{s}(\mathcal{T})\coloneqq \bigl\{\mathfrak{s}( V):\text{there exists a non-null Haar measure of } V\text{ in }\mathcal{T}\bigr\},
\]
 and for any $\mathfrak{s}\in\mathfrak{S}(h)$ we write
\begin{equation}
    \begin{split}
        \Gr^\mathfrak{s}&(h)\coloneqq\bigl\{V\in\Gr(h):\mathfrak{s}(V)=\mathfrak{s}\bigr\},\\
        \mathfrak{M}^\mathfrak{s}(h)\coloneqq&\bigl\{\lambda\mathcal S^h\llcorner V:\text{ for some }\lambda> 0 \text{ and }V\in\Gr^\mathfrak{s}(h)\bigr\}.
        \nonumber
    \end{split}
\end{equation}
\end{definition}


\vv

\subsection{Measures with density and their blowups}\label{dns}
Let $\mathfrak d$ be a homogeneous metric on $\mathbb P^n.$

\begin{definition}[Lower and upper densities]\label{drens}
If $\phi$ is a Radon measure on $(\mathbb{P}^n, \mathfrak d)$ and $h>0$, we define
$$
    \Theta_*^{h}(\phi,\mathfrak d, x)\coloneqq\liminf_{r\to 0} \frac{\phi(B_{\mathfrak d}(x,r))}{r^{h}}\qquad \text{and}\qquad \Theta^{h,*}(\phi,\mathfrak d, x)\coloneqq\limsup_{r\to 0} \frac{\phi(B_{\mathfrak d}(x,r))}{r^{h}}.
$$
We say that $\Theta_*^{h}(\phi,\mathfrak d, x)$ and $\Theta^{h,*}(\phi,\mathfrak d, x)$ are the \textit{lower} and \textit{upper $h$-density} of $\phi$ at the point $x\in\mathbb{P}^n$, respectively. Furthermore, we understand that $\phi$ \textit{has $h$-density} if
$$
0<\Theta^h_*(\phi,\mathfrak d, x)=\Theta^{h,*}(\phi,\mathfrak d, x)<\infty,\qquad \text{for }\phi\text{-almost any }x\in\mathbb{P}^n.$$
If the metric is clear from the context, we use the shorter notations $\Theta^h_*(\phi, x)\coloneqq\Theta^{h}_*(\phi,\mathfrak d, x)$ and $\Theta^{h,*}(\phi, x)\coloneqq \Theta^{h,*}(\phi,\mathfrak d, x)$.
\end{definition}

\begin{definition}
Let $\{\mu_k\}$ be a sequence of measures in $\mathcal{M}$\label{emme}, the set of Radon measures on $\mathbb{P}^n$.  If\label{convdeb}
$$\lim_{k\to\infty}\int f(x)\, d\mu_k(x)=\int f(x)\, d\mu(x)\qquad\text{for any }f\in {C}_c(\R^n),$$\label{compi}
we say that $\{\mu_k\}$ converges to $\mu$ and write $\mu_k\rightharpoonup \mu$.
\end{definition}

\begin{definition}[Tangent measures]\label{def:TangentMeasure}
Let $\phi$ be a Radon measure on $\mathbb{P}^n$. For any $x\in\mathbb{P}^n$ and any $r>0$ we define the measure
$$
    T_{x,r}\phi(E)\coloneqq\phi\bigl(x+\delta_r(E)\bigr), \qquad\text{for any Borel set }E\subseteq\mathbb{P}^n.
$$
Furthermore, we define $\mathrm{Tan}_{h}(\phi,x)$, the set of $h$-dimensional tangents to $\phi$ at $x$, to be the collection of the Radon measures $\nu$ for which there is an infinitesimal sequence $\{r_i\}$ such that
$r_i^{-h}T_{x,r_i}\phi\rightharpoonup \nu$.

We also denote by $\mathrm{Tan}(\phi,x)$ the collection of the Radon measures $\tilde \nu$ on $\mathbb P^n$ for which there are a sequence $\{c_i\}\subset \mathbb R$ and an infinitesimal sequence $\{r_i\}$ such that
$c_iT_{x,r_i}\phi\rightharpoonup \tilde \nu$.
\end{definition}
\vv

A crucial property of locally asymptotically doubling measures is that Lebesgue differentiation theorem holds and, thus, local properties are stable under restriction to Borel subsets. In particular, the next result is a direct consequence of \cite[Theorem 3.4.3]{HeinonenKoskelaShanmugalingam}, the Lebesgue differentiation Theorem in \cite[p. 77]{HeinonenKoskelaShanmugalingam}, and  \cite[Proposition 2.15]{Mattila2005MeasuresGroups}.

\begin{proposition}\label{tuttitg}
    Let $\phi$ be a Radon measure on $(\mathbb{P}^n, \mathfrak d)$ such that
    \begin{equation}\label{eqn:DEFASYMPDOUBLING}
0<\Theta_*^h(\phi,x)\leq \Theta^{h,*}(\phi,x)<\infty \qquad\text{ for $\phi$-almost every $x\in\mathbb{P}^n$.}
\end{equation}
Then:
\begin{itemize}
\item[(i)] $\Tan_h(\phi,x)\neq \emptyset$ for $\phi$-almost every $x\in\mathbb{P}^n$.
    \item[(ii)]For any Borel set $B\subseteq \mathbb{P}^n$ the measure $\phi\llcorner B$ satisfies \eqref{eqn:DEFASYMPDOUBLING}. Moreover, for $\phi$-almost every $x\in B$ we have the equalities
$$\Theta^h_{*}(\phi\llcorner B,x)=\Theta^h_{*}(\phi,x)\qquad \text{and}\qquad\Theta^{h,*}(\phi\llcorner B,x)=\Theta^{h,*}(\phi,x).$$
\item[(iii)] For every non-negative $\rho\in L^1(\phi)$ and for $\phi$-almost every $x\in\mathbb{P}^n$ we have
\[
    \mathrm{Tan}_h(\rho\phi,x)=\rho(x)\mathrm{Tan}_h(\phi,x).
\]
More precisely, for $\phi$-almost every $x\in\mathbb{P}^n$ the following holds:
\begin{equation}
    \begin{split}
        \text{if $r_i\to 0$ is such that}\quad r_i^{-h}T_{x,r_i}\phi \rightharpoonup \nu\quad \text{then}\quad r_i^{-h}T_{x,r_i}(\rho\phi)\rightharpoonup \rho(x)\nu.  \\
    \end{split}
\end{equation}
\item[(iv)] For $\phi$-almost every $x\in\mathbb{P}^n$, if $\mu\in\Tan_h(\phi,x)$, for any $y\in\supp(\mu)$ and $r>0$ we have that $r^{-h}T_{y,r}\mu\in\Tan_h(\phi,x)$.
\end{itemize}
\end{proposition}

\vv

A particularly relevant class of measures with density is the one of uniform measures.
\begin{definition}\label{uniform}
	We say that a Radon measure $\mu$ on $(\mathbb P^n,\mathfrak d)$ is \textit{$h$-uniform} if $0\in\supp(\mu)$ and
$$\mu(B(x,r))= r^h\qquad\text{for any $r>0$ and  $x\in\supp(\mu)$.}$$
	We denote the set of $h$-uniform measures with the symbol $\mathcal{U}(h)$.
We also write $\mathcal{UC}(h)$ for the family of \textit{dilation-invariant} $h$-uniform measures, i.e.~ those $\mu\in \mathcal{U}(h)$ for which $\lambda^{-h}T_{0,\lambda}\mu=\mu$ for all $\lambda>0$.
\end{definition}

\vv

\begin{proposition}\label{propup}
	Assume that $\phi$ is a measure with $h$-density on $\mathbb{P}^n$. Then, for $\phi$-almost every $x\in\mathbb{P}^n$ we have
	$$\Tan_h(\phi,x)\subseteq \Theta^h(\phi,x)\mathcal{U}(h).$$
\end{proposition}

\begin{proof}
	The proof of this proposition follows almost without modifications the one given in the Euclidean case in \cite[Proposition 3.4]{DeLellis2008RectifiableMeasures}.
\end{proof}

\begin{proposition}\label{propspt1}
	Let $\phi$ be a Radon measure on $\mathbb P^n$ and let $\mu\in\Tan_h(\phi,x)$ be such that $r_i^{-h}T_{x,r_i}\phi\rightharpoonup \mu$ for some $r_i\to 0$. Then, for any $y\in\supp(\mu)$ there exists a sequence $\{z_i\}_{i\in\N}\subseteq\supp(\phi)$ such that $\delta_{1/r_i}(z_i-x)\to y$.
\end{proposition}

\begin{proof}
	A simple argument by contradiction yields the claim; the proof follows verbatim its Euclidean analogue (see \cite[Proposition 3.4]{DeLellis2008RectifiableMeasures}).
\end{proof}

Given a uniform measure we can also define its ``blowups at infinity", which are often referred to as \textit{blowdowns}.
\begin{definition}\label{def:blowdowns}
    Let $\mu\in\mathcal U(h)$. We say that a Radon measure $\nu$ is a \textit{tangent at infinity} of $\mu$ if there exists a sequence $\{R_i\}\to\infty$ such that
\[R_i^{-h}T_{0,R_i}\mu\rightharpoonup \nu.\]

We denote as $\Tan_h(\mu,\infty)$ the set of tangent measures at infinity of $\mu$.
\end{definition}

The following proposition is a strengthened version of Proposition \ref{propup} for uniform measures. 
\begin{proposition}\label{uniformup}
	Assume that $\mu$ is an $h$-uniform measure on $\mathbb P^n$. Then for any $z\in\supp(\mu)\cup\{\infty\}$  we have
	$$\emptyset\neq\Tan_h(\mu, z)\subseteq \mathcal{U}(h).$$
\end{proposition}

\begin{proof}
	A straightforward adaptation of the proof of \cite[Lemma 3.6]{DeLellis2008RectifiableMeasures} yields the sought conclusion.
\end{proof}

\begin{lemma}\label{replica}
	If $\{\mu_i\}_{i\in\N}$ is a sequence of $h$-uniform measures on $\mathbb P^n$ converging weakly to a measure $\nu$, then:
	\begin{itemize}
		\item[(i)] $\nu$ is an $h$-uniform measure.
		\item[(ii)] If $y\in\supp(\nu)$ then there exists a sequence $\{y_i\}_{i\in\N}\subseteq \mathbb{P}^n$ such that $y_i\in\supp(\mu_i)$ and $y_i\to y$.
		\item[(iii)] If there exists a sequence $\{y_i\}_{i\in\N}\subseteq \mathbb{P}^n$ such that $y_i\in\supp(\mu_i)$ and $y_i\to y$, then $y\in\supp(\nu)$.
	\end{itemize}
\end{lemma}

\begin{proof}
	The proof of this lemma is an almost immediate adaptation of the proofs of Propositions \ref{propup} and \ref{propspt1}.
\end{proof}

\begin{proposition}\label{UComp}
	For any $h\in[0,n+2]$ the set $\mathcal{U}(h)$ is compact with respect to the convergence of measures.
\end{proposition}

\begin{proof}
	Let $\{\mu_i\}_{i\in\N}$ be a sequence of measures in $\mathcal{U}(h)$. Then, by definition we have $\mu_i(B(x,r))\leq r^h$. Proposition 1.12 in \cite{Preiss1987GeometryDensities} implies the existence of a Radon measure $\mu$ such that $\mu_i\rightharpoonup\mu$. Proposition \ref{replica}-(i) concludes the proof.
\end{proof}

\vv

\subsection{Basic properties of uniform measures in \texorpdfstring{$\mathbb{P}^n$}{Lg}}

In this subsection we present some elementary properties of uniform measures on $(\mathbb P^n, \mathfrak d)$, where $\mathfrak d$ is a homogeneous metric.

\begin{proposition}\label{isometrie}
	Let $\Sigma\colon(\mathbb{P}^n,\lVert\cdot\rVert)\to(\mathbb{P}^n,\lVert\cdot\rVert)$ be a surjective isometry. If $\mu\in \mathcal{U}(h)$ and there exists $u\in\supp(\mu)$ such that $\Sigma(u)=0$, then $\Sigma_\#(\mu)\in\mathcal{U}(h)$ and
	$\supp(\Sigma_\#(\mu))=\Sigma(\supp(\mu))$.
\end{proposition}

\begin{proof}
	See \cite[Proposition 2.8]{MerloG1cod}.
\end{proof}

\begin{proposition}\label{supportoK}
	If $\mu$ is an $h$-uniform measure on $(\mathbb{P}^n,\mathfrak d)$, then $\mu=\mathcal{C}^{h}\llcorner{\supp(\mu)}$.
\end{proposition}

\begin{proof}
	The claim follows immediately from \cite[Theorem 3.1]{FSSCArea} and Definition \ref{uniform} (see also \cite[Proposition 2.9]{MerloG1cod}).
\end{proof}

\begin{proposition}\label{verticalsamoa}
Suppose that $\mu$ is an $h$-uniform measure on $(\mathbb P^n,\mathfrak d)$ for which there exists $V\in\mathrm{Gr}(h)$ such that $\supp(\mu)\subseteq V$. Then,
$\mu=\mathcal{C}^h\llcorner V$.
\end{proposition}

\begin{proof}
By Proposition \ref{supportoK}, for any $x\in\supp(\mu)$ and any $r>0$ we have
\begin{equation}\label{eq:cmsuppmu1}
    \mathcal{C}^h\llcorner {\supp(\mu)} (B(x,r))=\mu(B(x,r))=r^h.
\end{equation}

Moreover, \cite[Theorem 3.1]{FSSCArea} implies that $\mathcal{C}^h\llcorner V$ is the Haar measure of $V$ with unit density. So, for any $r>0$ we have
\begin{equation}\label{eq:formulaCmsuppmu}
	\mathcal{C}^h\llcorner \supp(\mu) \bigl(B(0,r)\bigr)=r^h=\mathcal{C}^h\llcorner V (B(0,r)).
\end{equation}

We are left with the proof that $\supp (\mu)= V$. We argue by contradiction and assume that $\supp(\mu)\neq V$. Hence, since $\supp(\mu)$ is closed in $V$, there exist $p\in V$ and $r_0>0$ such that $B(p,r_0)\cap \text{supp} (\mu)=\emptyset$. We also observe that $\mathcal{C}^h(B(p,r_0)\cap V)>0$ since $\mathcal{C}^h\llcorner V$ is a Haar measure of $V$. Thus, as $0\in \supp(\mu)$ because $\mu$ is $h$-uniform, it holds
\begin{equation}
    \begin{split}
      \mathcal{C}^h\llcorner V\bigl(B\bigl(0,2(\lVert p\rVert +r_0)\bigr)\bigr)
        \geq& \mathcal{C}^h\bigl(B\bigl(0,2(\lVert p\rVert +r_0)\bigr)\cap \text{supp}(\mu)\bigr)+\mathcal{C}^h(B(p,r_0)\cap V)\\
        >&	\mathcal{C}^h\llcorner \supp(\mu)\bigl(B(0,2(\lVert p\rVert +r_0))\bigr)\overset{\eqref{eq:cmsuppmu1}}{=}2^h(\lVert p\rVert +r_0)^h,
        \nonumber
    \end{split}
\end{equation}
which contradicts \eqref{eq:formulaCmsuppmu} and finishes the proof.
\end{proof}

\vv

\begin{definition}[Radially symmetric functions]
	We say that a function $\varphi\colon\mathbb{P}^n\rightarrow\R$ is \textit{radially symmetric} if there exists a profile function $g\colon [0,\infty)\rightarrow \R$ such that $\varphi(z)=g(\lVert z\rVert)$.
\end{definition}

Integrals of radially symmetric functions with respect to uniform measures can be computed according to the following formula.

\begin{proposition}\label{prop5}
	Let $\mu\in\mathcal{U}(h)$ and suppose that $\varphi\colon\mathbb{P}^n\rightarrow\R$ is a radially symmetric non-negative function. Then, for any $u\in\supp(\mu)$ we have
	\begin{equation}\label{eq:rad_int}
		\int  \varphi(z-u)\, d\mu(z)=h\int_0^\infty r^{h-1}g(r)\, dr,
	\end{equation}
	where $g$ is the profile function associated to $\varphi$.
\end{proposition}

\begin{proof}
	The proof is a standard argument. First one proves the formula for simple functions of the form
	$\varphi(z)\coloneqq\sum_{i=1}^k a_i\chi_{B(0, r_i)}$,
	where $a_i,r_i\geq 0$ for any $i=1,\ldots,k$. The result for a general function $\varphi$ as in the statement follows by Beppo Levi's convergence theorem. We omit further details.
\end{proof}

An immediate application of the previous proposition is the following result.

\begin{corollary}
	\label{prop1}
	For any $p\geq 0$, any $\mu\in\mathcal{U}(h)$, and any $u\in\supp(\mu)$, we have
	\begin{equation}
		\int  \lVert z-u\rVert^p e^{-s\lVert   z-u\rVert^4}\, d\mu(z)=\frac{h}{4s^\frac{h+p}{4}}\,\Gamma\Bigl(\frac{h+p}{4}\Bigr),
		\nonumber
	\end{equation}
where $\Gamma(t)\coloneqq \int_0^\infty s^{t-1}e^{-s}\, ds$ for $t>0$ denotes Euler's gamma function.
\end{corollary}

\begin{proof}
See for instance \cite[Corollary 2.12]{MerloG1cod}.
\end{proof}

\vv

\subsection{Parabolic rectifiability}

\begin{definition}[$\mathscr{P}_h$ and $\mathscr{P}_h^*$-rectifiable measures, see \cite{MerloG1cod}]\label{def:PhRectifiableMeasure}
Let $\mathfrak d$ be a homogeneous distance on $\mathbb P^n$ and $h\in\{1,\ldots,n+2\}$. A Radon measure $\phi$ on $(\mathbb{P}^n,\mathfrak d)$ is said to be $\mathscr{P}_h$-rectifiable if for $\phi$-almost every $x\in \mathbb{P}^n$ we have:
\begin{itemize}
    \item[(i)]$0<\Theta^h_*(\phi,x)\leq\Theta^{h,*}(\phi,x)<+\infty$.
    \item[(\hypertarget{due}{ii})]There exists $V(x)\in\Gr(h)$ such that $\mathrm{Tan}_h(\phi,x) \subseteq \{\lambda\mathcal{H}^h\llcorner V(x):\lambda\geq 0\}$.
\end{itemize}
Furthermore, we say that $\phi$ is $\mathscr{P}_h^*$-rectifiable if (\hyperlink{due}{ii})  is replaced with the weaker condition
\begin{itemize}
    \item[(ii)*] $\mathrm{Tan}_h(\phi,x) \subseteq \{\lambda\mathcal{H}^h\llcorner V:\lambda\geq 0\,\,\text{and}\,\,V\in \Gr(h)\}$.
\end{itemize}
\end{definition}

\begin{definition}\label{def:Rademacher_property}
Let $h\in \{1,\ldots,n+2\}$ and $V\in \Gr(h)$. We say that a Lipschitz function $f\colon V\to V^\perp$ has the \textit{Rademacher Property} on a Borel set $B\subseteq V$, or that it is an \textit{R-Lipschitz map} on $B$, if for $\mathcal{H}^h$-almost every $x\in B$ there exists a homogeneous group homomorphism $L_x\colon V\to V^\perp$ such that
\begin{equation}
    \lim_{V \ni v\to 0}\frac{\bigl\lVert f(x+v)-f(x)-L_x[v]\bigr\rVert}{\lVert v\rVert}=0.
    \label{RProperty}
\end{equation}
The homogeneous homomorphism $L_x$ is said to be the \emph{differential} of $f$ at $x$ and it is denoted by $Df(x)$.
\end{definition}

\vv

\begin{definition}[see \cite{MatParRect}]\label{def:LGrect}
	A set $E\subset \mathbb{P}^n$ is said to be an $h$-dimensional $LG$-\textit{rectifiable} set, if there are countably many $W_i\in \Gr(h)$, compact sets $K_i\subseteq W_i$, and  R-Lipschitz maps $g_i\colon K_i\to W_i^{\perp}$ 
	such that 
	\[
	    \mathcal H^h\Bigl(E\setminus \bigcup_{j=1}^\infty \gr(g_i)\Bigr)=0,
	\]
	where $\gr(g_i)\coloneqq \{(y,g_i(y))\in \mathbb{P}^n:y\in K_i\}.$
\end{definition}

\vv
Mattila obtained a parabolic analogue of some of the classical characterizations of rectifiable sets. For the definitions of the properties involved in the next statement, we refer to \cite{MatParRect}.

\begin{theorem}[Mattila, \cite{MatParRect}, Theorem 1.1]\label{theorem_Mattila_parabolic_rect}
    Let $E\subset \mathbb P^n$ be $\mathcal H^h$-measurable and such that $\mathcal H^h(E)<\infty$. The following are equivalent:
    \begin{enumerate}
        \item   $E$ is LG-rectifiable.
        \item   $E$ has approximate tangent planes at $\mathcal H^h$-almost every point.
        \item   For $\mathcal H^h$-almost all $x\in E$ there is an $h$-flat measure $\lambda_x$ such that 
            \[
                \Tan (\mathcal H^h\llcorner E, x)=\{c\lambda_x: 0<c<\infty\}.
            \]
        \item   For $\mathcal H^h$-almost all $x\in E$, $\mathcal H^h\llcorner E$ has a unique tangent measure at $x$.
    \end{enumerate}
\end{theorem}

\vv

In the language of Definition \ref{def:PhRectifiableMeasure}, Theorem \ref{theorem_Mattila_parabolic_rect} implies the following.

\begin{theorem}\label{theorem:Mattila_parabolic_rectifiability}
	Let $E\subset \mathbb{P}^n$ be an $\mathcal H^h$-measurable set such that $\mathcal H^h(E)<\infty.$ The following are equivalent:
	\begin{enumerate}
		\item $E$ is LG-rectifiable.
		\item $\mathcal{H}^h\llcorner E$ is $\mathscr{P}_h$-rectifiable.
	\end{enumerate}
\end{theorem}

\begin{theorem}\label{mm0.0}Let $h\in\{1,\ldots, n+1\}$, $\mathfrak{s}\in \mathfrak{S}(h)$, and $V\in \Gr^\mathfrak{s}(h)$. Let $K$ be a compact subset of $V$, and suppose that $f\colon K\to V^\perp$ is a Lipschitz map such that
    it holds
    $$\Tan_h(\phi,x)\subseteq \mathfrak{M}^\mathfrak{s}(h)\qquad \text{for }\mathcal{H}^h\llcorner \mathrm{gr}(f)\text{-almost every }x\in \mathbb{P}^n.$$
    Then for $\mathcal{H}^h\llcorner \mathrm{gr}(f)$-almost every $x\in \mathbb{P}^n$ there exists $V(x)\in\Gr^\mathfrak{s}(h)$ such that
    \begin{equation}
         \Tan_h(\phi,x)\subseteq\{\lambda \mathcal{H}^{h}\llcorner V(x):\lambda\geq 0\} \qquad\text{for }\mathcal{H}^h\llcorner \mathrm{gr}(f)\text{-almost every }x\in \mathbb{P}^n.
         \label{eq:uniqtg}
    \end{equation}
    In particular $\mathrm{gr}(f)$ can be covered with countably many images of Lipschitz functions with the Rademacher property.
\end{theorem}

\begin{proof}
Note that $\Gr^{(0,1)}(2)$ and $\Gr^{(n,0)}(n)$ are singletons, thus if $\mathfrak{s}$ is $(0,1)$ or $(n,0)$ there is nothing to prove.

Let us consider the case $\mathfrak{s}=(h,0)$ and suppose that $\lambda\mathcal{H}^h\llcorner W\in \Tan_h(\phi,x)$ for some $W\in \Gr^\mathfrak{s}(h)$. Orthogonal transformations of $\R^{n+1}$ that preserve the modulus  of the last coordinate are isometries, so without loss of generality we assume that $V=\mathrm{span}(e_1,\ldots,e_k)$. Thus, there exists an infinitesimal sequence $r_i$ such that 
$$r_i^{-h}T_{x,r_i}\phi\rightharpoonup\lambda\mathcal{H}^h\llcorner W.$$

Thanks to Proposition \ref{propspt1} we know that given $w\in W$ there exists a sequence $\{z_i\}\subset \mathrm{gr}(f)$ such that $\delta_{1/r_i}(z_i-x)\to w$. All the points are chosen on $\mathrm{gr}(f)$, so we can find $y_i,y\in K$ such that $f(y_i)=z_i$, $f(y)=x$, and
\begin{equation}
    \lim_{i\to \infty}\Big(\frac{y_i-y}{r_i},\frac{f_H(y_i)-f_H(y)}{r_i},\frac{f_T(y_i)-f_T(y)}{r_i^2}\Big)=w=(w_{1,H},w_{2,H},0),
    \label{eq:ididid}
\end{equation}
where $w_{1,H}\in\R^k,\text{ }w_{2,H}\in\R^{n-k}$, and $f=(f_H,f_T)$ for $f_H\colon V\to \R^{n-k}$ and $f_T\colon V\to \R$.
Since $f_H$ is \emph{Euclidean} Lipschitz, by Rademacher's theorem we can assume that $y$ is such that there exists a linear function $L(y)[\cdot]$ for which
$$f_H\bigl(y+r_iw_{1,H}+o(r_i)\bigr)=f_H(y)+L(y)\bigl[r_iw_{2,H}+o(r_i)\bigr].$$

The identity above together with \eqref{eq:ididid} also implies that 
\begin{equation}
    w_{2,H}=L(y)[w_{1,H}]
    \label{eq:ididid2}
\end{equation}
and, since \eqref{eq:ididid2} is satisfied for any $w\in W$, this means that every $W$ chosen as above must coincide at $f(y)$ with the graph of $L(y)$. The parabolic differentiability of $f$ is an immediate consequence of \eqref{eq:ididid} and \eqref{eq:ididid2}.

Now let $h\geq 3$, consider the case $\mathfrak{s}=(h-2,1)$, and assume that $\lambda\mathcal{H}^h\llcorner W\in \Tan_h(\phi,x)$. As above, we suppose without loss of generality that $V=\mathrm{span}(e_1,\ldots,e_{h-2})\oplus \R e_{n+1}$ and we choose an infinitesimal sequence $r_i$ such that 
$$r_i^{-h}T_{x,r_i}\phi\rightharpoonup\lambda\mathcal{H}^h\llcorner W.$$
Thanks to Proposition \ref{propspt1} we know that for $w\in W$ there exists a sequence $\{z_i\}\subset \mathrm{gr}(f)$ such that $\delta_{1/r_i}(z_i-x)\to w$. In addition, since all the points are chosen on $\mathrm{gr}(f)$, we can find $y_i,y\in K$ such that $f(y_i)=z_i$, $f(y)=x$, and
\begin{equation}
    \lim_{i\to \infty}\Big(\frac{(y_i)_H-y_H}{r_i},\frac{f(y_i)-f(y)}{r_i},\frac{(y_i)_T-y_T}{r_i^2}\Big)=w=(w_{1,H},w_{2,H},w_T),
    \label{eq:ididid3}
\end{equation}
where $w_{1,H}\in\R^k,\text{ }w_{2,H}\in\R^{n-k}$, and $\text{ }w_T\in\R$. The choice of $w$ was arbitrary, thus we assume $w_T=0$. We further observe that $f(\cdot,y_T)$ is Euclidean Lipschitz, hence we can assume that $y$ is such that there exists a linear function $L(y)\colon\R^{h-2}\to \R^{n-h+2}$ which satisfies
\begin{equation}\label{eq:different_lemma222}
f\bigl(y_H+r_iw_{1,H}+o(r_i),y_T\bigr)=f(y)+r_iL(y)[w_{1,H}]+o(r_i).
\end{equation}

Therefore, from \eqref{eq:ididid3} and \eqref{eq:different_lemma222} we conclude that 
\begin{equation}\label{eq:different_lemma1}
    \begin{split}
        &\qquad\qquad\qquad r_iw_{2,H}+o(r_i)=f\bigl(y_H+r_iw_{1,H}+o(r_i),y_T+o(r_i^2)\bigr)-f(y)\\
        =&f\bigl(y_H+r_iw_{1,H}+o(r_i),y_T+o(r_i^2)\bigr)-f\bigl(y_H+r_iw_{1,H}+o(r_i),y_T\bigr)+r_iL(y)[w_{1,H}]+o(r_i)\\
        &\qquad\qquad\qquad\qquad\qquad\qquad=r_iL(y)[w_{1,H}]+o(r_i),
    \end{split}
\end{equation}
which yields $w_{2,H}=L(y)[w_{1,H}]$. As in the previous case, \eqref{eq:different_lemma1} implies that any $W$ chosen as above coincides with $\mathrm{gr}(L(y))\oplus \R e_{n+1}$, which depends only on the point $x=f(y)$. This concludes the proof of \eqref{eq:uniqtg}. The final part of the proposition follows from Theorem \ref{theorem:Mattila_parabolic_rectifiability}.
\end{proof}

\vv

\section{Marstrand's theorem and the structure of uniform measures}\label{section:Mastrand}

Throughout this section we understand that $\mathbb P^n$ is endowed with the Koranyi metric $d$.

\begin{theorem}[{\cite[Theorem 1.2]{Chousionis2015MarstrandsGroup}}]\label{th.marstrand}
Let $h\in[0,n+2]$ and suppose $\mu\in \mathcal{U}(h)$. Then $\supp(\mu)$ is an analytic variety whose Hausdorff dimension is integer and coincides with $h$.
\end{theorem}

\begin{proof} Since $\mathbb{P}^n$ is isometrically embedded in $\mathbb{H}^n$ equipped with the Koranyi norm, any uniform measure supported on the isometrically embedded copy of $\mathbb{P}^n$ in $\mathbb{H}^n$ is a uniform measure in $\mathbb{H}^n$. This, together with Proposition 3.1 and Theorem 1.1 in \cite{Chousionis2015MarstrandsGroup},
concludes the proof.
\end{proof}

\begin{proposition}\label{strutturaunifmeasvssurfacemeasure}
Assume that $\mu$ is an $(n+1)$-uniform measure on $(\mathbb{P}^n,d)$. Then there exists a constant $c>0$ depending only on $n$ such that $\mu(A)=c\sigma_{\supp (\mu)}(A)$ for every Borel subset $A\subset \mathbb{P}^n$.
\end{proposition}

\begin{proof} 
By Theorem \ref{th.marstrand} we have that $\supp(\mu)$ is an analytic manifold of Hausdorff dimension $n+1$, so by a result of Magnani as stated in \cite[Proposition 2.2]{Chousionis2015MarstrandsGroup} (see also the references therein, in particular \cite{MagnaniJGA2010} and \cite{Magnani_Vittone_2008}) and Lojasiewicz’s Structure Theorem for real analytic varieties (see for instance \cite[Theorem 5.2.3]{KPanalyticmanifold}) we have that $\supp(\mu)$ is the countable union of $C^{\infty}$-manifolds of Euclidean Hausdorff dimensions $n$ or $n+1$. However, this shows that $\supp(\mu)$ is  actually a countable union of (Euclidean) $n$-dimensional $C^{\infty}$-manifolds $\Sigma_i$. Let $G$
be the maximal open subset of $\R^{n+1}$ where $G\cap \supp(\mu)$ is an analytic submanifold of $\R^{n+1}$. Note that the above discussion together with \cite[\S 3.4.8]{Federer1996GeometricTheory} implies that $\mathcal{H}^{n}_{\mathrm{eu}}(\R^n\setminus G)=0$ and thus $\mu(\R^n\setminus G)=0$. See also the proof of \cite[Theorem 1.4]{KirchheimPreiss02}.

Thanks to \cite[Proposition 2.2-(ii)]{Chousionis2015MarstrandsGroup} and since the surface measure $\sigma_{\supp(\mu)}$ is mutually absolutely continuous with respect to $\mathcal{H}^{n+1}\llcorner \supp(\mu)$, we infer that
$$\limsup_{r\to 0}\frac{(\mu+\sigma_{\supp(\mu)})(B(x,2r))}{(\mu+\sigma_{\supp(\mu)})(B(x,r))}<\infty, \qquad \text {for $(\mu+\sigma_{\supp(\mu)})$-a.e. $x\in\mathbb{P}^n$.}$$
Thus the family $\mathcal F\coloneqq \{(x,B(x,r)):x\in G, 0<r<1\}$ is a $(\mu + \sigma_{\supp(\mu)})$-Vitali relation in the sense of \cite[2.8.16]{Federer1996GeometricTheory}, so the differentiation theorem \cite[2.9.5]{Federer1996GeometricTheory} implies that the function
	\begin{equation}\label{eq:deriv_differ_Marstrand}
	    f(x)\coloneqq \lim_{r\to 0}\frac{\mu(B(x,r))}{\bigl(\mu + \sigma_{\supp(\mu)}\bigr)(B(x,r))}
    \end{equation}
    is defined for $(\mu + \sigma_{\supp(\mu)})$-almost every $x\in \mathbb{P}^n.$ Furthermore, by \cite[2.9.7]{Federer1996GeometricTheory} it holds that
    \begin{equation}\label{eq:int_Mastrand}
        \mu(A)= \int_A f(z)\, d\bigl(\mu + \sigma_{\supp(\mu)} \bigr)(z)\qquad \text{ for all Borel set }A\subseteq G.
    \end{equation}
    
    An elementary computation shows that, if $\Sigma$ is a smooth manifold in $\mathbb{P}^n$, there exists a constant $\mathfrak{c}_n>0$ depending only on $n$ such that 
    \begin{equation}\label{eq:exp_Marstrand}
      \sigma_{\Sigma}(B(x,r))= \mathfrak{c}_nr^{n+1} + o(r^{n+1}),
    \end{equation}
    for $\sigma_{\Sigma}$-almost every $x\in\Sigma$. However, since $\sigma_{\supp(\mu)}$ is locally asymptotically doubling, Lebesgue differentiation theorem and \eqref{eq:exp_Marstrand} yield that, for $\sigma_{\supp(\mu)}$-almost every $x\in \Sigma_i$, we have
    $$\lim_{r\to 0}\frac{\sigma_{\supp(\mu)}(B(x,r))}{r^{n+1}}=\lim_{r\to 0}\frac{\sigma_{\supp(\mu)}(B(x,r))}{\sigma_{\supp(\mu)}(B(x,r)\cap \Sigma_i)}\frac{\sigma_{\supp(\mu)}(B(x,r)\cap \Sigma_i)}{r^{n+1}}=\mathfrak{c}_n,$$
    hence $\Theta^{n+1}(\sigma_{\supp(\mu)},x)=\mathfrak{c}_n$ for $\sigma_{\supp(\mu)}$-almost every $x\in\mathbb{P}^n$.
    In particular, \eqref{eq:deriv_differ_Marstrand} and \eqref{eq:exp_Marstrand} imply that $f(x)= (1+\mathfrak{c}_n)^{-1}$ for $\bigl(\mu + \sigma_{\supp(\mu)}\bigr)$-almost every $x\in\mathbb{P}^n$. Finally \eqref{eq:int_Mastrand} gives
    $\mu(A)= \mathfrak{c}_n^{-1}\sigma_{\supp(\mu)}(A)$ for all Borel sets $A\subset \mathbb P^n$, which concludes the proof.
\end{proof}

\vv
\begin{remark}\label{rem:remark_c_P_1}
    In the case $n=1,$ it is easy to see that the constant $\mathfrak c_1$ in \eqref{eq:exp_Marstrand} equals $2$. In particular we have that $c=1/2$ in the statement of Proposition \ref{strutturaunifmeasvssurfacemeasure}.
\end{remark}

\vv
We recall that $\mathcal C^{n+1}$ stands for the $(n+1)$-dimensional centered Hausdorff measure (see Definition \ref{Hausdro}).
\begin{proposition}\label{tgverticalseregolare}
Suppose that $\Gamma\subset \mathbb P^n$ is a (Euclidean) $n$-dimensional manifold of class $C^\infty$. Denote by $\mathfrak{n}(x)\in \R^{n+1}$ the smooth vector field orthogonal to $\Gamma$ and  $\hat {\mathfrak{n}}(x)\coloneqq\pi_H(\mathfrak{n}(x))$. Then, for  $\mathcal{C}^{n+1}\llcorner \Gamma$-almost every $x\in\mathbb{P}^n$ we have $\Tan_{n+1}(\mathcal{C}^{n+1}\llcorner \Gamma,x)=\{\mathcal{C}^{n+1}\llcorner \hat{\mathfrak{n}}(x)^\perp\}$.
\end{proposition}

\begin{proof} The existence of the $(n+1)$-dimensional density of $\mathcal{C}^{n+1}\llcorner \Gamma$ follows directly from \cite[Proposition 2.2]{Chousionis2015MarstrandsGroup} together with the smoothness of the natural embedding of $\mathbb{P}^n$ into $\mathbb{H}^n$. This observation and the area formula \cite[Theorem 3.1]{FSSCArea} further imply that 
\begin{equation}\label{eq:prop23aux1}
    \Theta^{n+1}(\mathcal{C}^{n+1}\llcorner \Gamma,x)=1\qquad \text{ for } \mathcal{C}^{n+1}\llcorner \Gamma\text{-a.e. } x\in\mathbb{P}^n.
\end{equation}

We define $\mathfrak{N}\coloneqq\{x\in\Gamma: \mathfrak{n}(x)=e_{n+1}\}$ and claim that $\mathcal{H}^{n+1}(\mathfrak{N})=0$.
In order to prove this, we first note that the continuity of $\mathfrak{n}(\cdot)$ implies that $\mathfrak{N}$ is a closed set in $\Gamma$.
Thanks to the smoothness of $\Gamma$ we can find a countable cover of $\mathfrak{N}$ with balls $\{B(x_i,r_i)\}_{i\in\N}$ so that for any $i\in\N$ there exists a $C^\infty$ function $f_i\colon\R^n\to \R$ such that
$$\Gamma\cap B(x_i,r_i)=\bigl\{(y,f_i(y)):y\in\R^n\bigr\}\cap B(x_i,r_i).$$

Then, the set $\mathfrak{N}\cap B(x_i,r_i)$ coincides with the intersection of the graph of $f_i$ on its critical set with the ball $B(x_i,r_i)$. However, by Sard's theorem (see for instance \cite[Theorem 3.4.3]{Federer1996GeometricTheory}) we have that $\mathcal{L}^1(\{w:df_i(w)=0\})=0$. Therefore the mutual absolute continuity of $\mathcal{H}^{n+1}$ and $\sigma_{\Gamma}$ implies that $\mathcal{H}^{n+1}(B(x_i,r_i)\cap \mathfrak{N})=0$ which, together with the choice of the cover $\{B(x_i,r_i)\}$, concludes the proof of the claim $\mathcal{H}^{n+1}(\mathfrak{N})=0$.

Now let $x\in\Gamma$ be a point for which $\pi_1(\mathfrak{n}(x))\neq 0$ and $\Tan_{n+1}(\mathcal{C}^{n+1}\llcorner \Gamma,x)\neq \emptyset$.
Fix $\mu\in \Tan_{n+1}(\mathcal{C}^{n+1}\llcorner \Gamma,x)$ and let $r_i$ be an infinitesimal sequence such that $$r_i^{-(n+1)}T_{x,r_i}\mathcal{C}^{n+1}\rightharpoonup \mu.$$ 

Proposition \ref{propspt1} yields that for any $y\in\supp(\mu)$ there exists a sequence $\{z_i\}\subset\Gamma$ such that $\delta_{r_i^{-1}}(z_i-x)\to y$. This implies in particular that 
$\Delta_i\coloneqq z_i-x-\delta_{r_i}(y)$ satisfies $\lim_{i\to \infty}\lVert\Delta_i\rVert/r_i= 0$. However, since $\Gamma$ is a smooth Euclidean $n$-dimensional manifold, then $\langle \mathfrak{n}(x),(z_i-x)/r_i\rangle\to 0$ as $i\to \infty$. This readily gives that
\begin{equation}
\nonumber
    0=\lim_{i\to \infty}\bigl\langle \pi_H(\mathfrak{n}(x)),y_H+r_i^{-1}\pi_H(\Delta_i)\bigr\rangle+\pi_T(\mathfrak{n}(x))\bigl(r_iy_T+r_i^{-1}\pi_T(\Delta_i)\bigr)=\bigl\langle \pi_H(\mathfrak{n}(x)),y_H\bigr\rangle,
\end{equation}
so we obtain $\supp(\mu)\subseteq \hat{\mathfrak{n}}(x)^\perp$. 
However, by \eqref{eq:prop23aux1} and Proposition \ref{propup} we have $\Tan_{n+1}(\mathcal{C}^{n+1}\llcorner \Gamma,x)\subseteq \mathcal{U}(n+1)$ for $\mathcal{C}^{n+1}\llcorner \Gamma$-almost every $x\in \mathbb{P}^n$. This together with Propositions \ref{supportoK} and \ref{verticalsamoa} concludes the proof.
\end{proof}

\begin{corollary}\label{existenceflatmeasure}
    Suppose that $\phi$ is a Radon measure on $\mathbb{P}^n$ with $(n+1)$-density. Then for $\phi$-almost every $x\in\mathbb{P}^n$ we have
    \[\Tan_{n+1}(\phi,x)\cap \mathfrak{M}(n+1)\neq \emptyset.\]
\end{corollary}

\begin{proof}
    This is an immediate consequence of Theorem \ref{th.marstrand} and Propositions \ref{tuttitg}, \ref{propup}, and \ref{tgverticalseregolare}.
\end{proof}

\vv

\subsection{Blowups of \texorpdfstring{$1$}{1}-codimensional uniform measures}

\begin{proposition}\label{bupunifarecones}
Let $\mu\in \mathcal{U}(n+1)$. For any $x\in\supp(\mu)$ and any $\nu\in\Tan_{n+1}(\mu,x)$ we have that $\nu$ is dilation-invariant, i.e. $\lambda^{-(n+1)}T_{0,\lambda}\nu=\nu$ for any $\lambda>0$.
\end{proposition}

\begin{proof}
Thanks to Theorem \ref{th.marstrand} we know that there exists a real analytic function $H\colon \R^{n+1}\to \R$ such that 
\begin{equation}\label{eq:Hsuppmu}
\bigl\{x\in\R^{n+1}:H(x)=0\bigr\}=\supp(\mu).
\end{equation}

Since $H(0)=0$, for some $k>0$ we can write
$$H(x)=\sum_{j=k}^\infty\sum_{\lvert\alpha_H\rvert+2\alpha_T=j} \frac{a_\alpha}{\alpha!}x^\alpha,$$
where $\alpha=(\alpha_H,\alpha_T)\in \mathbb N^n\times \mathbb N$ and we have $a_{\bar \alpha}\neq 0$ for some $\bar\alpha$ with $\lvert\bar\alpha_H\rvert+2\bar\alpha_T=k$. 

Let us fix $\nu\in\Tan_{n+1}(\mu,0)$ and suppose that $r_i$ is an infinitesimal sequence for which $r_i^{-(n+1)}T_{0,r_i}\mu\rightharpoonup\nu$. The first step of the proof is to show that for $P_k(x)\coloneqq\sum_{\lvert\alpha_H\rvert+2\alpha_T=k} \frac{a_\alpha}{\alpha!}x^\alpha$ we have that \begin{equation}\label{eq:suppnusubP0}
    \supp(\nu)\subseteq \{P_k=0\}.
\end{equation}

To prove \eqref{eq:suppnusubP0}, let $u\in \supp(\mu)$ and observe that Lemma \ref{replica}-(ii)-(iii) implies that there exists a sequence $\{z_i\}_{i\in\N}$ in $\supp(\mu)$ such that $\delta_{r_i^{-1}}(z_i)\to u$. Thus it holds
\begin{equation}\label{eq:identity_0_H_r}
    \begin{split}
        0\overset{\eqref{eq:Hsuppmu}}{=}&\frac{H(z_i)}{r_i^k}=r_i^{-k}\sum_{j=k}^\infty\sum_{\lvert\alpha_H\rvert+2\alpha_T=j} \frac{a_\alpha}{\alpha!}z_i^\alpha\\
        =&\sum_{\lvert\alpha_H\rvert+2\alpha_T=k} \frac{a_\alpha}{\alpha!}\delta_{r_i^{-1}}(z_i)^\alpha+\sum_{j=k+1}^\infty r_i^{j-k}\Big(\sum_{\lvert\alpha_H\rvert+2\alpha_T=j} \frac{a_\alpha}{\alpha!}\delta_{r_i}^{-1}(z_i)^\alpha\Big)\eqqcolon \mathcal{I}_i+ \mathcal{II}_i.
    \end{split}
\end{equation}

It is easily seen that $H\bigl(\delta_{r_i}^{-1}(z_i)\bigr)$ is uniformly bounded for any $i\in\N$, so  $$\sum_{j=k+1}^\infty \Bigl(\sum_{\lvert\alpha_H\rvert+2\alpha_T=j} \frac{a_\alpha}{\alpha!}\delta_{r_i}^{-1}(z_i)^\alpha\Bigr)$$ is uniformly bounded as well.
Therefore, $\mathcal{II}_i\to 0$ as $i\to \infty$, and the first equality in \eqref{eq:identity_0_H_r} readily gives that $\lim_{i\to \infty}\mathcal I_i=0$ as well.
This implies in particular that
$$P_k(u)=\lim_{i\to \infty}P_k\bigl(\delta_{r_i^{-1}}(z_i)\bigr)=\lim_{i\to \infty}\sum_{\lvert\alpha_H\rvert+2\alpha_T=k} \frac{a_\alpha}{\alpha!}\delta_{r_i^{-1}}(z_i)^\alpha=0,$$
which proves \eqref{eq:suppnusubP0}.

By hypothesis the polynomial $P_k$ is non-trivial, so the zero-level set $\Gamma\coloneqq\{P_k=0\}$ is contained in an $n$-dimensional analytic manifold. Furthermore since $\Gamma$ supports the measure $\nu$, which is mutually absolutely continuous with respect to $\mathcal{H}^{n}_{\eu}\llcorner \supp(\nu)$, we have that $\Gamma$ \emph{is} an $n$-dimensional analytic manifold and in particular a Euclidean $n$-rectifiable set. 

It is immediate to see that $\Gamma$ is invariant under parabolic dilations. Therefore, the set
$$\Sigma(P_k)\coloneqq\{x\in\Gamma:\text{the Euclidean tangent to } \Gamma \text{ at }x \text{ does not exist or is horizontal}\},$$
must be a cone as well. 

The second step of the proof consists in showing that if $y\in \Gamma\setminus \supp(\nu)$ and $z\in\supp(\nu)$ is such that $\lvert y-z\rvert=\mathrm{dist}(y,\supp(\nu))$ then $z\in \Sigma(P_k)$. Let us argue by contradiction and assume that there are $y$ and $z$ as above such that $z\not\in\Sigma(P_k)$. By Lojaciewicz's Structure Theorem (see \cite[Theorem 5.2.3]{KPanalyticmanifold}) these points must be contained in a set of Euclidean Hausdorff dimension $n-1$ and thus of $\mathcal{H}^{n+1}$-null measure. In addition, by Proposition \ref{tgverticalseregolare}, we infer that the blowup of $\Gamma$ at $z$ must be an element $\mathfrak{n}(z)^\perp$ of $\Gr(n+1)$. Hence, Proposition \ref{verticalsamoa} implies that $\Tan_{n+1}(\nu,z)=\{\mathcal{C}^{n+1}\llcorner \mathfrak{n}(z)^\perp \}$. 

If $z_H=y_H$, the blowup of the ball $B(y,\lvert z-y\rvert)$ at the point $z$ is $\{w\in\mathbb{P}^n:w_T\geq 0\}$ and this contradicts the fact that $\Tan_{n+1}(\nu,z)=\{\mathcal{C}^{n+1}\llcorner \mathfrak{n}(z)^\perp \}$.
If on the other hand $z_H\neq y_H$, the parabolic blowup of the ball $B(y,\lvert z-y\rvert)$ at the point $z$ is easily seen to be $\{w\in\mathbb{P}^{n}:\langle w_H,z_H-y_H\rangle\leq 0\}$ and this forces $\mathfrak{n}(z)$ to be parallel to $z_H-y_H$, otherwise as above we would have that the support of any tangent of $\nu$ at $z$ should be contained in a strict subset of $\mathfrak{n}(z)^\perp$, which is excluded thanks to Proposition \ref{verticalsamoa} and Lemma \ref{replica}. 

The third and final step of the proof is to show that if $C$ is a connected component of $\Gamma\setminus \Sigma(P_k)$ then either $C\subseteq \supp(\nu)$ or $C\cap \supp(\nu)=\emptyset$. Let us first note that since $C$ is connected and $C\cap \supp(\nu)$ is relatively closed in $\supp(\nu)$ if we prove that $\supp(\nu)$ is also relatively open in $C$, then the claim is proved. To show this, let us assume by contradiction that this is not the case and that we can find $z\in \supp(\nu)$ for which there exists a sequence of points $y_i$ in $\Gamma\setminus \Sigma(P_k)$ converging to $z$. The previous step shows that for any $i\in\N$ we have that $z_H-(y_i)_H$ is parallel to $\mathfrak{n}(z)$. However, since the blowup of $\Gamma$ is the plane $\mathfrak{n}(z)^\perp$, we can assume without loss of generality that $\bigl(z_H-(y_i)_H\bigr)/\lvert z_H-(y_i)_H\rvert\to u\in \mathfrak{n}(z)^\perp$ and we conclude that
$$1=\Big\lvert\Big\langle\mathfrak{n}(z),\lim_{i\to \infty}\frac{z_H-(y_i)_H}{\lvert z_H-(y_i)_H\rvert}\Big\rangle\Big\rvert=\lvert\langle \mathfrak{n}(z),u\rangle\rvert=0,$$
which is a contradiction. 

Since $\Sigma(P_k)$ is $\nu$-null, $\supp(\nu)$ must coincide with the closure of the union of some of the connected components of $\Gamma\setminus \Sigma(P)$. This implies thanks to Theorem \ref{th.marstrand} that $\nu$ is dilation-invariant.
\end{proof}

\vv

\begin{corollary}\label{cor:flatae}
Let $\mu\in\mathcal{U}(n+1)$. For $\mu$-almost every $x\in \mathbb{P}^n$ there exists $V_x\in\Gr(n+1)$ such that
$$\Tan_{n+1}(\mu,x)\subseteq \{\lambda\mathcal{C}^{n+1}\llcorner V_x:\lambda>0\}.$$
\end{corollary}

\begin{proof}
Propositions \ref{tuttitg} and \ref{bupunifarecones} imply that there exists $\nu\in\mathcal U(n+1)$ such that
$$\Tan_{n+1}(\mu,x)\subseteq\{\lambda\nu:\lambda>0\}.$$
By \cite{Mattila2005MeasuresGroups} this further implies that $\nu$ is the Haar measure of a closed subgroup of $\mathbb P^n$ of Hausdorff dimension $n+1$, which is also dilation-invariant by Proposition \ref{bupunifarecones}. This concludes the proof.
\end{proof}

\vv
We conclude the section with the classification of $1$-codimensional uniform measures on $\mathbb P^1.$ This is used in the proof of Theorem \ref{theorem:main_theorem_Preiss} in $\mathbb P^1$ (see Proposition \ref{preiss1}).

\begin{proposition}\label{uniformmeasuresinP1}
In $\mathbb{P}^1$, we have that
$\mathcal{U}(2)=\{\mathcal{C}^2\llcorner \mathcal{V}\}$, where we recall that $\mathcal{V}\coloneqq\{(0,s):s\in\R\}$.
\end{proposition}

\begin{proof}
Let $\mu\in \mathcal{U}(2)$. By Proposition \ref{strutturaunifmeasvssurfacemeasure} and Remark \ref{rem:remark_c_P_1} we have that  $\mu=2^{-1}\,\sigma_{\supp(\mu)}$. Since $\supp(\mu)$ is an analytic variety by Theorem \ref{th.marstrand}, it can be written as the union of countably many smooth manifolds of (topological) dimension $1$. However, Propositions \ref{tgverticalseregolare} and \ref{cor:flatae} guarantee that at $\mathcal{C}^2$-almost every point the Euclidean normal to $\supp(\mu)$ is not vertical. This means in particular that for $\mathcal{C}^2\llcorner \supp(\mu)$-almost every $x\in\supp (\mu)$ there exist $r=r(x)>0$ and a smooth map $f\colon \mathcal{V}\to \R$ such that $$\mathrm{gr}(f)\cap B(x,r)=\supp(\mu)\cap B(x,r).$$

 {Hence, for any $y\in B(x,r/2)\cap \supp(\mu)$ and any $0<s<r/2$ we have 
\begin{equation}
    \begin{split}
       \,\mathcal{L}^1\bigl(\pi_T(B(y,s))\bigr)&=2 s^2=2\, \mu(B(y,s))=\sigma_{\supp(\mu)}(B(y,s))\\
       &\overset{\eqref{eq:def_sigma_measure}}{=}\int\mathcal{H}^0\bigl(B(y,s)\cap \{z\in \supp(\mu):\pi_T(z)=t\} \bigr)\,dt\\
        &=\int\mathcal{H}^0\bigl(B(y,s)\cap \{z\in \mathrm{gr}(f):\pi_T(z)=t\} \bigr)\,dt=\mathcal{L}^1\bigl(\pi_T(\gr(f)\cap B(y,s))\bigr).
        \label{eq:eq:unifo}
    \end{split}
\end{equation}

It is immediate to see that if $f$ is not constant then, for some $s>0$, we have $$\mathcal{L}^1\bigl(\pi_T(\gr(f)\cap B(y,s))\bigr)<\mathcal{L}^1\bigl(\pi_T(B(y,s))\bigr),$$ which would contradict \eqref{eq:eq:unifo}.}
Hence, $f$ must be constant and in particular $$B(y,r/2)\cap \supp(\mu)=(x+\mathcal{V})\cap B(y,r/2)$$
for any $y\in B(x,r/2)\cap \supp(\mu)$.

Since $\mu$ is asymptotically doubling, we can find countably many disjoint balls $B(x_i,r_i)$
such that $\mu(\mathbb{P}^1\setminus \bigcup_iB(x_i,r_i))=0$, $r_i\in (0,r(x_i)/2)$, and 
$$B(y,r_i)\cap \supp(\mu)=(x_i+\mathcal{V})\cap B(y,r_i),$$
for $\mu$-almost every $x\in \mathbb{P}^1$. Thus, there is a sequence of points $z_i\in\mathbb{P}^1$ such that $$\supp(\mu)\subseteq \bigcup_{i\in\N}(z_i+\mathcal{V}).$$

Since $\supp(\mu)$ coincides with the zero set of an analytic function, this implies that for any $i\in\N$ the set $\supp(\mu)\cap (z_i+\mathcal{V})$ can be seen as analytic submanifolds of $\R$. Hence Lojasiewicz’s Structure Theorem for real analytic varieties (see \cite[Theorem 5.2.3]{KPanalyticmanifold}) yields that these sets are either a discrete collection of points or the entire line $\mathcal{V}$. However, it is easily seen that $\mu$ is uniform only if the sequence $z_i$ is constituted of just one point.
\end{proof}

\vv

\section{Marstrand-Mattila rectifiability criterion}\label{section:MM_rect_theorem}

Throughout this section we endow $\mathbb P^n$ with a metric $\mathfrak d$ which is invariant under translations, i.e. 
\begin{equation}\label{eq:transl_inv_d}
  \dd(z+x,z+y)=\dd(x,y)\qquad\text{for any }x,y,z\in\mathbb{P}^n,
\end{equation}
and homogeneous with respect to parabolic dilations, and we denote $\lVert x\rVert\coloneqq \dd(x,0)$ for $x\in\mathbb P^n.$

Let $\phi$ be a Radon measure on $(\mathbb P^n, \dd)$ that is supported on a compact set $K\subset \mathbb P^n$. For $\vartheta,\gamma\in\N$ we define
\begin{equation}
    E(\vartheta,\gamma)\coloneqq \bigl\{x\in K:\vartheta^{-1}r^h\leq \phi(B(x,r))\leq \vartheta r^h\text{ for any }0<r<1/\gamma\bigr\}.
    \label{eq:A1}
\end{equation}


\begin{proposition}\label{prop:cpt}
For any $\vartheta,\gamma\in\N$, the set $E(\vartheta,\gamma)$ defined in \eqref{eq:A1} is compact.
\end{proposition}

\begin{proof}
See \cite[Proposition 1.14]{MerloMM}.
\end{proof}

\vv

\begin{proposition}\label{prop::E}
Let $\phi$ be a Radon measure on $(\mathbb P^n, \dd)$ supported on a compact set $K\subset \mathbb P^n$ and such that $
0<\Theta^h_*(\phi,x)\leq \Theta^{h,*}(\phi,x)<\infty
$ for $\phi$-almost every $x\in\mathbb P^n$. Then 
\begin{equation}\label{eq:phiPEthetagamma}
    \phi\Bigl(\mathbb{P}^n\setminus \bigcup_{\vartheta,\gamma\in\N} E(\vartheta,\gamma)\Bigr)=0.
\end{equation}
\end{proposition}

\begin{proof}
Let $w\in K\setminus \bigcup_{\vartheta,\gamma} E(\vartheta,\gamma)$, which implies that either $\Theta^h_*(\phi,w)=0$ or $\Theta^{h,*}(\phi,w)=\infty$. Since $
0<\Theta^h_*(\phi,x)\leq \Theta^{h,*}(\phi,x)<\infty
$ for $\phi$-almost every $x\in\mathbb P^n$, this proves \eqref{eq:phiPEthetagamma}.
\end{proof}
\vv

We recall here a useful proposition about the structure of Radon measures.
\begin{proposition}[{\cite[Proposition 1.17 and Corollary 1.18]{MerloMM}}]\label{prop:MutuallyEthetaGamma}
Let $\phi$ be a Radon measure on $(\mathbb{P}^n, \dd)$ supported on a compact set $K\subset \mathbb P^n$ such that $0<\Theta^h_*(\phi,x)\leq \Theta^{h,*}(\phi,x)<\infty $ for $\phi$-almost every $x\in \mathbb{P}^n$. For every $\vartheta,\gamma\in \mathbb N$ we have that $\phi\llcorner E(\vartheta,\gamma)$ is mutually absolutely continuous with respect to $\mathcal{S}^h\llcorner E(\vartheta,\gamma)$.
\end{proposition}

\vv

\begin{definition}\label{def:Fk}
Given two Radon measures $\phi$ and $\psi$ on $\mathbb P^n$ and a compact set $K\subset \mathbb P^n$, we define 
\begin{equation}
    F_K(\phi,\psi)\coloneqq \sup\Bigl\{\Bigl|\int f\, d\phi - \int f\, d\psi\Bigr|:f\in \mathrm{Lip}_1^+(K)\Bigr\},
    \label{eq:F}
\end{equation}
where $\mathrm{Lip}_1^+(K)$ is the set of all functions $f\in \Lip(K)$ such that $f\geq 0$ and $\Lip(f)\leq 1$.

For $K=B(x,r)$ we also write $F_{x,r}\coloneqq F_{B(x,r)}$.
\end{definition}

\begin{remark}\label{rem:ScalinfFxr}
Few computations which we omit show that
\[
    F_{x,r}(\phi,\psi)=rF_{0,1}(T_{x,r}\phi,T_{x,r}\psi).
\]

Furthermore, the triangle inequality  holds for $F_K$; indeed, if $\phi_1,\phi_2,\phi_3$ are Radon measures and $f\in\lip(K)$, then
\begin{equation*}
\begin{split}
\Big\lvert\int f\, d\phi_1-\int f\, d\phi_2\Big\rvert&\leq \Big\lvert\int f\, d\phi_1-\int f\, d\phi_3\Big\rvert+\Big\lvert\int f\, d\phi_3-\int f\, d\phi_2\Big\rvert\\ &\leq F_K(\phi_1,\phi_2)+F_K(\phi_2,\phi_3).
\end{split}
\end{equation*}
The arbitrariness of $f\in\lip(K)$ implies that $F_K(\phi_1,\phi_2)\leq F_K(\phi_1,\phi_3)+F_K(\phi_3,\phi_2)$.
\end{remark}

For the proof of the following criterion we refer to \cite[Proposition 1.10]{MerloMM}.
\begin{proposition}\label{prop:WeakConvergenceAndFk}
Let $\{\mu_i\}$ be a sequence of Radon measures on $\mathbb P^n$. Let $\mu$ be a Radon measure on $\mathbb P^n$. The following are equivalent:
\begin{enumerate}
    \item $\mu_i\rightharpoonup \mu$.
    \item $F_K(\mu_i,\mu)\to 0$, for every $K\subseteq \mathbb P^n$ compact. 
\end{enumerate}
\end{proposition}

Now we are going to define a functional that, in some sense, quantifies how far a measure is from being flat around a point $x\in\mathbb P^n$ and at a scale $r>0$ (see Proposition \ref{prop::4.4(4)}).

\begin{definition}\label{def:metr}
For $x\in\mathbb{P}^n$, $h\in\{1,\ldots,n+2\}$, $r>0$, and $G\subseteq \mathfrak{M}(h)$ we define the functional
\begin{equation}\label{eqn:dxr}
    d_{x,r}(\phi,G)\coloneqq\inf_{\nu\in G} \frac{F_{x,r}(\phi,\nu)}{r^{h+1}}. 
\end{equation}
\end{definition}

\begin{remark}\label{rem:dxrContinuous}
It is routine to check that, whenever $h\in \mathbb N$ and $r>0$ are fixed, the function $x\mapsto d_{x,r}(\phi,\mathfrak M(h))$ is continuous. The proof is an adaptation of \cite[Proposition 2.2-(ii)]{MerloMM}.
\end{remark}

\vv

For $E\subseteq \mathbb P^n$ and $r>0$ we denote $B(E,r)\coloneqq \{x\in E: d(x,E)\leq r\}$ and, for 
$G\subseteq \Gr(h)$, we define
\[
    \mathfrak{M}(h,G)\coloneqq \bigl\{\lambda\mathcal C^h\llcorner V:\text{ for some }\lambda> 0 \text{ and }V\in G\bigr\}.
\]

For the Euclidean and Carnot groups analogue of the next proposition we refer to \cite[27, 4.4(4)]{Preiss1987GeometryDensities} and \cite[Proposition 2.28-(i)]{antonelli2022rectifiable} respectively.
\begin{proposition}\label{prop::4.4(4)}
Suppose that $h\in\{1,\ldots,n+2\}$, $\phi$ is a  Radon measure on $(\mathbb P^n,\dd)$ supported on a compact set $K\subset \mathbb P^n$, and let $G\subseteq \Gr(h)$. If for $\vartheta, \gamma\in\mathbb N$ there exist $x\in E(\vartheta,\gamma)$, $\sigma\in(0,2^{-10(h+1)}\vartheta^{-1})$, and $0<t<1/(2\gamma)$ such that
\begin{equation}\label{eq:prop4.4.4hp1}
    d_{x,t}(\phi,\mathfrak{M}(h,G))\leq \sigma^{h+4},
\end{equation}
then there exists $V\in G$ such that  whenever $y,z\in B(x,t/2)\cap (x+V)$ and $r,s\in [\sigma t, t/2]$ we have
\begin{equation}\label{eq:4.4.4_eq}
    \phi\bigl(B(y,r)\cap B(x+{V},\sigma^2t)\bigr)\geq \bigl(1-2^{10(h+1)}\vartheta\sigma\bigr)\bigl(\frac{r}{s}\bigr)^h\phi(B(z,s)).
\end{equation}
\end{proposition}

\begin{proof}First of all, we notice that by \eqref{eq:prop4.4.4hp1} and the definition of $d_{x,t}(\phi,\mathfrak{M}(h,G))$ there exist ${V}\in G$ and $\lambda>0$ such that
\begin{equation}\label{eq:fxtsmall}
    F_{x,t}\bigl(\phi,\lambda\mathcal{C}^h\llcorner (x+V)\bigr)\leq \sigma^{h+3}t^{h+1}.
\end{equation}

We claim that for any $w\in B(x,t/2)\cap (x+V)$, $\tau\in(0,t/2],$ and $\rho\in(0,\tau]$ we have
\begin{equation}
    \phi(B(w,\tau))\leq \lambda\mathcal{C}^h\llcorner(x+V)(B(w,\tau+\rho))+\sigma^{h+3}t^{h+1}/\rho,\label{eq:1.1031}
\end{equation}
and
\begin{equation}
    \lambda\mathcal{C}^h\llcorner(x+V)(B(w,\tau-\rho))\leq \phi\bigl(B(w,\tau)\cap B(x+V,\rho)\bigr)+\sigma^{h+3}t^{h+1}/\rho.\label{eq:1.1032}
\end{equation}

 In order to prove \eqref{eq:1.1031} we define $g(z)\coloneqq\min\{1,\rho^{-1}\dist(z,\mathbb{P}^n\setminus U(w,\tau+\rho))\}$, which satisfies $g\in\Lip(B(x,r))$, $\supp(g)\subseteq B(w,\tau+\rho)$, and $g\equiv 1$ on $B(w,\tau)$. Then, using \eqref{eq:fxtsmall}, we obtain
\begin{align}
    \phi(B(w,\tau))&\leq \int g(z)\, d\phi(z)\leq \int g(z)\,d\bigl[\lambda\mathcal{C}^h\llcorner (x+V)\bigr](z)+\text{Lip}(g) F_{x,t}(\phi,\lambda\mathcal{C}^h\llcorner (x+V))\nonumber \\ &\leq \lambda\mathcal{C}^h\llcorner (x+V)(B(w,\tau+\rho))+\sigma^{h+3}t^{h+1}/\rho.\nonumber
\end{align}
Similarly, to prove \eqref{eq:1.1032} we let $h(z)\coloneqq\min\bigl\{1,\rho^{-1}\dist\bigl(z,\mathbb{P}^n\setminus (U(w,\tau)\cap U(x+V,\rho))\bigr)\bigr\}$ and write
\begin{equation}
\begin{split}
     \lambda\mathcal{C}^h\llcorner(x+V)(B(w,\tau-\rho))&\leq \int h(z) \,d\bigl[ \lambda\mathcal{C}^h\llcorner(x+V)\bigr](z) \\
     &\leq \int h(z)\,d\phi(z)+\Lip(h) F_{x,t}\bigl(\phi,\lambda\mathcal{C}^h\llcorner (x+V)\bigr)\\&\leq \phi\bigl(B(w,\tau)\cap B(x+V,\rho)\bigr)+\sigma^{h+3}t^{h+1}/\rho.\nonumber 
\end{split}
\end{equation}

Since $x\in E(\vartheta,\gamma)$, the choices $w=x$, $\tau=t/4$, and $\rho=\sigma^2 t$ in \eqref{eq:1.1031} imply that
\begin{equation}
\begin{split}
\vartheta^{-1}(t/4)^h&\leq \phi(B(x,t/4))\leq \lambda\mathcal{C}^h\llcorner(x+V)\bigl(B(x,(1/4+\sigma^2)t)\bigr)+\sigma^{h+1}t^{h}\\ &=\lambda (1/4+\sigma^2)^ht^h+\sigma^{h+1}t^h.
\label{eq:chin4.4.4}
\end{split}
\end{equation}
The assumption $\sigma\leq (2^{10(h+1)}\vartheta)^{-1}$ yields $\sigma^{h+1}\leq (8^h\vartheta)^{-1}$, so from \eqref{eq:chin4.4.4} we infer
\begin{equation}\label{eq:estlambda}
    \vartheta^{-1}4^{-h} \leq \lambda(1/4+\sigma^2)^h+\sigma^{h+1} \quad \text{and in particular}\quad \lambda \geq \vartheta^{-1}2^{-3h},
\end{equation}
where we exploited the facts that $1/4+\sigma^2<1$, $\sigma^{h+1}\leq (8^h\vartheta)^{-1}$, and $4^{-h}-8^{-h}\geq 8^{-h}$.

Let us now prove that \eqref{eq:1.1031} and \eqref{eq:1.1032} imply \eqref{eq:4.4.4_eq}. Since by hypothesis $\min\{r,s\}\geq \sigma t$, with the choice $\rho=\sigma^2t$ we have $\rho<\min\{r,s\}$. Furthermore since $r,s\in [\sigma t, t/2]$ and $y,z\in B(x,t/2)\cap (x+V)$, the bounds \eqref{eq:1.1031} and \eqref{eq:1.1032} imply
\begin{equation}
    \begin{split}
        \frac{\phi(B(y,r)\cap B(x+V,\rho))}{\phi(B(z,s))}&\geq\frac{\lambda\mathcal{C}^h\llcorner(x+V)(B(y,r-\rho))-\sigma^{h+3}t^{h+1}\rho^{-1}}{\lambda\mathcal{C}^h\llcorner(x V)(B(z,s+\rho))+\sigma^{h+3}t^{h+1}\rho^{-1}} \\&=\frac{r^h}{s^h}\frac{\lambda(1-\sigma^2tr^{-1})^h-\sigma^{h+1}(t/r)^h}{\lambda (1+\sigma^2ts^{-1})^{h}+\sigma^{h+1}(t/s)^h}\\
        &\geq\frac{r^h}{s^h}\frac{\lambda(1-\sigma)^h-\sigma^{h+1}(t/r)^h}{\lambda (1+\sigma)^{h}+\sigma^{h+1}(t/s)^h}\geq \frac{r^h}{s^h}\frac{\lambda(1-\sigma)^h-\sigma}{\lambda (1+\sigma)^{h}+\sigma},
        \nonumber
    \end{split}
\end{equation}
where we are using $\sigma t/r\leq 1$ and $\sigma t/s\leq 1$.
Since $2h\sigma\leq 1$, we have that $(1+\sigma)^h\leq 1+2h\sigma$, that can be easily proved by induction on $h$. This together with \eqref{eq:estlambda} and Bernoulli's inequality $(1-\sigma)^h\geq 1-\sigma h$ allows us to finally infer that
$$
\frac{\phi(B(y,r)\cap B(x+V,\rho))}{\phi(B(z,s))}\geq \frac{r^h}{s^h} \frac{1-(\lambda h+1)\sigma\lambda^{-1}}{1+ (2h\lambda+1)\sigma\lambda^{-1}}\geq (1-2^{10(h+1)}\vartheta\sigma)\frac{r^h}{s^h},$$
where the last inequality comes from the fact that $\sigma\leq 1/2^{10(h+1)}\vartheta$, from \eqref{eq:estlambda} and some easy algebraic computations that we omit. An easy way to verify the last inequality is to show that $(1-\widetilde\alpha\sigma)/(1+\widetilde\beta\sigma)\geq 1-\widetilde\gamma\sigma$, where $\widetilde\alpha\coloneqq(\lambda h+1)/\lambda$, $\widetilde\beta\coloneqq(2h\lambda+1)/\lambda$ and $\widetilde\gamma\coloneqq2^{10(h+1)}\vartheta$, and observe that the latter inequality is implied by the fact that $\widetilde\alpha+\widetilde\beta-\widetilde\gamma\leq 0$.
\end{proof}

\vv

The next proposition is mostly an adaptation of \cite[Proposition 3.6]{antonelli2022rectifiable}, which generalizes \cite[Lemma 5.2]{Preiss1987GeometryDensities}. Given both its pivotal role and its length, we prefer to write its proof in full detail for completeness. The proof of Proposition \ref{prop::5.2}-(vi) is based on that of \cite[Proposition 3.10]{antonelli2022rectifiable}.
\begin{proposition}\label{prop::5.2}
Let $h\in\{1,\ldots,n+2\}$ and $\mathfrak{s}\in\mathfrak{S}(h)$.
Further let $r>0$, $\varepsilon\in (0,2\, 5^{-h-5}]$, $r_1\coloneqq\bigl(1-(\varepsilon/h)\bigr)r$, and $\mu\coloneqq2^{-7-5h}h^{-3}\varepsilon^2$.
Let $\phi$ be a Radon measure on $(\mathbb P^n,\dd)$ and $z\in\supp(\phi)$.

We define $Z(z,r_1)$ to be the set of triplets $(x,s, V)\in B(z, 2\,r_1)\times (0, 2r]\times \Gr^{\mathfrak s}(h)$ such that
\begin{equation}
    \phi(B(y,t))\geq (1-\varepsilon)\Bigl(\frac{t}{2r}\Bigr)^h\phi(B(z,2r)),
    \label{eq:2.9mm}
\end{equation}
whenever $y\in B(x, 2s)\cap (x+ V)$ and $t\in [\mu s,  2s]$.

The geometric assumption we make on $\phi$ is that we can find a compact set $E\subseteq B(z, 2r_1)$ such that $z\in E$, 
\begin{equation}
    \phi(B(z, 2r_1)\setminus E)\leq 2^{-h}\mu^{h+1} \phi(B(z, 2r_1)),
    \label{eq:densityE}
\end{equation}
and such that for any $x\in E$ and every $s\in(0, 2r-\dd(x,z)]$ there exists $ V\in\Gr^{\mathfrak s}(h)$ such that $(x,s, V)\in Z(z,r_1)$. Furthermore, we assume that there exists $  W\in \Gr^{\mathfrak{s}}(h)$ such that $(z,r,  W)\in Z(z,r_1)$.

Let us write $P\coloneqq P_W$, namely the orthogonal projection on $W$,
and denote as $T(u,s)\coloneqq P^{-1}(B(u,s)\cap   W)$ the cylinder with center $u\in W$ and radius $s>0$.

For any $u\in P(B(z,r_1))$ let $s(u)\in[0,r]$ be the smallest number such that for any $s(u)<s\leq r$ the following properties hold:
\begin{enumerate}
    \item $E\cap T(u,s/(4h))\neq \emptyset$.
    \item $\phi\big(B(z, 2r)\cap T(u, s))\leq \mu^{-h}(s/ (2r))^h\phi(B(z, 2r))$.
\end{enumerate}
Finally, we define:

{\begin{itemize}
    \item $\text{ }\text{ }A\coloneqq\{u\in P(B(z,r_1)):s(u)=0\}$.
    \item $A_1$ as the set of $u\in P(B(z,r_1))$ such that $s(u)>0$ and
    \[
        \phi\big(B(z, 2r)\cap T(u, s(u))\big)\geq \varepsilon^{-1}\Bigl(\frac{s(u)}{2r}\Bigr)^h\phi(B(z, r)).
    \]
    \item $A_2$ as the set of $u\in P(B(z,r_1))$ such that  $s(u)>0$ and
    \[
        \phi\Bigl((B(z, 2r)\setminus E)\cap T\Bigl(u,\frac{s(u)}{4h}\Bigr) \Bigr)\geq 2^{-1}\Bigl(\frac{s(u)}{8h r}\Bigr)^h\phi(B(z, r)).
    \]
\end{itemize}
}

Then we have:
\begin{itemize}
    \item[(i)] $s(u)\leq   2h\mu r$ for every $u\in P(B(z,r_1))$.
    \item[(ii)] The function  $u\mapsto s(u)$ is lower semicontinuous on $P(B(z,r_1))$ and, as a consequence, $A$ is compact.
    \item[(iii)] $ P(B(z,r_1))= A\cup A_1\cup A_2$.
    \item[(iv)]$\mathcal{C}^h(P(B(z,r))\setminus A)\leq 2^{3h}5^{h+3} \mathcal{C}^h(P(B(0,1)))\varepsilon r^h$.
    \item[(v)] $P(E\cap P^{-1}(A))=A$, $\mathcal{H}^h(E\cap P^{-1}(A))>0$, and there is a constant $\mathfrak{C}>1$ such that
    $$\mathfrak{C}^{-1}\mathcal{H}^h(E\cap P^{-1}(A))\leq \phi(E\cap P^{-1}(A))\leq \mathfrak{C}\mathcal{H}^h(E\cap P^{-1}(A)).$$
    \item[(vi)] There exists a Lipschitz function $f\colon V\to V^\perp$ such that
    $$\phi\bigl(\mathrm{gr}(f)\cap E\cap P^{-1}(A)\bigr)>0.$$
    \end{itemize}
\end{proposition}

\begin{proof} We prove each point of the proposition in a separate paragraph. For brevity we write $Z\coloneqq Z(z,r_1)$. Moreover, $P$ is a homogeneous homomorphism, so the statement is translation-invariant and we assume $z=0$ without loss of generality.

Before proving (i) first observe that if $(0,r,  W)\in Z$, we infer from \eqref{eq:2.9mm} that
\begin{equation}
    \phi(B(0, 2r_1))\geq (1-\varepsilon)\Bigl(\frac{r_1}{r}\Bigr)^h\phi(B(0,2r)),
    \nonumber
\end{equation}
so
\begin{equation}
\begin{split}
      \phi(B(0, 2r)\setminus B(0, 2r_1))
      &= \phi(B(0, 2r))-\phi(B(0,2 r_1))\leq \phi(B(0, 2r))\bigl(1-(1-\varepsilon)(r_1/r)^h\bigr) \\
      &=\phi(B(0, 2r))\bigl(1-(1-\varepsilon)\bigl(1-(\varepsilon/h)\bigr)^h\bigr)\leq 2\varepsilon \phi(B(0, 2r)),
\end{split}
  \label{bd:bd1}
\end{equation}
where in the last inequality we used that $h \mapsto (1-\varepsilon/h)^h$ is an increasing function.

\smallskip

\item\paragraph{\textbf{Proof of (i):}} Let $u\in P(B(0,r_1))$ and $  2\mu h r<s\leq r$. The condition $2\mu h r<s$ implies
$$\phi(B(0, 2r)\cap T(u, s))\leq \phi(B(0, 2r))< \mu^{-h}(s/ (2r))^h\phi(B(0, 2r)).$$

Defined $v\coloneqq u-\delta_\mu(u)$, we note that $v\in   W$ and 
\[
    \dd(v,u)\overset{\eqref{eq:transl_inv_d}}{=}\dd\bigl(\delta_\mu(u),0\bigr) =\mu \,\dd(u,0)\leq\mu r,
\]
where the second equality follows from the homogeneity of $\dd$.

Furthermore, for every $\Delta\in B(0,\mu r)$ we have
{\begin{equation}
    \begin{split}
    \dd\bigl(0,u-\delta_\mu(u)+\Delta\bigr)&\leq \mu\|u\|+\|u\|+\|\Delta\|\leq \mu r_1+r_1+\mu r \leq 2  r_1,
    \label{incl:palli}
\end{split}
\end{equation}}
where in the inequality above we used the fact that $r_1>r/2$.
Thus, on the one hand we have $B(v,\mu r)\subseteq B(u,\mu r)$ and on the other \eqref{incl:palli} reads
\begin{equation}\label{eqn:Deduzione}
B(v,\mu r)\subseteq B(0,2r_1).
\end{equation}

Since $(0,r,  W)\in Z$, the definitions of $Z$ and $E$ imply that
\begin{equation}
    \phi(B(v,\mu r))\overset{\eqref{eq:2.9mm}}{\geq} (1-\varepsilon)\mu^h 2^{-h}\phi(B(0,r))\geq (1-\varepsilon)\mu^h2^{-h} \phi(B(0,2r_1))\overset{\eqref{eq:densityE}}{>} \phi(B(0,2r_1)\setminus E).
    \label{eq:ineq:balls}
\end{equation}
Furthermore, thanks to \eqref{eqn:Deduzione}, \eqref{eq:ineq:balls} and the definition of $T(\cdot,\cdot)$, we also infer that
$$
\emptyset\neq E\cap B(v,\mu r) \subseteq E\cap B(u,2\mu r)\subseteq E\cap T\bigl(u,s/(4h)\bigr),
$$
where the last inclusion holds since $2\mu r\leq \mu r /2 <s/(4h)$.

\vv

\paragraph{\textbf{Proof of (ii):}} Let $u\in P(B(0,r_1))$ and $0<s\leq s(u)$. By definition of $s(u)$, up to possibly increasing $s$ such that it still holds $0<s\leq s(u)$, there are two cases: either
\begin{equation}
    \phi(B(0,2r)\cap T(u,s))>(1+\tau)^h\mu^{-h}(s/(2r))^h\phi(B(0,2r)),
    \label{eq:defs(u)}
\end{equation}
 for some $\tau>0$ or
\begin{equation}
    E\cap T(u,s/(4h))= \emptyset.
    \label{eq:defs(u)2}
\end{equation}

We remark that the point (i) implies $s(u)\leq r$.
Thus, for $v\in P(B(0,r_1))$  such that $\dd(u,v)\leq \min \{\tau s, (r-s)/2\}$, we have 
\begin{equation}\label{eq:choice_v_close_u}
    s+\dd(u,v)\leq (1+\tau)s\qquad  \text{ and } \qquad s+\dd(u,v)\leq \frac{s+r}{2}\leq r.
\end{equation}

If \eqref{eq:defs(u)} holds, this implies that
\begin{equation}\label{eqn:EXT1}
\begin{split}
        \phi\bigl(B(0,2r)\cap T(v,s+ \dd(u,v))\bigr)&>\phi\bigl(B(0,2r)\cap T(u,s)\bigr)\geq (1+\tau)^h\mu^{-h}(s/(2r))^h\phi(B(0,2r)) \\
        &\overset{\eqref{eq:choice_v_close_u}}{\geq}\mu^{-h} \Bigl(\frac{s+\dd(u,v)}{2r}\Bigr)^h\phi(B(0,2r)).
\end{split}
\end{equation}

On the other hand, if \eqref{eq:defs(u)2} holds, then
\begin{equation}\label{eqn:EXT2}
E\cap T\Bigl(v,\frac{s-4h\dd(u,v)}{4h}\Bigr)\subseteq E\cap T\Bigl(u,\frac{s}{4h}\Bigr)=\emptyset.
\end{equation}

Taking into account \eqref{eqn:EXT1} and \eqref{eqn:EXT2}, this shows that \[s(v)\geq\min\{s-4h\dd(u,v),s+\dd(u,v)\}=s-4h\dd(u,v)\]  for $v\in P(B(0,r_1))$  such that $\dd(u,v)\leq \min \{\tau s, (r-s)/2\}$. Hence, we have
$\liminf_{v\to u}s(v)\geq s$ for any $s\leq s(u)$ for which at least one between \eqref{eq:defs(u)} and \eqref{eq:defs(u)2} holds. In particular from the definition of $s(u)$ we deduce the existence of a sequence $s_i\to s(u)^-$ such that at each $s_i$ at least one between \eqref{eq:defs(u)} and \eqref{eq:defs(u)2} holds. In conclusion we infer
\[
    \liminf_{v\to u} s(v)\geq s(u).
\]

\smallskip

\paragraph{\textbf{Proof of (iii):}} The inclusion $P(B(z,r_1))\supseteq A\cup A_1\cup A_2$ holds trivially by the definition of $A$, $A_1$, and $A_2$, so we are left with the proof of the converse relation. Suppose that $P(B(0,r_1))\neq A\cup A_1$ and let $u\in P(B(0,r_1))\setminus (A\cup A_1)$. Since $u\not\in A\cup A_1$, then $s(u)>0$ and
\begin{equation}
    \phi\big(B(0,2r)\cap T(u,s(u))\big)< \varepsilon^{-1}(s(u)/(2r))^h\phi(B(0,2r)).
    \label{eq:contro}
\end{equation}

Thanks to the definition of $s(u)$, for any $0<s<s(u)$, up to eventually increasing $s$ in such a way that it still holds $0<s<s(u)$, we have
\begin{equation}
    \phi(B(0,r)\cap T(u,s))>\mu^{-h}(s/r)^h\phi(B(0,r))
    \label{eq:defs(u)prim}
\end{equation}
or
\begin{equation}
    E\cap T(u,s/(4h))= \emptyset.
    \label{eq:defs(u)2prim}
\end{equation}

Let us assume that there exists $s<s(u)$ such that $\eqref{eq:defs(u)2prim}$ does not hold. Then 
\[
   E\cap T(u,t/(4h))\neq \emptyset  \qquad \text{ for all } \quad t\in [s,s(u))
\]
and, by \eqref{eq:defs(u)prim}, there exists a sequence $t_i<s(u)$ such that $t_i\to s(u)$ and
\begin{equation}
    \phi(B(0,2r)\cap T(u,t_i))>\mu^{-h}(t_i/(2r))^h\phi(B(0,2r)).
    \label{eq:defs(u)prim_i}
\end{equation}

Thus it holds
\begin{equation}
    \begin{split}
        \Bigl(\frac{s(u)}{2\mu r}\Bigr)^h\phi(B(0,  2r))&=\lim_{i\to+\infty}\mu^{-h}\Bigl(\frac{t_i}{2r}\Bigr)^h\phi(B(0,2 r))        \overset{\eqref{eq:defs(u)prim_i}}{\leq} \limsup_{i\to +\infty}\phi(B(0,  2r)\cap T(u,  t_i)) \\
        & \leq \phi(B(0,  2r)\cap T(u,  s(u)))  
        \overset{\eqref{eq:contro}}{<} \varepsilon^{-1}\Bigl(\frac{s(u)}{2r}\Bigr)^h\phi(B(0,  2r)),
    \end{split}
\end{equation}
that yields a contradiction because of the choice of $\mu$ and $\varepsilon$. So, for any  $0<\rho<s(u)$ we have $E\cap  T(u,\rho/(4h))=\emptyset$ and hence
\begin{equation}
    \begin{split}
        E\cap \textrm{int}\bigl(T(u,s(u)/(4h))\bigr)= \emptyset,
        \nonumber
    \end{split}
\end{equation}
where $\textrm{int}(\cdot)$ denotes the interior of the set.

Let us now define the constants
$$
    \bar s\coloneqq 16hs(u)/\varepsilon,\qquad\text{and}\qquad\sigma\coloneqq(2h-1)\varepsilon/(32h^2).
$$
Thanks to the point (i) and the definitions of $\mu$ and $\bar s$, we deduce that
\begin{equation}
        0<s(u)\leq \bar s\leq r-r_1, \qquad \text{and}\qquad \mu\leq \sigma\leq 1.
\end{equation}
The compactness of $E$ and the definition of $s(u)$ yield
$$E\cap T\bigl(u,s(u)/(4h)\bigr)\neq \emptyset.$$

Fix $x\in E\cap T\bigl(u,s(u)/(4h)\bigr)$ and assume  $ V\in  \Gr^{\mathfrak s}(h)$ to be such that $(x,\bar s, V)\in Z$. We claim that
\begin{equation}
    \lVert P(  y-x)\rVert\geq \sigma\lVert   y-x\rVert,\qquad\text{for every }y\in x+  V.
    \label{eq:lipcond}
\end{equation}

In order to prove \eqref{eq:lipcond}, we assume by contradiction that there exists $y\in x+ V$ such that $\lVert   y-x\rVert=1$ and $\lVert P(  y-x)\rVert<\sigma$. Let us fix $w\in B(0,\sigma \bar s)$ and let $t\in\mathbb R$ be such that $\lvert t\rvert\leq s(u)/(4h\sigma)$. Then, we have
\begin{equation}
    \dd(0,x+\delta_t(  y-x)+w)\leq \dd(0,x)+\lvert t\rvert \bigl\lVert   y-x\bigr\rVert+\sigma \bar s\leq \dd(0,x)+\frac{  s(u)}{4h\sigma}+\sigma \bar s.
    \label{eq:bdB(C7)}
\end{equation}

Thanks to the choice of the constants and the bound $s(u)\leq   2h\mu r$ from (i), we infer that
\begin{equation}
\begin{split}
    \frac{  s(u)}{4h\sigma}+\sigma \bar s &=   s(u)\Bigl(1-\frac{1}{2h}+\frac{8h}{(2h-1)\varepsilon}\Bigr)
    \leq   \frac{\varepsilon^2r}{2^{7}h^{2}}\Bigl(1-\frac{1}{2h}+\frac{8h}{(2h-1)\varepsilon}\Bigr)\leq   \frac{\varepsilon r}{h},
    \label{eq:birdi1}
\end{split}
\end{equation}
where in the second inequality above we are using that $\mu=2^{-7-5h}h^{-3}\varepsilon^2$ by definition.
Thus, since $x\in B(0,r_1)$, we gather \eqref{eq:bdB(C7)},  \eqref{eq:birdi1}, and infer that 
\begin{equation}\label{eqn:Estdeltat}
\dd(0,x+\delta_t(  y-x)+w)\leq 2r_1+  \frac{\varepsilon r}{h}=2r,
\end{equation}
where the last equality comes from the definition of $r_1$. As a consequence of \eqref{eqn:Estdeltat} we finally deduce that
$$
    B(x+\delta_t(  y-x),\sigma \bar s)\subseteq B(0,2r), \qquad\text{for any }\lvert t\rvert\leq s(u)/(4h\sigma).
$$

We now prove that for any $\lvert t\rvert\leq  s(u)/(4h\sigma)$ and any $w\in B(0,\sigma \bar s)$, we have
\begin{equation}\label{eqn:TIME}
x+\delta_t(  y-x) w\in T(u,s(u)).
\end{equation}
Indeed, it holds $$P(x+\delta_t(  y-x)+w)=P(x)+\delta_t(P(  y-x))+P(w),$$ so the assumption $x\in T(u,s(u)/(4h))$ yields
\begin{equation}\label{eq:ddup1}
  \dd(u,P(x))\leq s(u)/(4h).
\end{equation}

Observe that by definition we have \(\sigma\bar s=(1-(1/2h))s(u)\).
Then, by \eqref{eqn:TIME} and the fact that $\|P(w)\|\leq \sigma \bar s$ since $P$ is $1$-Lipschitz, we can estimate
\begin{equation}
    \begin{split}
\dd\bigl(u,P(x)+\delta_t(P(  y-x))+P(w)\bigr)&\leq \dd(u,P(x))+\lvert t\rvert\lVert P(  y-x)\rVert+\sigma \bar s\\
&\overset{\eqref{eq:ddup1}}{\leq} \frac{s(u)}{4h}+\frac{s(u)}{4h}+\sigma \bar s=\frac{s(u)}{2h}+\Big(1-\frac{1}{2h}\Big)s(u)
=s(u).
\nonumber
    \end{split}
\end{equation}

In conclusion, the above computations yield that
\begin{equation}\label{eqn:TIME2}
B\bigl(x+\delta_t(  y-x),\sigma \bar s\bigr)\subseteq B(0,2r)\cap T(u,s(u)), \quad \text{ for any }\lvert t\rvert\leq s(u)/(4h\sigma).
\end{equation}

Now observe that
\begin{equation}\label{eq:mmtl31}
    (1-\varepsilon)\Bigl(1-\frac{1}{2h}\Bigr)^h\frac{16h^2}{(2h-1)^{2}}= 4(1-\varepsilon)\Bigl(1-\frac{1}{2h}\Bigr)^{h-2}\geq 2(1-\varepsilon)\geq 1,
\end{equation}
so by applying Fubini's theorem to the function
$$F(t,z)\coloneqq \chi_{B(0,\sigma \bar s)}\bigl(z-x+\delta_t(x-y)\bigr), \qquad (t,z)\in \Bigl[-\frac{s(u)}{4h\sigma},\frac{s(u)}{4h\sigma}\Bigr]\times \mathbb{P}^n,$$
noticing that when $|t|\leq s(u)/(4h\sigma)$ we have \eqref{eqn:TIME2}, and since $x\in E$ implies that $(x, \bar s,  V)\in Z$ for some $  V\in { \Gr^{\mathfrak s}(h)}$, we write
\begin{equation}
\begin{split}
    \phi\bigl(B(0,2r)\cap T(u,s(u))\bigr)&\geq (2\sigma \bar s)^{-1}\int_{-s(u)/(4h\sigma)}^{s(u)/(4h\sigma)}\phi\bigl(B(x+\delta_t(  y-x),\sigma s)\bigr)\,  dt\\
    &\geq (2\sigma \bar s)^{-1}\Bigl(\frac{s(u)}{2h\sigma}\Bigr)(1-\varepsilon)\Bigl(\frac{\sigma \bar s}{2r} \Bigr)^h\phi(B(0,  2r))\\
    &= (1-\varepsilon)\Bigl(1-\frac{1}{2h}\Bigr)^h\frac{16h^2}{(2h-1)^{2}}\varepsilon^{-1}\Bigl(\frac{s(u)}{2r}\Bigr)^h\phi(B(0,  2r))\\
    &\overset{\eqref{eq:mmtl31}}{\geq} \varepsilon^{-1}\Bigl(\frac{s(u)}{2r}\Bigr)^h\phi(B(0,  2r)).
\end{split}
\label{eqn:ContR}
\end{equation}

However, \eqref{eqn:ContR} contradicts the assumption $u\notin A_1$ (see \eqref{eq:contro}), thus the claim \eqref{eq:lipcond} holds, so  $P\lvert_{ V}$ is an homomorphism and hence injective.

In particular we can find $w\in x+ V$ such that $P(w)=u$ and we infer that $d(u,P(x))\leq s(u)/(4h)$. So, we conclude that
\begin{equation}\label{eqn:xw}
\lVert  w-x\rVert\overset{\eqref{eq:lipcond}}{\leq}\sigma^{-1}\lVert P(w)-P(x)\rVert=\sigma^{-1}\lVert u- P(x)\rVert\leq \frac{s(u)}{4h\sigma}.
\end{equation}

Then, with \eqref{eqn:xw} we can repeat the same computation we performed in \eqref{eq:bdB(C7)}-\eqref{eq:birdi1}-\eqref{eqn:Estdeltat} and obtain $B(w,s(u)/(4h))\subseteq B(0,r)$. Furthermore, since $P(w)=u$, $B(w,s(u)/(4h))\subseteq T(w,s(u)/(4h))=T(u,s(u)/(4h))$,  and $\text{int}(T(u,s(u)/(4h)))\cap E=\emptyset$ we have
\begin{equation}
    U\Bigl(w,\frac{s(u)}{4h}\Bigr)\subseteq (B(0,2r)\setminus E)\cap \text{int}\Bigl(T\Bigl(u,\frac{s(u)}{4h}\Bigr)\Bigr).
    \label{eq:claimfinale3}
\end{equation}

We claim that \eqref{eq:claimfinale3} concludes the proof of item (iii). Indeed we have $(x,\bar s,  V)\in Z$, and the bounds $\mu \bar s\leq s(u)/(4h)\leq \bar s$ together with \eqref{eqn:xw} imply $w\in B(x,\bar s)\cap (x+V)$, so by approximation and using the hypothesis we obtain the inclusion
\begin{equation}
    \phi\Bigl(U\Bigl(w,\frac{s(u)}{4h}\Bigr)\Bigr)\geq (1-\varepsilon)\Bigl(\frac{s(u)}{8hr}\Bigr)^h\phi(B(0,r)).
    \label{eq:claimfinale3.2}
\end{equation}

We gather \eqref{eq:claimfinale3} and \eqref{eq:claimfinale3.2}, and deduce that
\begin{equation}
    \phi\Big((B(0,2r)\setminus E)\cap \text{int}\Bigl(T\Bigl(u,\frac{s(u)}{4h}\Bigr)\Bigr)\Big)\geq (1-\varepsilon)\Bigl(\frac{s(u)}{8h r}\Bigr)^h\phi(B(0,2r))
    \nonumber
\end{equation}
and thus $u\in A_2$, which proves item (iii).

\smallskip

\paragraph{\textbf{Proof of (iv):}} Let $\tau>1$. By \cite[Theorem 2.8.4]{Federer1996GeometricTheory} we have that there exists a countable set $D\subseteq A_1$ such that:
\begin{itemize}
    \item[($\alpha)$] $\{B(w,s(w))\cap  W:w\in D\}$ is a disjointed subfamily of $\{B(w,s(w))\cap  W:w\in A_1\}$.
    \item[($\beta)$] For any $w\in A_1$ there exists $u\in D$ such that $B(u,s(u))\cap B(w,s(w))\cap  W\neq \emptyset$ and $s(w)\leq \tau s(u)$.
\end{itemize} 

Furthermore, if for every $u\in A_1$ we define
\begin{equation}\label{eqn:Bhat}
\begin{split}
    \widehat{B}(u,  s(u))\coloneqq\bigcup\bigl\{B(w,  s(w))\cap  W:w\in A_1,
    ~B(u,  s(u))\cap B(w,  s(w))\cap  W\neq \emptyset,~s(w)\leq \tau s(u) \bigr\},
\end{split}
\end{equation}
 thanks to \cite[Corollary 2.8.5]{Federer1996GeometricTheory} we deduce that
 \begin{equation}\label{eq:mmp4e1}
     A_1\subseteq \bigcup_{u\in A_1} B(u,  s(u))\cap  W\subseteq \bigcup_{w\in D} \widehat{B}(w,  s(w)).
 \end{equation}
    Triangle inequality and an elementary computation yield
\begin{equation}
    \widehat{B}(u,  s(u))\subseteq   W\cap B\bigl(u,(1+2\tau)  s(u)\bigr), \qquad \text{for every $u\in A_1$}.
    \label{eq:inclussion}
\end{equation}

Since $D\subseteq A_1$ and  $T(u, s(u))\subseteq P^{-1}(B(u,  s(u))\cap   W)$ for every $u\in A_1$, we exploit the fact that $\{B(w,  s(w))\cap  W:w \in D\}$ is a disjointed family and obtain that
\begin{equation}\label{eq:phiBleqe1sum}
   \phi(B(0, 2r))\geq \sum_{u\in D}\phi\bigl(B(0,   2 r)\cap T(u, s(u))\bigr)\geq \varepsilon^{-1}\sum_{u\in D}(s(u)/   (2r))^h\phi(B(0, 2r)), 
\end{equation}
where the last inequality above comes from the inclusion $D\subseteq A_1$. Observe that \eqref{eq:phiBleqe1sum} also reads \[\sum_{u\in D}s(u)^h\leq      \varepsilon 2^hr^h\]
and thanks to \eqref{eq:mmp4e1} and \eqref{eq:inclussion} we infer that
\begin{equation}
\begin{split}
        \mathcal{C}^h(A_1)&\overset{}{\leq} \sum_{u\in D}\mathcal{C}^h\bigl(B(u,(1+2\tau)  s(u))\cap  W\bigr)=(1+2\tau)^h\sum_{u\in D}s(u)^h\leq     (1+2\tau)^h\varepsilon 2^hr^h.
\end{split}
        \label{eq:bdes1}
\end{equation}

Similarly to what we did for $D$, we can construct a countable set $D^\prime\subseteq A_2$ such that $\{B(u,   s(u)/(4h))\cap  W:u\in D^\prime\}$ is a disjointed family and the collection $\{\widehat{B}(u ,  s(u)/(4h)):u\in D^\prime\}$, constructed analogously to \eqref{eqn:Bhat}, covers the set $A_2$. Analogously to \eqref{eq:inclussion} we have 
\[\widehat{B}\Bigl(u,\frac{s(u)}{4h}\Bigr)\subseteq   W\cap B\Bigl(u,(1+2\tau)\frac{s(u)}{4h}\Bigr) \qquad\textrm{ for every } u\in A_2.\]

Moreover, since 
$$T\Bigl(u,\frac{s(u)}{4h}\Bigr)\subseteq P^{-1}\Bigl(B\Bigl(u, \frac{s(u)}{4h}\Bigr)\cap   W\Bigr) \qquad \text{ for every } u\in A_2,$$
 and $\{B(u, s(u)/(4h))\cap  W:w \in D'\}$ is a disjointed family, we conclude
\begin{equation}\label{eq:prop36_aaa}
    \begin{split}
        \phi(B(0,   2r)\setminus E)&\geq \sum_{u\in D^\prime}\phi\Bigl((B(0, 2r)\setminus E)\cap T\Bigl(u,\frac{s(u)}{4h}\Bigr)\Bigr)\geq 2^{-1}\phi(B(0,  2 r))\sum_{u\in D^\prime}\Bigl(\frac{s(u)}{8h   r}\Bigr)^h,
    \end{split}
\end{equation}
where the last inequality holds since $D'\subseteq A_2$.

From \eqref{eq:prop36_aaa}, \eqref{bd:bd1}, and the fact that $0\in E$, we infer that
\begin{equation}\label{eq:prop36_eq1}
    \begin{split}
        \sum_{u\in D^\prime}\Bigl(\frac{s(u)}{8h   r}\Bigr)^h&\leq\frac{2\phi(B(0,   2r)\setminus E)}{\phi(B(0,   2r))}\leq 2\cdot\frac{\phi(B(0,   2r)\setminus B(0,   2r_1))+\phi(B(0,   2r_1)\setminus E)}{\phi(B(0,   2r))} \\
        &\leq2\cdot\frac{2\varepsilon \phi(B(0,   2r)) +\mu^{h+1}2^{-h}   \phi(B(0,  2 r))}{\phi(B(0,   2r))}\leq 10\varepsilon.
        \end{split}
\end{equation}
Consequently, we deduce that
\begin{equation}
\begin{split}
     \mathcal{C}^h(A_2)&\leq \sum_{u\in D^\prime} \mathcal{C}^h\Bigl(  W\cap B\Bigl(u,(1+2\tau)\frac{s(u)}{4h}\Bigr)\Bigr)\\
     &=(1+2\tau)^h\sum_{u\in D^\prime} (s(u)/4h)^h\overset{\eqref{eq:prop36_eq1}}{\leq} 10(1+2\tau)^h  \varepsilon2^h r^h.
\end{split}
    \label{eq:bdes2}
\end{equation}

We remark that $P(B(0,1))\supseteq B(0,1)\cap  W$ and $\mathcal{C}^h(B(0,1)\cap  W)=1$ imply $\mathcal{C}^h\bigl(P(B(0,1))\bigr)\geq 1$. Thus, we gather \eqref{eq:bdes1}, \eqref{eq:bdes2}, item (iii) of this proposition, and conclude that
\begin{equation}
\begin{split}
        \mathcal{C}^h(P(B(0,r))\setminus A)&\leq \mathcal{C}^h(P(B(0,r))\setminus P(B(0,r_1)))+\mathcal{C}^h(A_1)+\mathcal{C}^h(A_2)\\
        &\leq \mathcal{C}^h(P(B(0,1)))r^h(1-(1-\varepsilon/h)^h)+    (1+2\tau)^h\varepsilon 2^hr^h+10(1+2\tau)^h     \varepsilon 2^hr^h\\
        &\leq 50(1+2\tau)^h   \mathcal{C}^h(P(B(0,1))) \varepsilon 2^hr^h.
\end{split}
\nonumber
\end{equation}
We finally choose $\tau=2$ and item (iv) follows.
\smallskip

\paragraph{\textbf{Proof of (v):}} Let $u\in A$ and note that since $s(u)=0$, for any $s>0$ we have that
$$
    E\cap T(u,s/(4h))\neq \emptyset.
$$

The sets $E\cap T(u,s/(4h))$ are compact, so the finite intersection property yields
$$
    \emptyset\neq E\cap \bigcap_{s>0}T(u,s/(4h))=E\cap P^{-1}(u).
$$

This implies that $u\in P(E\cap P^{-1}(u))$ for every $u\in A$, and as a consequence $A\subseteq P(E\cap P^{-1}(A))$. Since the inclusion $ P(E\cap P^{-1}(A))\subseteq A$ is obvious we finally infer that
$A=P(E\cap P^{-1}(A))$. Moreover, thanks to item (iv) and the choice $\varepsilon<2^{-3h}5^{-h-5}$, it holds
\begin{equation}
\begin{split}
     \mathcal{C}^h(A)&\geq \mathcal{C}^h\bigl(P(B(0,r))\bigr)-\mathcal{C}^h\bigl(P(B(0,r))\setminus A\bigr)\\
     &\geq \mathcal{C}^h\bigl(P(B(0,1))\bigr)r^h-5^{h+3} 2^{-3h}\mathcal{C}^h\bigl(P(B(0,1))\bigr) \varepsilon r^h\geq\frac{24}{25}r^h>0,
\end{split}
    \nonumber
\end{equation}
so we conclude that $\mathcal{S}^h(A)>0$ by the equivalence of $\mathcal{C}^h\llcorner  W$ and $\mathcal{S}^h\llcorner  W$  in Remark \ref{remarkone}.

The fact that $P$ is $1$-Lipschitz further gives that
$$0<\mathcal{S}^h(A)=\mathcal{S}^h\bigl(P(E\cap P^{-1}(A))\bigr)\leq\mathcal{S}^h(E\cap P^{-1}(A)).$$

For any $s$ sufficiently small and $u\in A$, by definition of $s(u)$ and $A$, we have
\begin{equation}\label{eq:propMM_main_aux1}
    \phi(B(x, s))\leq \phi\big(B(0,   2r)\cap T(u, s))\leq \mu^{-h}(s/ (2r))^h\phi(B(0,   2r)),
\end{equation}
whenever $x\in E\cap P^{-1}(u)$, where the first inequality comes from the fact that $x\in E\subseteq B(0,   r_1)$. Finally by \cite[2.10.17-(2)]{Federer1996GeometricTheory} and \eqref{eq:propMM_main_aux1} we infer
\begin{equation}
    \phi (E\cap P^{-1}(A))\leq   2^{-h}\mu^{-h}\frac{\phi(B(0,   r))}{r^{h}}\mathcal{S}^h (E\cap P^{-1}(A)).
    \label{eq:bdhaus1}
\end{equation}

On the other hand, if we assume $x\in E$ and $s$ sufficiently small, we have $(x,s, V)\in Z$ for some $ V\in { \Gr^\mathfrak{s}(h)}$. By using the definition of $Z$, this implies that
$$\phi(B(x,s))\geq (1-\varepsilon)(s/ (2r)^h\phi(B(0,   2r)),$$
and thus by \cite[2.10.19(3)]{Federer1996GeometricTheory}, we have
\begin{equation}
\phi\llcorner E\geq (1-\varepsilon)\frac{\phi(B(0, r))}{2^hr^{h}}\mathcal{S}^h\llcorner E.
    \label{eq:bdhaus2}
\end{equation}

We gather \eqref{eq:bdhaus1} and \eqref{eq:bdhaus2}, and conclude the proof of item (v).

\smallskip

\paragraph{\textbf{Proof of (vi):}}

Let $\widetilde{A}$ be the set of those $u\in A$ for which there exists $\rho(u)>0$ such that
\begin{equation}
    \phi(B(0, 2r)\cap T(u,s))\leq 2(1-\varepsilon)^4(s/ (2r))^h\mathcal{C}^h\bigl(P(B(0,1))\bigr)\phi(B(0, 2r)),
    \label{eq:num1uno}
\end{equation}
for all $0<s<\rho(u)$. We claim that $\widetilde{A}$ is a Borel set. To prove this, we first note that
$$
\widetilde{A}=\bigcup_{k\in\N}\bigl\{u\in A: \eqref{eq:num1uno}\text{ holds for any }0<s<1/k\bigr\}\eqqcolon\bigcup_{k\in\N}\widetilde{A}_k.
$$

Let us show that $\widetilde A_k$ is a compact set for any $k\in\N$ and, in order to do this, fix $k$ and assume that $\{u_i\}_{i\in\N}$ is a sequence of points of $\widetilde A_k$. Since $\widetilde A_k\subseteq A$ and $A$ is compact we can suppose that, up to a non re-labelled subsequence, $u_i$  converges to some $u\in A$. Thus, for every $0<s<1/k$ the following bounds hold
\begin{equation}
    \begin{split}
        \phi(B(0,  2r)\cap T(u, s))&\leq  \limsup_{i\to\infty}\phi\bigl(B(0,  2r)\cap T(u_i,s+\dd(u,u_i))\bigr)\\
        &\overset{\eqref{eq:num1uno}}{\leq} 2(1-\varepsilon)^4\mathcal{C}^h\bigl(P(B(0,1))\bigr)(s/ (2r))^h\phi(B(0, 2r)).    \nonumber
    \end{split}
\end{equation}
This proves that $\widetilde{A}_k$ is compact, so $\widetilde{A}$ is an $F_\sigma$-set and in particular Borel. 

Notice that, since $r_1<r$, by a compactness argument one finds that there exists $\widetilde s\coloneqq \widetilde s(r_1,r)$ such that whenever $u\in P(B(0,r_1))$, then $P(B(u,\widetilde s))\subseteq P(B(0,r))$. The family
$$\mathcal{B}\coloneqq\bigl\{P(B(u,s)):u\in A\setminus \widetilde A,\,\,\text{and }s\leq \widetilde s\text{ does not satisfy \eqref{eq:num1uno}}\bigr\}$$ is a fine cover of $A\setminus \widetilde A$ by the very definition of $\widetilde A$. Thus \cite[2.8.17]{Federer1996GeometricTheory} with a routine argument
implies that $\mathcal{B}$ is a $\mathcal{S}^h\llcorner(A\setminus\widetilde A)$-Vitali relation (\cite[2.8.16]{Federer1996GeometricTheory}). Therefore, the set $A\setminus \widetilde{A}$ can be covered $\mathcal{S}^h$-almost all by a sequence of disjointed projected balls $\{P(B(u_k,s_k))\}_{k\in\N}$ such that $u_k\in A\setminus \widetilde{A}$ and
\begin{equation}\label{eq:propMM_aux3}
    \phi(B(0, 2r)\cap T(u_k, s_k))> 2(1-\varepsilon)^4\mathcal{C}^h\bigl(P(B(0,1))\bigr)(s_k/ (2r))^h\phi(B(0, 2r)),
\end{equation}
for every $k\in\N$. Note that since by definition $T(u_k, s_k)= P^{-1}\bigl(P(B(u_k,s_k))\bigr)$, we get that $\{T(u_k, s_k)\}_{k\in\mathbb N}$ is a disjointed family of cylinders. Moreover, from the very definition of $\widetilde s$, since $u_k\in P(B(0,r_1))$ and $s_k\leq \widetilde s$, we have that $P(B(u_k,s_k))\subseteq P(B(0,r))$. This implies that
\begin{equation}\label{eq:propMM_aux2}
\begin{split}
    \phi(B(0, r))=\phi(T(0,r)\cap B(0, 2r))&\geq\sum_{k\in\mathbb N} \phi(B(0, 2r)\cap T(u_k, s_k))\\
    &\overset{\eqref{eq:propMM_aux3}}{>} 2(1-\varepsilon)^4\mathcal{C}^h\bigl(P(B(0,1))\bigr)2^{-h}r^{-h}\phi(B(0,2r))\sum_{k\in\N} s^h_k.
\end{split}
\end{equation}

Therefore, we have
\begin{equation}
    \begin{split}
\mathcal{C}^h(A\setminus \widetilde{A})&=\sum_{k\in\N}\mathcal{C}^h\bigl(P(B(u_k,s_k))\bigr)\leq \mathcal{C}^h\bigl(P(B(0,1))\bigr)\sum_{k\in\N} s^h_k\\
&\overset{\eqref{eq:propMM_aux2}}{<} 2^{-1}(1-\varepsilon)^{-4}r^h \leq 2^{-1}(1-\varepsilon)^{-5}\mathcal{C}^h\bigl(P(B(0,1))\bigr)r^h\\
&\leq \frac{27}{50}\mathcal{C}^h\bigl(P(B(0,1))\bigr)r^h.
\nonumber
    \end{split}
\end{equation}

Furthermore, from the previous inequality and item (iv) we deduce that
\begin{equation}
\begin{split}
    \mathcal{C}^h(\widetilde{A})&=\mathcal{C}^h\bigl(P( B(0,r))\bigr)-\mathcal{C}^h\bigl(P(B(0,r))\setminus A\bigr)-\mathcal{C}^h(A\setminus\widetilde{A})\\
    &>\mathcal{C}^h\bigl(P(B(0,1))\bigr)r^h -2^{3h}5^{h+3}\varepsilon\, \mathcal{C}^h\bigl(P(B(0,1))\bigr)r^h-\frac{27}{50}\mathcal{C}^h\bigl(P(B(0,1))\bigr)r^h\\
    &\geq\Bigl(1-\frac{1}{25}-\frac{27}{50}\Bigr)\mathcal{C}^h\bigl(P(B(0,1))\bigr)r^h>\frac{2}{5}\mathcal{C}^h\bigl(P(B(0,1))\bigr)r^h.
\end{split}
    \nonumber
\end{equation}

Since $\widetilde{A}$ is measurable, we can find a compact set $\hat{A}\subseteq \widetilde A$ of diameter smaller that $\delta(1+h/\varepsilon)^{-1}$ and $\delta\in (0,\varepsilon r/h)$ such that $\mathcal{C}^h(\hat{A})>0$ and \eqref{eq:num1uno} holds for any $u\in \hat A$ and $s\in (0,\delta)$.
This can be done by taking an interior approximation with compact sets of $\widetilde A$.

Thanks to item (v) we know that
\begin{equation}
    \hat{A}\subseteq A= P(E\cap P^{-1}(A)),
    \label{eq:dundee}
\end{equation}
and thus for any $u\in \hat{A}$ we can find $x\in E$ such that $P(x)=u$.

Since $P$ restricted to $E\cap P^{-1}(A)$ is surjective on $\hat{A}$ as remarked in \eqref{eq:dundee}, thanks to the axiom of choice there exists a function $f\colon\hat{A}\to E\cap P^{-1}(A)$ such that $P(f(u))=u$. We claim that we have
\begin{equation}
    \lVert f(y)-f(x)\rVert\leq2h\lVert y-x\rVert/\varepsilon \qquad \text{ for every }x,y\in \hat{A}.
\label{eq:::num16}
\end{equation}

In order to prove the latter claim, assume by contradiction that there exist $x,y\in \hat{A}$ such that 
\begin{equation}
    \lVert f(y)-f(x)\rVert>2h\lVert y-x\rVert/\varepsilon.
\label{eq:::num16par}
\end{equation}
The assumption \eqref{eq:::num16par} implies in particular that $B(f(x),h\lVert y-x\rVert/\varepsilon)\cap B(f(y),h\lVert y-x\rVert/\varepsilon)=\emptyset.$
Consequently
\begin{equation}
    \begin{split}
        \phi\Bigl(B(0,2r)\cap& P^{-1}\Bigl(B\Bigl(x,\Bigl(1+\frac{h}{\varepsilon}\Bigr)\lVert x-y\rVert\Bigr)\cap W\Bigr)\Bigr)
        \\
        &\geq 2(1-\varepsilon)\mathcal{C}^h\bigl(P(B(0,1))\bigr)\Bigl(\Bigl(1+\frac{h}{\varepsilon}\Bigr)\frac{\lVert x-y\rVert}{2r}\Bigr)^h\phi(B(0,2r))\\
        &>2(1-\varepsilon)^2\mathcal{C}^h\bigl(P(B(0,1))\bigr)\Bigl(\Bigl(1+\frac{h}{\varepsilon}\Bigr)\frac{\lVert x-y\rVert}{2r}\Bigr)^h\phi(B(0,2r)).
    \end{split}
\end{equation}
which contradicts \eqref{eq:num1uno}.

Let us finally observe that $x,y\in \hat{A}$ and $\mathrm{diam}(\hat{A})<\delta(1+h/\varepsilon)^{-1}$ imply $(1+h/\varepsilon)\lVert x-y\rVert<\delta$, hence $f$ is a Lipschitz map of $\hat{A}$ onto $f(\hat{A})$, which concludes the proof of the proposition.
\end{proof}

\vv

It is possible to repeat the arguments of \cite[Section 3.1]{antonelli2022rectifiable} and prove the following result.

\begin{theorem}\label{struct:strat}
Assume that $\phi$ is a Radon measure on $(\mathbb{P}^{n},\dd)$ such that:
\begin{itemize}
    \item[(i)] There exists $h\in\{0,\ldots,n+2\}$ such that $0<\Theta^{h}_*(\phi,x)\leq \Theta^{h,*}(\phi,x)<\infty$ for $\phi$-almost every $x\in \mathbb{P}^n$.
    \item[(ii)] $\Tan_h(\phi,x)\subseteq \mathfrak{M}(h)$  for $\phi$-almost every $x\in \mathbb{P}^n$.
\end{itemize}

Then, for $\phi$-almost every $x\in \mathbb{P}^n$ the set $\mathfrak{s}(\mathrm{Tan}_h(\phi,x))\subseteq \mathfrak{S}(h)$ is a singleton.
In addition, if for every $x\in \mathbb{P}^n$ we define
\begin{equation}
    \begin{split}
        \mathfrak{s}(\phi,x)\coloneqq\begin{cases}
        \mathfrak{s} &\text{if }\mathrm{Tan}_h(\phi,x)\subseteq \mathfrak{M}(h)\text{ and }\mathfrak{s}(\mathrm{Tan}_h(\phi,x))\text{ is the singleton }\{\mathfrak{s}\},\\
        0 &\text{otherwise},
        \end{cases}
        \nonumber
    \end{split}
\end{equation}
then the map $\mathfrak{s}(\phi,\cdot)$ is well-defined, $\phi$-measurable, and non-zero $\phi$-almost everywhere.
\end{theorem}

\vv

All the results proved in this section finally allow us to infer regularity of a measure from the flatness of tangents, and prove the parabolic Marstrand-Mattila rectifiability criterion.

\begin{proof}[Proof of Theorem \ref{thm:MMconormale:intro}]
\cref{struct:strat} implies that for $\phi$-almost every $x\in \mathbb{P}^n$ the elements of $\Tan_h(\phi,x)$ all share the same stratification vector and that
$\mathscr{T}_\mathfrak{s}\coloneqq\{x\in K: \mathfrak{s}(\phi,x)=\mathfrak{s}\}$ is a  $\phi$-measurable set. Thus, if we prove that for any $\mathfrak{s}\in \mathfrak{S}(h)$ there exists a \emph{differentiable} Lipschitz function as in the statement of the theorem whose image has positive $\phi\llcorner \mathscr{T}_\mathfrak{s}$-measure, the theorem is proved by the locality of tangents (see \cref{tuttitg}) and since the sets $\mathscr{T}_\mathfrak{s}$ cover $\phi$-almost all $K$. Thus, we can assume without loss of generality that there exists  $\mathfrak{s}\in\mathfrak{S}(h)$ such that for $\phi$-almost every $x\in K$ we have $\mathfrak{s}(\phi,x)=\mathfrak{s}$.

Let $\widetilde\varepsilon\leq 2^{-3h}5^{-10(h+5)}$ and $\widetilde\mu\coloneqq2^{-7-5h}h^{-3}\widetilde\varepsilon^2$. \textit{We now check that we can find a compact subset $E$ of $K$ such that the hypotheses of Proposition \ref{prop::5.2} are satisfied with $\tilde{\varepsilon},\tilde{\mu}$.} More precisely we prove that there are $\vartheta,\gamma\in\N$, a $\phi$-positive compact subset $E$ of $E(\vartheta,\gamma)$, and a point $z\in E$ such that:
\begin{itemize}
    \item[(\hypertarget{defrho}{1})] There exists $\rho_z>0$ for which $\phi(B(z,2\rho)\setminus E)\leq \tilde\mu^{h+1}\phi(B(z,2\rho))$ for any $0<\rho<\rho_z$.
    \item[(2)] There exists $r_0\in(0,2^{-3h}5^{-10(h+5)}\gamma^{-1}]$ such that for any $w\in E$ and any $0<\rho\leq r_0$ we can find $V_{w,\rho}\in \Gr^{\mathfrak s}(h)$ such that:
    \begin{enumerate}
    \item[(\hypertarget{pointii1}{2.A})] $F_{w,8\rho}\bigl(\phi,c\,\mathcal{C}^h\llcorner (w+V_{w,\rho})\bigr)\leq (2^{-3}\vartheta^{-1}\tilde\mu)^{(h+3)}\cdot(8\rho)^{h+1}$ for some $c>0$.
    \item[(\hypertarget{pointii2}{2.B})] Whenever $y\in B(w,2\rho)\cap (w+V_{w,\rho})$ and $t\in [\tilde\mu \rho, 2\rho]$ we have
$$\phi(B(y,t))\geq (1-\varepsilon)(t/(2\rho))^h\phi(B(w,2\rho)).$$
\end{enumerate}
 \item[(\hypertarget{defrho(z)}{3})] There exists an infinitesimal sequence $\{{\rho_i(z)}\}_{i\in\N}\subseteq (0,\min\{r_0,\rho_z\}]$ such that for any $i\in\N$, any $w\in E$, and any $\rho\in(0,{\rho_i(z)}]$ we have
    \[\phi(B(w,2\rho))\geq (1-\varepsilon)(\rho/{\rho_i(z)})^h\phi(B(z,2\rho_i(z))).\]
    \end{itemize}

For any $a,b>0$ we define ${F}(a,b)$ to be the set of those points $x\in K$ for which
$$
    \phi(B(x,r))\geq br^h,\qquad\text{for any }r\in (0,a).
$$
One easily proves, with the same argument used in \cite[Proposition 1.14]{MerloMM}, that the sets ${F}(a,b)$ are compact. Hence, 
\[
    \widetilde{{F}}(a,b)\coloneqq\bigcap_{p=1}^\infty {F}(a,(1-\varepsilon)b)\setminus F(a/p,b),
\]
are Borel sets. Item (1) yields that $\mathbb{P}^n$ can be covered $\phi$-almost all by countably many sets $\widetilde{{F}}(a,b)$. In particular, thanks to \cref{prop::E} we can find $a,b\in \R$ and $\vartheta,\gamma\in\N$ such that $\phi(\widetilde{{F}}(a,b)\cap E(\vartheta,\gamma))>0$. Since $\widetilde{{F}}(a,b)\cap E(\vartheta,\gamma)$ is measurable, there exists a compact subset of  $\widetilde{{F}}(a,b)\cap E(\vartheta,\gamma)$ with $\phi$-positive measure and we denote it with $F$.

Observe that $\mathrm{Tan}_h(\phi,x)\subseteq \mathfrak{M}^\mathfrak{s}(h)$ for $\phi$-almost every $x\in F$ and that the functions $x\mapsto d_{x,kr}(\phi,\mathfrak M^\mathfrak{s}(h))$ are continuous in $x$ for every $k,r>0$ by \cref{rem:dxrContinuous}. Hence, \cref{tuttitg} together with Severini-Egoroff Theorem imply the existence of a $\phi$-positive compact subset ${E}$ of $F$ and $r_0\leq5^{-10(h+5)}\gamma^{-1}$ such that
 \begin{equation}\label{eq:::num4}
 \begin{split}
     d_{x,4\rho}\bigl(\phi,\mathfrak{M}(h,\Gr^{\mathfrak s}(h))\bigr)\leq (4^{-1}\vartheta^{-1}\tilde\mu)^{(h+4)}\quad \text{ for any }x\in E\text{ and any }0<\rho\leq r_0.
 \end{split}
       \end{equation}

Fix $z$ to be a density point of $E$ with respect to $\phi$, and let us show that $E$ and $z$ satisfy the requirements of the proposition. First observe that, by construction, $\phi(E)>0$ and $E\subseteq E(\vartheta,\gamma)$. Second, since $z$ is a density point of $E$, item (\hyperlink{hp1}{1}) follows if we choose $\rho_z$ small enough. Moreover, the bound \eqref{eq:::num4} directly implies item (2.A). 
 
Let us prove the remaining items. Since $E\subseteq E(\vartheta,\gamma)$, $r_0<\gamma/32$ and $2^{-3}\vartheta^{-1}\tilde\mu\leq 2^{-10(h+1)}\vartheta$, choosing $\sigma=2^{-3}\vartheta^{-1}\tilde\mu$ and $t=8 \rho$ in \cref{prop::4.4(4)} we have that there exists $V_{w,\rho}\in \Gr^{\mathfrak s}(h)$ such that
\begin{equation}\label{eq:thm413aux1}
    \phi\bigl(B(y,r)\cap B(w+V_{w,\rho},2^{-3}\vartheta^{-2}\tilde\mu^2\rho)\bigr)\geq (1-2^{10(h+1)}2^{-3}\tilde\mu)(r/s)^h\phi(B(v,s)),
\end{equation}
whenever $y,v\in B(w,4  \rho)\cap   (w+V_{w,\rho})$ and $\vartheta^{-1}\tilde\mu\rho\leq r,s\leq 4  \rho$. Since $$2^{10(h+1)}2^{-3} \tilde\mu\leq \varepsilon,$$ the choices $s= 2 \rho$ and $v=w$ in \eqref{eq:thm413aux1} imply that
\[
    \phi(B(y,r))\geq (1-\varepsilon)(r/ (2\rho))^h\phi(B(w, 2 \rho)),
\]
for any $\tilde\mu \rho\leq r\leq  2 \rho$ and any $y\in B(w,  2\rho)\cap (w+ V_{w,\rho})$, which proves item (2.B). 

In order to verify item (3), note that since $z\in E\subseteq \widetilde{F}(a,b)$ there is an infinitesimal sequence $\{\rho_i(z)\}_{i\in\N}$ such that
\begin{equation}
    \frac{\phi\bigl(B(z, 2\rho_i(z))\bigr)}{(2\rho_i(z))^h}\leq b.
    \label{eq:Eab1}
\end{equation}
Moreover, for any $w\in E$, and any $0<\rho<a$ we have
\begin{equation}
   b\leq \frac{1}{1-\varepsilon}\frac{\phi(B(w, 2\rho))}{(2\rho)^h}.
   \label{eq:Eab2}
\end{equation}

We gather \eqref{eq:Eab1} and \eqref{eq:Eab2}, and finally infer that for any $i\in\N$, any $w\in E$ and any $\rho\in(0,a)$ we have
$$
    \frac{\phi(B(z, 2\rho_i(z)))}{(2\rho_i(z))^h}\leq \frac{1}{1-\varepsilon}\frac{\phi(B(w, 2\rho))}{\rho^h},
$$
which concludes the proof of item (3).

\medskip

Let $E\subseteq K$ be the compact set and $z\in E$ the point constructed above. Furthermore, let $\widetilde\varepsilon\leq\varepsilon\leq 2^{-3h}5^{-h-5}$ and $\mu\coloneqq2^{-7-5h}h^{-3}\varepsilon^2$ be such that $(1-\widetilde\varepsilon)^2\geq (1-\varepsilon)$. We define
$$
r\coloneqq\rho_1(z),\qquad \text{and}\qquad r_1\coloneqq\Bigl(1-\frac{\varepsilon}{h}\Bigr)r,$$
where $\rho_z$ was introduced in (\hyperlink{defrho}{1}).

Let us check that the compact set $E\cap B(z,2r_1)$ satisfies the hypothesis of \cref{prop::5.2} with respect to the choices $\varepsilon,\mu,r$. First of all, since $r<\rho_1(z)$, item (\hypertarget{defrho}{1}) and the inequality $\widetilde\mu\leq \mu$ imply that \eqref{eq:densityE} holds. Secondly, item (2.\hyperlink{pointii2}{B}) and the bound $r\leq r_0$ yield that for any $w\in E$ and any $0<\rho<2r$ there exists $V_{w,\rho}\in \Gr^\mathfrak{s}(h)$ such that whenever $y\in B(w,2r)\cap (w+V_{w,\rho})$ and $t\in [\mu \rho, 2\rho]$ we have
\[
    \phi(B(y,t))\geq (1-\widetilde\varepsilon)\Bigl(\frac{t}{2\rho}\Bigr)^h\phi(B(w,2\rho)).
\]

Furthermore, since $r<\rho_1(z)$, thanks to item (\hyperlink{defrho(z)}{3}), we finally infer that for any $w\in E$ and any $0<\rho<r$ we have
\begin{equation}
    \begin{split}
\phi(B(y,t))&\geq (1-\widetilde\varepsilon)\Bigl(\frac{t}{2\rho}\Bigr)^h\phi(B(w,\rho))\geq (1-\widetilde\varepsilon)^2\Bigl(\frac{t}{r}\Bigr)^h\phi(B(z,2r))\geq (1-\varepsilon)\Bigl(\frac{t}{2r}\Bigr)^h\phi(B(z,2r)),  
    \end{split}
\end{equation}
whenever $y\in B(w,2r)\cap  ( w+V_{w,\rho})$ and $t\in [\mu \rho, 2\rho]$. 
Hence, we have shown that the hypotheses of \cref{prop::5.2} are satisfied by $z$ and $E\cap B(z,2r_1)$ with the choices of $r,r_1,\varepsilon,\mu$ as above. 

\medskip

{We are left with showing that for $\phi$-almost every $x\in \mathbb{P}^n$ there exists $V(x)\in \Gr^\mathfrak{s}(h)$ such that}
\begin{equation}\label{eq:thmmmfincl1}
    \Tan_h(\phi,x)\subseteq \{\lambda\mathcal{H}^h\llcorner V(x):\lambda>0\}.
\end{equation}
The previous paragraph and Proposition \ref{prop::5.2} imply that there exist $W\in\Gr^\mathfrak{s}(h)$, a compact set $K\subseteq W$, and a Lipschitz map $f\colon W\to W^\perp$ such that $\phi(f(K))>0$. Note that thanks to Proposition \ref{tuttitg} we know that $\phi\llcorner f(K)$ still satisfies items (\hyperlink{hp1}{i}) and (\hyperlink{hp2}{ii}) in the statement of Theorem \ref{thm:MMconormale:intro} with the same $h$ as $\phi$. 
This concludes the proof of \eqref{eq:thmmmfincl1} thanks to Theorem \ref{mm0.0}.
\end{proof}
\vv

We can now prove Theorem \ref{theorem:main_theorem_Preiss} in the case of $\mathbb P^1$ endowed with the Koranyi metric $d$ (see \eqref{eq:Korany_metric}).

\begin{proposition}\label{preiss1}
Suppose that $\phi$ is a Radon measure on $(\mathbb{P}^1,d)$ such that 
$$0<\lim_{r\to 0}\frac{\phi(B_d(x,r))}{r^2}<\infty.$$

Then $\phi$ is absolutely continuous with respect to $\mathcal{H}^2$ and it is $\mathscr{P}_2$-rectifiable. More specifically, there are countably many compact sets $K_i\subseteq \mathcal{V}$ and Lipschitz maps $g_i\colon K_i\to \R$ with the Rademacher property such that \[\phi\Bigl(\mathbb{P}^1\setminus \bigcup_{i\in\N} g_i(K_i)\Bigr)=0.\]
\end{proposition}

\begin{proof}
This is an immediate consequence of Propositions \ref{uniformup}, \ref{uniformmeasuresinP1} and Theorem \ref{thm:MMconormale:intro}. 
\end{proof}

\vv

\section{Moments and their expansion}\label{section:moments_expansion}

Until the end of Section \ref{section:BWGL} we understand that $\mathbb P^n$ is endowed with the Koranyi distance $d$ and we denote as $h$ a number in $\{1,\ldots,n+2\}$.

	A natural alternative to the scalar product on the parabolic space is the \textit{polarization} of the norm $\|\cdot\|$, namely
	\[
		V(u,z)\coloneqq \frac{\|z\|^4 + \|u\|^4 - \|z-u\|^4}{2}, \qquad \text{ for } z,u\in \mathbb{P}^n. 
	\]
	
	\begin{proposition}\label{prop6}
		The function $V(u,z)$ can be decomposed as 
		\begin{equation}\label{eq:decompos_polarization}
			2V(u,z)=L(u,z)+Q(u,z)+T(u,z),
		\end{equation}
		where:
		\begin{itemize}
			\item[(i)] $L(u,z)\coloneqq4\lvert z_H\rvert^2\langle u_H,z_H\rangle$.
			\item[(ii)] $Q(u,z)\coloneqq2z_Tu_T-4\langle z_H,u_H\rangle^2-2\lvert z_H\rvert^2\lvert u_H\rvert^2$.
			\item[(iii)] $T(u,z)\coloneqq4\lvert u_H\rvert^2\langle u_H,z_H\rangle$.
		\end{itemize}
	\end{proposition}

	\begin{proof}By a direct computation of $V$ we get
		\begin{equation*}
			\begin{split}
				2V(u,z)&=\lvert z_H\rvert^4+z_T^2+\lvert u_H\rvert^4+u_T^2-\lvert z_H-u_H\rvert^4-( z_T-u_T)^2\\
				&= \lvert z_H\rvert^4+\lvert u_H\rvert^4-\bigl(\lvert z_H\rvert^2-2\langle z_H,u_H\rangle+\lvert u_H\rvert^2\bigr)^2+2z_Tu_T\\
				&=\lvert z_H\rvert^4+\lvert u_H\rvert^4-\bigl(\lvert z_H\rvert^4+4\langle z_H,u_H\rangle^2+\lvert u_H\rvert^4+2\lvert z_H\rvert^2\lvert u_H\rvert^2\\
				&\qquad -4\lvert z_H\rvert^2\langle z_H,u_H\rangle-4\lvert u_H\rvert^2\langle z_H,u_H\rangle\bigr)+2z_Tu_T\\
				&=-4\langle z_H,u_H\rangle^2-2\lvert z_H\rvert^2\lvert u_H\rvert^2+4\lvert z_H\rvert^2\langle z_H,u_H\rangle+4\lvert u_H\rvert^2\langle z_H,u_H\rangle+2z_Tu_T,
			\end{split}
		\end{equation*}
		which proves \eqref{eq:decompos_polarization} after regrouping the summands.
	\end{proof}

\vv

\begin{remark}\label{rk1}
	The polynomials $L,Q$, and $T$ defined in Proposition \ref{prop6} are respectively $1,2,3$-$\delta_\lambda$-homogeneous in the first variable, i.e.
	\[
		L(\delta_\lambda(u),z)=\lambda L(u,z),\qquad Q(\delta_\lambda(u),z)=\lambda^2Q(u,z),\qquad T(\delta_\lambda(u),z)=\lambda^3T(u,z),
	\]
	and are respectively $3,2,1$-$\delta_\lambda$-homogeneous in the second entry, namely
	\[
		L(u,\delta_\lambda(z))=\lambda^3 L(u,z),\qquad Q(u,\delta_\lambda(z))=\lambda^2Q(u,z),\qquad T(u,\delta_\lambda(z))=\lambda T(u,z).
	\]
	Moreover, thanks to the definitions of $L$ and $T$ it is immediate to see that $L(z,u)=T(u,z)$.
\end{remark}

\vv

Cauchy-Schwarz inequality and the definitions of $L,Q$, and $T$ readily give us bound on those quantities. 

\begin{proposition}\label{prop7}
	For any $z,u\in\mathbb{P}^n$ the following estimates hold:
	\begin{itemize}
		\item[(i)] $\lvert L(u,z)\rvert\leq4\lVert u\rVert \lVert z\rVert^3$.
		\item[(ii)] $\lvert Q(u,z)\rvert\leq 8\lVert z\rVert^2\lVert u\rVert^2$.
		\item[(iii)] $\lvert T(u,z)\rvert\leq4\lVert z\rVert \lVert u\rVert^3$.
	\end{itemize}
\end{proposition}

One of the most important properties of the polarization function $V$ is a Cauchy-Schwartz-type inequality, that will turn out to be fundamental for our computations.

\begin{proposition}[Cauchy-Schwarz inequality for $V$]\label{prop4}
	For any $u,z\in\mathbb{P}^n$ we have
	\begin{equation}
		\lvert V(u,z)\rvert\leq 2\lVert u\rVert \lVert z\rVert(\lVert u\rVert+\lVert z\rVert)^2\nonumber.
	\end{equation}
\end{proposition}

\begin{proof}
The proof can be obtained following verbatim \cite[Proposition 3.4]{MerloG1cod}.
\end{proof}

The following definition extends from the Euclidean spaces to $\mathbb P^n$ the notion of moment of a uniform measure given by Preiss in \cite{Preiss1987GeometryDensities}.

\begin{definition} [Preiss's moments]\label{defimom}

	Let $\mu\in \mathcal U(h)$. For $k\in\N$, $s>0$, and $u_1,\ldots, u_k \in\mathbb{P}^n$, we define
	\begin{equation}
		b_{k,s}^\mu(u_1,\ldots,u_k)\coloneqq\frac{s^{k+\frac{h}{4}}}{k!C(h)}\int \prod_{i=1}^k 2V(u_i,z)e^{-s\lVert z\rVert^4}d\mu(z)\nonumber,
	\end{equation}
	where $C(h)\coloneqq\Gamma\left(\frac{h}{4}+1\right)$ and $b_{0,s}^\mu(u)\coloneqq 1$. Moreover, if $u_1=\ldots=u_k=u$, we simplify the notation to
	$$b_{k,s}^\mu(u)\coloneqq b_{k,s}^\mu(u,\ldots,u).$$
\end{definition}

The Cauchy-Schwartz inequality for $V$ allows us to obtain the following estimates.

\begin{proposition}\label{prop2}
	For any $\mu\in\mathcal{U}(h)$, any $k\in\N$, any $s>0$ and any $u\in\mathbb{P}^n$, we have
	\begin{equation}
		\lvert b_{k,s}^\mu(u)\rvert\leq 16^{k}\frac{\bigl(\lVert u\rVert s^\frac{1}{4}\bigr)^k}{k!}\frac{\Gamma(\frac{h+3k}{4})}{\Gamma\left(\frac{h}{4}\right)}\bigl((\lVert u\rVert s^\frac{1}{4})^{2k}+1\bigr).
		\nonumber
	\end{equation}
\end{proposition}

\begin{proof}
The proof can be obtained following verbatim \cite[Proposition 3.5]{MerloG1cod}.
\end{proof}

\vv

\begin{definition}\label{calpha}
	For any $\mu\in\mathcal{U}(h)$, any $\alpha\in\N^3\setminus\{(0,0,0)\}$, and any $s>0$ we define the functions $c_{\alpha,s}^\mu \colon\bigotimes_{i=0}^{\lvert \alpha\rvert}\mathbb{P}^n\to \R$ as
	\begin{equation} 
		c_{\alpha,s}^\mu (u)\coloneqq\frac{1}{\alpha_1!\alpha_2!\alpha_3!}\frac{s^{\lvert \alpha\rvert+\frac{h}{4}}}{C(h)}\int  L(u,z)^{\alpha_1} Q(u,z)^{\alpha_2} T(u,z)^{\alpha_3} e^{-s\lVert z\rVert^4}d\mu(z),
		\nonumber
	\end{equation}
	where $\lvert \alpha\rvert\coloneqq\alpha_1+\alpha_2+\alpha_3$. Moreover, for any $\ell\in\N$, we set
	$$A(\ell)\coloneqq\bigl\{\alpha\in\N^3\setminus \{(0,0,0)\}: \alpha_1+2\alpha_2+3\alpha_3\leq \ell\bigr\}.$$
\end{definition}

\vv

The moments $b_{k,s}^\mu$ can be expressed by means of the functions $c_{\alpha,s}^\mu $ defined above. In particular, we have

\begin{equation}
	\begin{split}
		 b_{k,s}^\mu(u)&=\frac{s^{k+\frac{h}{4}}}{k!C(h)}\int (2V(u,z))^k e^{-s\lVert z\rVert^4}d\mu(z)\\
		&\overset{\eqref{eq:decompos_polarization}}{=}\frac{s^{k+\frac{h}{4}}}{k!C(h)}\int (L(u,z)+Q(u,z)+T(u,z))^k e^{-s\lVert z\rVert^4}\, d\mu(z)\\
		&=\frac{s^{k+\frac{h}{4}}}{k!C(h)}\int \sum_{\lvert\alpha\rvert=k}\frac{k!}{\alpha_1!\alpha_2!\alpha_3!}L(u,z)^{\alpha_1} Q(u,z)^{\alpha_2} T(u,z)^{\alpha_3} e^{-s\lVert z\rVert^4}\, d\mu(z)\\
		&=\sum_{\lvert\alpha\rvert=k} c_{\alpha,s}^\mu (u).
		\label{splitter}
	\end{split}
\end{equation}

\begin{proposition}\label{prop9}
	For any $\mu\in\mathcal{U}(h)$ and any $\alpha\in\N^3\setminus\{(0,0,0)\}$, there exists a constant $D(\alpha)>0$ such that for any $s>0$ and any $u\in\mathbb{P}^n$ we have
	\begin{equation}
		\lvert c_{\alpha,s}^\mu (u)\rvert\leq D(\alpha)(s^{1/4}\lVert u\rVert)^{\alpha_1+2\alpha_2+3\alpha_3}.
		\nonumber
	\end{equation}
\end{proposition}

\begin{proof}
The proof can be obtained following verbatim \cite[Proposition 3.6]{MerloG1cod}.
\end{proof}

\vv

\begin{proposition}\label{coni1}
	Assume that $\mu\in\mathcal{U}(h)$ is invariant under dilations, i.e. for any $\lambda>0$ we have $\lambda^{-h} T_{0,\lambda}\mu=\mu$,
	where $T_{0,\lambda}\mu$ was introduced in Definition \ref{def:TangentMeasure}. Then, for any $\alpha\in\N^3\setminus\{(0,0,0)\}$ and any $s>0$ we have
	\[
		c_{\alpha,s}^\mu =s^{\frac{\alpha_1+2\alpha_2+3\alpha_3}{4}}c^\mu_{\alpha,1}.
	\]
\end{proposition}

\begin{proof}
See \cite[Proposition 3.7]{MerloG1cod}.
\end{proof}

\vv

\subsection{Expansion formulas for moments}\label{subsection:expansion_moments_candidate_quadric}
In this subsection we present the expansion formula \eqref{eq1} for the moments of uniform measures. Moreover, in Proposition \ref{prop8} we start to flesh out the complex algebra of the inequality \eqref{eq1} in order to build the desired quadric containing $\supp(\mu)$.
%
%

We omit the proof of the following proposition, as it follows closely that of \cite[Proposition 3.9]{MerloG1cod}, which is in turn based on its Euclidean analogue (see \cite[Section 3.4]{Preiss1987GeometryDensities} or \cite[Lemma 7.6]{DeLellis2008RectifiableMeasures}).

\begin{proposition}[Expansion formula]\label{expanzione}
	There exists a constant $G(h)>0$ such that for any $\mu\in\mathcal{U}(h)$, any $s>0$, any $q\in\N$, and any $u\in\supp(\mu)$ we have
	\begin{equation}
		\bigg\lvert \sum_{k=1}^{4q}b_{k,s}^\mu(u)-\sum_{k=1}^q\frac{s^k\lVert u\rVert^{4k}}{k!}\bigg\rvert\leq G(h)(s\lVert u\rVert^4)^{q+\frac{1}{4}}\bigl(2+(s\lVert u\rVert^4)^{2q}\bigr)\label{eq1}.
	\end{equation} 
\end{proposition}
\vv

The identity \eqref{splitter} together with Proposition \ref{expanzione} for $q=1$  imply that
\begin{equation}
	\bigg\lvert \sum_{k=1}^{4}\sum_{\lvert\alpha\rvert=k} c_{\alpha,s}^\mu (u)-s\lVert u\rVert^{4}\bigg\rvert\leq G(h)(s\lVert u\rVert^4)^{\frac{5}{4}}\bigl(2+(s\lVert u\rVert^4)^{2}\bigr).
	\label{numeroo14}
\end{equation}

In the next proposition we reduce the complexity of the algebraic expression in the right-hand side of \eqref{numeroo14}.

\begin{proposition}	\label{prop8}
	For any $\mu\in \mathcal{U}(h)$, any $s>0$, and any $u\in\supp(\mu)$ we have
	\begin{equation}
		\Big\lvert \sum_{\alpha\in A(4)}c_{\alpha,s}^\mu (u)-s\lVert u\rVert^{4}\Big\rvert\leq (s\lVert u\rVert^4)^{\frac{5}{4}} B(s^\frac{1}{4}\lVert u\rVert),
		\label{eqquart}
	\end{equation}
	where $B(\cdot)$ is a suitable polynomial whereas $c_{\alpha,s}^\mu (\cdot)$ and $A(4)$ where defined in Definition \ref{calpha}.
\end{proposition}

\begin{proof}
The proof can be obtained following verbatim \cite[Proposition 3.10]{MerloG1cod}.
\end{proof}

\vv

\begin{corollary}\label{quartic}
If $\mu\in \mathcal{U}(h)$ is dilation-invariant, then for any $u\in\supp(\mu)$ we have
$$ c_{(4,0,0),1}^\mu(u)+c_{(2,1,0),1}^\mu(u)+c_{(0,2,0),1}^\mu(u)+c_{(1,0,1),1}^\mu(u)-\lVert u\rVert^4=0.$$
\end{corollary}

\begin{proof}
This immediately follows from Propositions \ref{coni1}
 and \ref{prop8}.
 \end{proof}

\vv

\subsection{Construction of the candidate quadric containing the support}
\label{CDDQQ}

Throughout this section, we fix $\mu\in \mathcal U(h)$.

	\begin{definition} 
		\label{definitioncurve}
		For $s\in(0,\infty)$ we let:
		\begin{itemize}
			\item [(i)] The \emph{horizontal barycenter} of the measure $\mu$ at time $s$ to be the vector in $\R^{n}$ $$b(s)=b_\mu(s)\coloneqq\frac{4s^{\frac{1}{2}+\frac{h}{4}}}{C(h)}\int \lvert   z_H\rvert^2   z_H e^{-s\lVert z\rVert^4}d\mu(z).$$ 
			\item[(ii)] The symmetric matrix $\mathcal{Q}(s)$ associated to the measure $\mu$ at time $s$ to be the element of $\text{Sym}(n)$
			\begin{equation}
				\begin{split}
		\mathcal{Q}(s)=\mathcal{Q}_\mu(s)\coloneqq&\frac{8 s^{\frac{3}{2}+\frac{h}{4}}}{C(h)}\int \lvert z_H\rvert^4 z_H\otimes z_H
					e^{-s\lVert z\rVert^4}\,d\mu(z)\\
			&\qquad-\frac{s^{\frac{1}{2}+\frac{h}{4}}}{C(h)}\int (4z_H\otimes z_H+2\lvert z_H\rvert^2\textrm{id}_{n}) e^{-s\lVert z\rVert^4}\, d\mu(z).\nonumber
				\end{split}
			\end{equation}
			\item [(iii)]
			The \emph{vertical barycenter} of the measure $\mu$ at time $s$ to be the real number
	$$\mathcal{T}(s)=\mathcal{T}_\mu(s)\coloneqq\frac{2s^{1+\frac{h}{4}}}{C(h)}\int z_T e^{-s\lVert z\rVert^4}\,d\mu(z).$$
		\end{itemize}
	\end{definition}

\vv

Now we use Proposition \ref{prop8} in order to simplify the algebra of inequality \eqref{eq1} and to prove the existence of a constant $C>0$ such that
\begin{equation}
	\Big\lvert\langle b(s), u_H\rangle+\langle \mathcal{Q}(s)u_H, u_H\rangle+\mathcal{T}(s)  u_T\Big\rvert\leq C s^\frac{1}{4}\lVert u\rVert^3.
	\nonumber
\end{equation}

Then, in the second half of this subsection, we prove that $b(\cdot)$, $\mathcal{Q}(\cdot)$, and $\mathcal{T}(\cdot)$ are bounded curves as $s$ goes to $0$ and therefore by compactness we can find $\overline{b}$, $\overline{\mathcal{Q}}$ and $\overline{\mathcal{T}}$ for which
\[
	\langle\overline{b},u_H\rangle+\langle u_H,\overline{\mathcal{Q}} u_H\rangle+\overline{\mathcal{T}}u_T=0 \qquad \text{for all }u\in \supp(\mu).
\]

\begin{proposition}\label{prop10}
	For any $s>0$ and any $u\in\supp(\mu)$ it holds
	\begin{equation}\label{eq:boundprop10}
		\Big\lvert\langle b(s), u_H\rangle+\langle \mathcal{Q}(s) u_H, u_H\rangle+\mathcal{T}(s)  u_T\Big\rvert\leq s^\frac{1}{4}\lVert u\rVert^3B^\prime\bigl(s^\frac{1}{4}\lVert u\rVert\bigr),
	\end{equation}
	where  $B^\prime(\cdot)$ is a suitable polynomial and $b(\cdot)$, $\mathcal{Q}(\cdot)$, and $\mathcal{T}(\cdot)$ were introduced in Definition \ref{definitioncurve}.
\end{proposition}

\begin{proof}
	Since $\lvert \alpha\rvert\leq \alpha_1+2\alpha_2+3\alpha_3$ for any $\alpha\in\N^3$, then
		\begin{equation}
		A(2)\supset \{\alpha\in\N^3\setminus\{(0,0,0)\}:1\leq\lvert \alpha\rvert\leq 4\}\eqqcolon\mathcal{A}(4).
		\label{numeroo15}
	\end{equation}
	Furthermore we remark that for any $\alpha\in \mathcal{A}(4)\setminus A(2)$, by definition, it holds $\alpha_1+2\alpha_2+3\alpha_3\geq 3$. Therefore, Proposition \ref{prop9} implies that
	\begin{equation}
		\begin{split}
			\sum_{\alpha\in \mathcal{A}(4)\setminus A(2)}\lvert c_{\alpha,s}^\mu (u)\rvert&\leq (s^\frac{1}{4}\lVert u\rVert)^3\sum_{\alpha\in \mathcal{A}(4)\setminus A(2)}D(\alpha)(s^\frac{1}{4}\lVert u\rVert)^{\alpha_1+2\alpha_2+3\alpha_3-3}\\
			&=(s^\frac{1}{4}\lVert u\rVert)^3B^{\prime\prime}(s^\frac{1}{4}\lVert u\rVert).
			\nonumber
		\end{split}
	\end{equation}
	where $B^{\prime\prime}(t)\coloneqq\sum_{\substack{\lvert\alpha\rvert\leq 4\\ \alpha\not\in A(2)}}D(\alpha)t^{\alpha_1+2\alpha_2+3\alpha_3-3}$. Hence, triangle inequality and Proposition \ref{prop8} give
	\begin{equation}
		\begin{split}
			\Big\lvert \sum_{\alpha\in A(2)}c_{\alpha,s}^\mu (u)\Big\rvert\leq& \Big\lvert \sum_{\alpha\in A(4)}c_{\alpha,s}^\mu (u)-s\lVert u\rVert^4\Big\rvert+s\lVert u\rVert^4+\sum_{\alpha\in \mathcal{A}(4)\setminus A(2)}\lvert c_{\alpha,s}^\mu (u)\rvert\\
			\leq&(s\lVert u\rVert^4)^{\frac{5}{4}} B(s^{1/4}\lVert u\rVert)+s\lVert u\rVert^4+(s^\frac{1}{4}\lVert u\rVert)^3B^{\prime\prime}(ss^{1/4}\lVert u\rVert),
		\end{split}
		\nonumber
	\end{equation}
	where $B^\prime(t)\coloneqq t^2 B(t)+t+B^{\prime\prime}(t)$. A simple computation yields \[A(2)=\{(1,0,0),(2,0,0),(0,1,0)\}\]  and thus to conclude the proof of the proposition we are left to show that
	\begin{equation}
		c_{(1,0,0),s}(u)+c_{(2,0,0),s}(u)+c_{(0,1,0),s}(u)=\sqrt{s}\bigl(\langle b(s),u_H\rangle+\langle u_H,\mathcal{Q}(s) u_H\rangle+\mathcal{T}(s)u_T\bigr),
		\label{numeroo16}
	\end{equation}
	where, for the sake of brevity, we omit the subscript $\mu$ in the terms on the left-hand side.
	We now compute the summands in the left-hand side of \eqref{numeroo16}. First observe that
	\begin{equation}\label{eq:c100}
		\begin{split}
			c_{(1,0,0),s}(u)&=\frac{s^{1+\frac{h}{4}}}{C(h)}\int L(u,z)e^{-s\lVert z\rVert^4}d\mu(z)\\
			&=s^\frac{1}{2}\Big\langle  u_H, \frac{s^{\frac{1}{2}+\frac{h}{4}}}{C(h)}\int 4\lvert   z_H\rvert^2   z_H e^{-s\lVert z\rVert^4}d\mu(z)\Big\rangle=s^\frac{1}{2}\langle  u_H. b(s)\rangle,
		\end{split}
	\end{equation}
	
	Secondly, an explicit computation of $c_{(2,0,0),s}(u)$ yields the first part of the quadric $\mathcal{Q}(s)$. More specifically, we have
	\begin{equation}\label{eq:c200}
		\begin{split}
			c_{(2,0,0),s}(u)&=\frac{1}{2}\frac{s^{2+\frac{h}{4}}}{C(h)}\int L(u,z)^2e^{-s\lVert z\rVert^4}\,d\mu(z)\\
			&=\frac{16 s^{2+\frac{h}{4}}}{2C(h)}\int \lvert z_H\rvert^4\langle z_H,u_H\rangle^2 e^{-s\lVert z\rVert^4}d\mu(z)
			=s^\frac{1}{2}\langle  u_H,\mathcal{Q}_1(s) u_H \rangle,
		\end{split}
	\end{equation}
	where
	\begin{equation}
		\begin{split}
			\mathcal{Q}_1(s)\coloneqq&\frac{8 s^{\frac{3}{2}+\frac{h}{4}}}{C(h)}\int \lvert z_H\rvert^4 z_H\otimes z_H
			e^{-s\lVert z\rVert^4}\,d\mu(z).
		\end{split}
		\label{eq:Q2}
	\end{equation}
	
	Finally, we see that $c_{(0,1,0),s}(u)$ involves the vertical barycenter and the second part of $\mathcal{Q}(s)$. Indeed,
	\begin{equation}\label{eq:c010}
		\begin{split}
			 c_{(0,1,0),s}(u)&=\frac{s^{1+\frac{h}{4}}}{C(h)}\int Q(u,z) e^{-s\lVert z\rVert^4}\,d\mu(z)\\
			&=-\frac{s^{1+\frac{h}{4}}}{C(h)}\int (4\langle    z_H, u_H\rangle^2+2\lvert z_H\rvert^2\lvert u_H\rvert^2) e^{-s\lVert z\rVert^4}\,d\mu(z)+\frac{s^{1+\frac{h}{4}}}{C(h)}\int 2z_Tu_T e^{-s\lVert z\rVert^4}\,d\mu(z)\\
			&=-s^\frac{1}{2}\langle \mathcal{Q}_2(s) [u_H],u_H\rangle+s^\frac{1}{2}\mathcal{T}(s)  u_T,
		\end{split}
	\end{equation}
	where
	\begin{equation}
		\begin{split}
			\mathcal{Q}_2(s)\coloneqq\frac{s^{\frac{1}{2}+\frac{h}{4}}}{C(h)}\int (4z_H\otimes z_H+2\lvert z_H\rvert^2\textrm{id}_{n}) e^{-s\lVert z\rVert^4}d\mu(z).
		\end{split}
		\label{eq:Q4}
	\end{equation}
	
	Observe that $\mathcal{Q}(s)=\mathcal{Q}_1(s)-\mathcal{Q}_2(s)$, which proves the claim \eqref{numeroo16}.
\end{proof}

\vv

We endow $\text{Sym}(n)$ with the operator norm $\lvert\cdot\rvert$.
\begin{proposition}\label{bddcurve}
	Both $\mathcal{Q}(s)$ and $\mathcal{T}(s)$ are bounded functions on $(0,\infty)$. To be precise:
	\begin{itemize}
		\item[(i)]
		There exists a constant $C_1>0$ such that $\sup_{s\in (0,\infty)}\lvert Q(s)\rvert\leq C_1$.
		\item [(ii)]
		There exists a constant $C_2>0$ such that $\sup_{s\in(0,\infty)}\lvert \mathcal{T}(s)\rvert\leq C_2$.
	\end{itemize}
\end{proposition}
\begin{proof}
The proof can be obtained following verbatim \cite[Proposition 3.12]{MerloG1cod}.
\end{proof}

\begin{remark}
	Proposition \ref{bddcurve} implies in particular that the function $s\mapsto Tr(\mathcal{Q}(s))$ is bounded.
\end{remark}

\vv

From Proposition \ref{bddcurve} we deduce that for any infinitesimal sequence $\{s_j\}_{j\in\N}$, by compactness we can extract a subsequence $\{s_{j_k}\}_{k\in\N}$, such that $\mathcal{Q}(s_{j_k})$ and $\mathcal{T}(s_{j_k})$ converge to some $\tilde{\mathcal{Q}}$ and $\tilde{\mathcal{T}}$ respectively. Therefore by Proposition \ref{prop10} we have
\begin{equation}
	\begin{split}
		0\leq&\lim_{k\to\infty} \Big\lvert\langle b(s_{j_k}), u_H\rangle+\langle \mathcal{Q}(s_{j_k}) u_H, u_H\rangle+\mathcal{T}(s_{j_k}) u_T\Big\rvert
		\leq \lim_{k\to\infty} s_{j_k}^{1/4}\lVert u\rVert^3B^\prime(s_{j_k}^{1/4}\lVert u\rVert)=0.
		\nonumber
	\end{split}
\end{equation}

Hence, for any $u\in\supp(\mu)$ it holds
\begin{equation}
	\lim_{k\to\infty}\langle b(s_{j_k}), u_H\rangle=-\langle \tilde{\mathcal{Q}} u_H, u_H\rangle-\tilde{\mathcal{T}}  u_T.
	\label{eq8}
\end{equation}

\begin{proposition}
	There exists $\overline{B}\in V$ such that $\lim_{k\to\infty}b(s_{j_k})=\overline{B}$.
\end{proposition}

\begin{proof}
The proof can be obtained following verbatim \cite[Proposition 3.13]{MerloG1cod}.
\end{proof}


\vv

If $\mu$ is invariant under dilations we can find a candidate (non-degenerate) quadric containing $\supp(\mu)$.

\begin{proposition}\label{CONO}
	If $\mu$ is invariant under dilations, i.e. $\lambda^{-h}T_{0,\lambda}\mu=\mu$ for any $\lambda>0$, then:
	\begin{itemize}
		\item[(i)]$b(s)=0$ for any $s>0$.
		\item[(ii)]$\langle u_H,\mathcal{Q}(1) u_H\rangle+\mathcal{T}(1)  u_T=0 \text{ for any $u\in\supp(\mu)$.}$
	\end{itemize}
\end{proposition}

\begin{proof}
See \cite[Proposition 3.14]{MerloG1cod}.
\end{proof}

\vv

\section{Uniform measures of dimension \texorpdfstring{$n+1$}{n+1} and quadratic surfaces}\label{section:unif_measure_quadratic_surface}

	\begin{definition}[Degenerated uniform measures] 
For $h\in \{1,\ldots,n+2\}$ we denote by $\mathcal{D}\mathcal{U}(h)$ the family of those $h$-uniform measures on $(\mathbb P^n, d)$ such that
	\begin{equation}
	    	\lim_{s\to 0} \mathcal{Q}(s)=0,
	    	\label{eq:degenerate}
	\end{equation}
where $\mathcal{Q}=\mathcal Q_\mu$ is as in Definition \ref{definitioncurve}-(ii).
\end{definition}
\vv

\begin{proposition}\label{prop:notdeg}
Let $\mu\in\mathcal{U}(n+1)\setminus \mathcal{D}\mathcal{U}(n+1)$. Then there are $\mathcal Q\in \mathrm{Sym}(n)\setminus \{0_n\}$, $b\in\R^n$, and $\tau\in\R$ such that 
$$\supp(\mu)\subseteq \bigl\{x\in\mathbb{P}^n:\langle x_H,\mathcal Q[x_H]+b\rangle+\tau x_T=0\bigr\}.$$
\end{proposition}

\begin{proof}
Since $\mu\in\mathcal{U}(n+1)\setminus \mathcal{D}\mathcal{U}(n+1)$, by Proposition \ref{bddcurve} there exist a sequence $s_i\to 0$ and $\mathcal{ Q}\in\mathrm{Sym}(n)\setminus \{0_n\}$ such that $\lim_{i\to \infty}\mathcal{Q}(s_i)=\mathcal{Q}$. This together with Proposition \ref{prop10} concludes the proof. 
\end{proof}

\vv
The definition of $\mathcal{Q}(s)$ immediately implies the following result.

\begin{proposition}\label{prop35}
If $\mu\in \mathcal{D}\mathcal{U}(h)$ then
$$0=\lim_{s\to 0}8 s^{\frac{3}{2}+\frac{h}{4}}\int \lvert z_H\rvert^6
					e^{-s\lVert z\rVert^4}\,d\mu(z)
		-(2n+4)s^{\frac{1}{2}+\frac{h}{4}}\int \lvert z_H\rvert^2 e^{-s\lVert z\rVert^4}\, d\mu(z).$$
\end{proposition}
\vv

\vv

\begin{proposition}\label{polinomialconv}
Let $\{\mu_i\}_{i\in\N}\subseteq \mathcal{U}(h)$ be a sequence converging to some $\mu\in \mathcal{U}(h)$. Then for any $\lambda>0$ and any polynomial $P$ on $\mathbb{P}^n$ we have
  \[
  	\lim_{i\to \infty}\int P(z)e^{-\lambda\lVert z\rVert^4}\,d\mu_i(z)=\int P(z)e^{-\lambda\lVert z\rVert^4}\, d\mu(z).
  \]
\end{proposition}

\begin{proof}
Fixed $P$ and $\lambda>0$, it is not hard to see that there exists $R_1>0$ such that 
\begin{equation}\label{eq:p641}
    \lvert P(z)\rvert e^{-\lambda\lVert z\rVert^4}\leq e^{-\lambda\lVert z\rVert^4/2}, \qquad \text{ for any } \lVert z\rVert\geq R_1.
\end{equation}

Note that for any $\nu\in\mathcal U(h)$ and any $R_2\geq R_1$ we have
\begin{equation*}
\begin{split}
        \Big\lvert\int_{B(0,R_2)^c} P(z)e^{-\lambda\lVert z\rVert^4}\,d\nu(z)\Big\rvert&\leq \int_{B(0,R_2)^c} \lvert P(z)\rvert e^{-\lambda\lVert z\rVert^4}\,d\nu(z)\\
        &\overset{\eqref{eq:p641}}{\leq} \int_{B(0,R_2)^c} e^{-\lambda\lVert z\rVert^4/2}\,d\nu(z)\overset{\eqref{eq:rad_int}}{=} h\int_{R_2}^\infty s^{h-1} e^{-\lambda s^4/2}\,ds.
\end{split}
\end{equation*}
Then, for any $\epsilon>0$ we can choose $R_2$ in such a way that
\begin{equation}\label{eq:p642}
    \Big\lvert\int_{B(0,R_2)^c} P(z)e^{-\lambda\lVert z\rVert^4}d\nu(z)\Big\rvert\leq \epsilon, \qquad \text{ for any }\nu\in \mathcal U(h).
\end{equation}

Hence, if we let $\eta$ be a smooth positive function such that $\eta=1$ on $B(0,R_2)$ and $\eta=0$ on  $B(0,2R_2)^c$ triangle inequality, the convergence of $\mu_i$ to $\mu$, and \eqref{eq:p642} yield
\begin{equation}\label{eq:p653}
\begin{split}
      &\Big\lvert\lim_{i\to \infty}\int P(z)e^{-\lambda\lVert z\rVert^4}d\mu_i(z)-\int P(z)e^{-\lambda\lVert z\rVert^4}d\mu(z)\Big\rvert\\
    &\qquad \leq \Big\lvert\lim_{i\to \infty}\int (1-\eta(z))P(z)e^{-\lambda\lVert z\rVert^4}d\mu_i(z)\Big\rvert\\
    &\qquad \quad+ \Big\lvert\lim_{i\to \infty}\int \eta(z)P(z)e^{-\lambda\lVert z\rVert^4}d\mu_i(z)-\int \eta(z)P(z)e^{-\lambda\lVert z\rVert^4}d\mu(z)\Big\rvert\\
    &\qquad \quad+\Big\lvert\int (1-\eta(z))P(z)e^{-\lambda\lVert z\rVert^4}d\mu(z)\Big\rvert<2\epsilon.
\end{split}
\end{equation}
As the choice of $\epsilon$ is arbitrary, \eqref{eq:p653} concludes the proof.
\end{proof}

 If $\mu\in\mathcal{D}\mathcal{U}(h)$ we can infer quite a lot of information on the structure of $\mu$ at infinity. To see this, let
	 $\nu\in \Tan_h(\mu,x)$ and consider a sequence $\{R_i\}_{i\in\N}$ such that $R_i\to \infty$ and $R_i^{-h}T_{0,R_i} \mu\rightharpoonup \nu$. Thus
	\begin{equation}
	    \begin{split}
	        0&=\lim_{i\to \infty} \mathcal{Q}_\mu(\lambda R_i^{-1})=\lim_{i\to \infty}\frac{8 (\lambda R_i)^{-\frac{3}{2}-\frac{h}{4}}}{C(h)}\int \lvert z_H\rvert^4 z_H\otimes z_H
					e^{-(\lambda R_i)^{-1}\lVert z\rVert^4}\,d\mu(z)\\&\qquad-\lim_{i\to\infty}\frac{\lambda^{\frac{1}{2}+\frac{h}{4}}R_i^{-\frac{1}{2}-\frac{h}{4}}}{C(h)}\int \bigl(4z_H\otimes z_H+2\lvert z_H\rvert^2\textrm{id}_{n}\bigr) e^{-\lambda R_i^{-1}\lVert z\rVert^4}\,d\mu(z)\\
			&=\lim_{i\to \infty}\frac{8\lambda^{\frac{3}{2}+\frac{h}{4}} }{C(h)}\int \lvert z_H\rvert^4 z_H\otimes z_H
					e^{-\lambda\lVert z\rVert^4}\, \frac{d\bigl[T_{x,R_i^{1/4}}\mu\bigr](z)}{R_i^{h/4}}\\
			&\qquad-\lim_{i\to\infty}\frac{\lambda^{\frac{1}{2}+\frac{h}{4}}}{C(h)}\int \bigl(4z_H\otimes z_H+2\lvert z_H\rvert^2\textrm{id}_{n}\bigr) e^{-\lambda\lVert z\rVert^4}\, \frac{d\bigl[T_{x,R_i^{1/4}}\mu\bigr](z)}{R_i^{h/4}}\\
			&	=\frac{8 \lambda^{\frac{3}{2}+\frac{h}{4}} }{C(h)}\int \lvert z_H\rvert^4 z_H\otimes z_H
					e^{-\lambda\lVert z\rVert^4}d\nu(z)-\frac{\lambda^{\frac{1}{2}+\frac{h}{4}}}{C(h)}\int \bigl(4z_H\otimes z_H+2\lvert z_H\rvert^2\textrm{id}_{n}\bigr) e^{-\lambda\lVert z\rVert^4}\,d\nu(z),
	    \end{split}
	\end{equation}
	where the last equality holds because of the weak convergence $R_i^{-h}T_{0,R_i} \mu\rightharpoonup \nu$ and Proposition \ref{polinomialconv}.
In particular, for any $u\in \R^n$ such that $\lvert u\rvert=1$ and any $\lambda>0$ we obtain
\begin{equation}
    \lambda^{\frac{1}{2}+\frac{h}{4}}\int \bigl(8\lambda\lvert z_H\rvert^4 -4\bigr)\langle z_H,u\rangle^2
					e^{-\lambda\lVert z\rVert^4}\, d\nu(z)=  2\lambda^{\frac{1}{2}+\frac{h}{4}}\int \lvert z_H\rvert^2e^{-\lambda\lVert z\rVert^4}\,d\nu(z).	\label{eq:iddege}
\end{equation}

\vv

We recall that $\mathcal{V}\coloneqq \{(0,s)\in \R^n\times \R: s\in \R\}$.
\begin{proposition}\label{propvert}
If $\mu$ is an $h$-uniform measure on $\mathbb{P}^n$ and $\supp(\mu)\subseteq \mathcal{V}$, then $h=2$ and $\mu=\mathcal{C}^2\llcorner\mathcal{V}$.
\end{proposition}

\begin{proof}
Since $\supp(\mu)\subseteq \mathcal{V}$, then $\mu(B(z,r)\cap \mathcal{V})=\mu(B(z,r))=r^h$ for any $z\in\supp(\mu)$ and any $r>0$. Note that
$$B(z,r)\cap \mathcal{V}=\{(0,s)\in\R^{n}\times\R:\lvert s-  z_T\rvert<r^2\}=B_{n+1}(z,r^2)\cap \mathcal{V},$$
where \label{euclidean} $B_{n+1}(z,r^2)$ denotes the Euclidean ball in $\mathbb R^{n+1}$ of center $z$ and radius $r^2$.

This implies that
$\mu(B_{n+1}(z,r))=r^{h/2}$, so $\mu$ is a $h/2$-uniform measure with respect to Euclidean balls and its support is contained in $\mathcal{V}$. Marstrand's Theorem implies that $h/2$ is integer and, since $\mathcal{V}$ is $1$-dimensional, by differentiation we deduce that $h/2$ is either $0$ or $1$. As we excluded by hypothesis the case $h=0$ and thanks to the classification of $1$-uniform measures in $\R^{n+1}$ provided by \cite{Preiss1987GeometryDensities}, we deduce that $\mu=\frac{1}{2}\mathcal{H}_{\eu}^1\llcorner \mathcal{V}$. Since $h=2$ and $\supp(\mu)=\mathcal{V}$, Proposition \ref{supportoK} concludes the proof.
\end{proof} 
\vv

\begin{proposition}\label{disconnectdegenerated}
    Let $\mu\in\mathcal{D}\mathcal{U}(h)$. Then either $h=2$ or there exists a constant $\newep\label{eps:1}=\oldep{eps:1}(h)>0$ such that
    \begin{equation}\label{eq:prop47_statement}
        \lambda^{\frac{3}{2}+\frac{h}{4}}\int \lvert z_H\rvert^4\langle z_H,u\rangle^2
					e^{-\lambda\lVert z\rVert^4}\, d\nu(z)>\oldep{eps:1}
    \end{equation}
    for any $\nu\in \Tan_h(\mu,\infty)$, any $\lambda>0$, and any $u\in\R^n$ with $\lvert u\rvert=1$.
\end{proposition}

\begin{proof}
We argue by contradiction. More specifically, let us suppose that there exist sequences $\mu_i\in\mathcal{U}(h)$, $\nu_i\in\Tan_h(\mu_i,\infty)$, and $\lambda_i>0$ such that
\[
	\lambda_i^{\frac{1}{2}+\frac{h}{4}}\int \bigl(8\lambda_i\lvert z_H\rvert^4 -4\bigr)\langle z_H,u\rangle^2
					e^{-\lambda_i\lVert z\rVert^4}\, d\nu_i(z)\leq 1/i.
\]	

					The latter inequality also reads
    \begin{equation}
    \begin{split}
                   2\int \lvert z_H\rvert^2&e^{-\lVert z\rVert^4}\, \frac{d\bigl[ T_{0,\lambda_i^{-1/4}}\nu_i\bigr](z)}{\lambda_i^{-h/4}}=2\lambda_i^{\frac{1}{2}+\frac{h}{4}}\int \lvert z_H\rvert^2e^{-\lambda_i\lVert z\rVert^4}\, d\nu_i(z)\\
                   &\overset{\eqref{eq:iddege}}{=}\lambda_i^{\frac{1}{2}+\frac{h}{4}}\int \bigl(8\lambda_i\lvert z_H\rvert^4 -4\bigr)\langle z_H,u\rangle^2
					e^{-\lambda_i\lVert z\rVert^4}\, d\nu_i(z)\leq 1/i.
					\label{eq:id1}
    \end{split}
    \end{equation}
    
     Since $\nu_i\in \mathcal U(h)$, the measures $\lambda_i^{h/4}T_{0,\lambda_i^{-1/4}}\nu_i$ are $h$-uniform, too. Thus there exists $\nu\in \mathcal U(h)$ such that, possibly passing to a subsequence, by Propositions \ref{replica} and \ref{UComp} we have $\lambda_i^{h/4}T_{0,\lambda_i^{-1/4}}\nu_i\rightharpoonup \eta$. In addition, taking the limit in \eqref{eq:id1} we obtain
    \begin{equation}
         \int \lvert z_H\rvert^2e^{-\lVert z\rVert^4}\, d\eta(z)=0,
    \end{equation}
    and thus $\supp(\eta)\subseteq \mathcal{V}$ thanks to Proposition \ref{propvert}. This in turn implies that $h=2$. Hence, if $h\neq 2$ there exists $\tilde \varepsilon=\tilde\varepsilon(h)>0$ such that for any $\mu\in\mathcal{D}\mathcal{U}(h)$ we have
    \[
    	\lambda^{\frac{1}{2}+\frac{h}{4}}\int (8\lambda\lvert z_H\rvert^4 -4)\langle z_H,u\rangle^2
					e^{-\lambda\lVert z\rVert^4}\, d\nu(z)>\tilde\varepsilon,
	\]
					for any $\nu\in\Tan_h(\mu,\infty)$ and any $\lambda>0$, which immediately implies \eqref{eq:prop47_statement}.
\end{proof}

\vv

Proposition \ref{disconnectdegenerated} has the following important consequence.

	\begin{corollary}\label{noflatatinfty}	    If there exists $\mu\in\mathcal{DU}(n+1)$ such that 
	  $\Tan_{n+1}(\mu,\infty)\cap \mathfrak{M}(n+1)\neq \emptyset$, then $n=1$.
	\end{corollary}

	\begin{proof}
We argue by contradiction: we assume that there exists $\nu\in\Tan_{n+1}(\mu,\infty)\cap \mathfrak{M}(n+1)$ for $n>1$. By Corollary \ref{homplanesstructure} there exists $v\in\R^n\setminus \{0\}$ such that $V=v^\perp\otimes \R e_{n+1}$ and $\supp(\nu)=V$, where $v^\perp$ denotes the orthogonal complement of $v$ inside $\R^n$. Hence we have
	\[
		\int \lvert z_H\rvert^4\langle z_H,v\rangle^2
					e^{-\lVert z\rVert^4}\,d\mathcal{H}^{n+1}\llcorner V(z)=0,
	\]
					which contradicts Proposition \ref{disconnectdegenerated} and concludes the proof.
	\end{proof}
	
\vv

\section{Structure of non-degenerated uniform measures}\label{eq:sec_non_deg_unif_meas}

Throughout this section we assume that $\mu\in\mathcal{U}(n+1)\setminus \mathcal{DU}(n+1)$. Then, Proposition \ref{prop:notdeg} implies that there exist $\mathcal Q,b,$ and $\tau$ such that
\[
	\supp(\mu)\subseteq\bigl\{u\in\mathbb{P}^n:\langle b,u_H\rangle+\langle u_H,\mathcal{Q}u_H\rangle+\tau u_T=0\bigr\}\eqqcolon\mathbb{K}(b,\mathcal{Q},\tau).
\]
Furthermore, we define
\[
	\Sigma(b,\mathcal{Q},\tau)\coloneqq \{u\in\mathbb{K}(b,\mathcal{Q},\tau):b+2\mathcal{Q}u_H=0\}.
\]

\begin{proposition}\label{prop.tangreg}
For any $p\in\supp(\mu)\setminus \Sigma(b,\mathcal{Q},\tau)$ we have
$$\Tan_{n+1}(\mu,p)\subseteq \bigl\{\mathcal{C}^{n+1}\llcorner (b+2\mathcal{Q}p_H)^\perp\bigr\}.$$
\end{proposition}

\begin{proof}
An argument analogous to that in \cite[Proposition 4.3]{MerloG1cod} shows that $\supp (\nu)\subset (\mathcal{Q}p_H)^\perp$ for any $\nu\in \Tan(\mu,p)$. Hence, Proposition \ref{verticalsamoa} concludes the proof.
\end{proof}

\vv

\subsection{Uniform measures with \texorpdfstring{$\tau\neq 0$}{t neq 0} do not exists}\label{sectneq0notexist}

Throughout this section, if not otherwise specified, we assume that $n\geq 2$.

\begin{proposition}\label{prop:touchhor}
Suppose that $\nu\in \mathcal U(n+1)$ is such that  $\supp(\nu)\subseteq \mathbb{K}(0,\mathcal{D},-1)$. If we let $y\in\mathbb{K}(0,\mathcal{D},-1)\setminus \supp(\nu)$ and take $z\in\supp(\nu)$ such that
$$\lvert z_H-y_H\rvert=\dist_{\mathrm{eu}}\bigl(y_H,\pi_H(\supp(\nu))\bigr),$$
then $z\in \Sigma(0,\mathcal{D},-1)$.
\end{proposition}

\begin{proof}
First of all let us observe that $\mathcal{D}\neq 0$, otherwise $\supp(\nu)\subseteq \{x_{n+1}=0\}$ and this would imply that $\nu=0$ thanks to Proposition \ref{supportoK} and the fact that the plane $\{x_{n+1}=0\}$ has dimension $n$.

Assume by contradiction that $z\not\in \Sigma(0,\mathcal{D},-1)$. Then, by Propositions \ref{propspt1} and \ref{prop.tangreg} we infer that for any $w\in (\mathcal{D}z_H)^\perp$ and any infinitesimal sequence $\{r_i\}_{i\in\N}$ there exists a sequence $\{ z_i\}_{i\in\N}\subseteq \supp(\nu)$ such that
\[
	\delta_{1/r_i}(z_i-z)\to w.
\]
In particular, there exists a sequence $\{\Delta_i\}\subseteq \mathbb{P}^n$ such that $\lVert \Delta_i\rVert$ is infinitesimal and:
\begin{itemize}
    \item[\hypertarget{(i)}{(i)}] $(z_i)_H\coloneqq z_H+r_iw_H+r_i(\Delta_i)_H$.
    \item[(ii)] $(z_i)_T\coloneqq z_T+r_i^2w_T+r_i^2(\Delta_i)_T$.
\end{itemize}

Since $z_i\in\supp(\nu)$ for all $i\in \mathbb N$, these points also belong to $\mathbb{K}(0,\mathcal{D},-1)$. This implies in particular that
\begin{equation}
\begin{split}
        z_T+r_i^2w_T+r_i^2(\Delta_i)_T=&(z_i)_T=\langle (z_i)_H,\mathcal{D}(z_i)_H\rangle\\
        =&\langle z_H+r_iw_H+r_i(\Delta_i)_H,\mathcal{D}[z_H+r_iw_H+r_i(\Delta_i)_H]\rangle\\
        =&\langle z_H,\mathcal{D}z_H\rangle+2r_i\langle z_H,\mathcal{D}w_H\rangle+2r_i\langle z_H,\mathcal{D}(\Delta_i)_H\rangle+r_i^2\langle w_H,\mathcal{D} w_H\rangle\\
        &\qquad\qquad\qquad+2r_i^2\langle w_H,\mathcal{D} (\Delta_i)_H\rangle+r_i^2\langle(\Delta_i)_H,\mathcal{D} (\Delta_i)_H\rangle.
        \label{eq:quad1}
\end{split}
\end{equation}

Moreover, $z\in\mathbb{K}(0,\mathcal{D},-1)$ and $w\in (\mathcal{D}z_H)^\perp$, so \eqref{eq:quad1} simplifies to
\begin{equation}
\begin{split}
        r_i^2w_T+&r_i^2(\Delta_i)_T
        =2r_i\langle z_H,\mathcal{D}(\Delta_i)_H\rangle+r_i^2\langle w_H,\mathcal{D} w_H\rangle\\
        &\qquad\qquad\qquad+2r_i^2\langle w_H,\mathcal{D} (\Delta_i)_H\rangle+r_i^2\langle(\Delta_i)_H,\mathcal{D} (\Delta_i)_H\rangle.
        \label{eq:quad2}
\end{split}
\end{equation}
Since by construction we know that $\lvert (z_i)_H-y_H\rvert^2\geq \lvert z_H-y_H\rvert^2$, \hyperlink{(i)}{(i)} and a simple computation yield
\begin{equation}
    \begin{split}
        2\bigl\langle z_H-y_H, r_iw_H+r_i(\Delta_i)_H\bigr\rangle+r_i^2\lvert w_H+(\Delta_i)_H\rvert^2\geq 0.
        \label{eq:quad4}
    \end{split}
\end{equation}

We divide \eqref{eq:quad4} by $r_i$, take the limit as $i\to \infty$, and conclude that
\begin{equation}
    \langle z_H-y_H, w_H\rangle\geq 0.
    \label{eq:quad3}
\end{equation}
However, since every element of $(\mathcal{D}z_H)^\perp$ satisfy \eqref{eq:quad3}, we have that there exists $\lambda\neq 0$ such that $\lambda \mathcal{D}z_H=z_H-y_H$. Thus we collect \eqref{eq:quad2} and \eqref{eq:quad4}, and write
\begin{equation}
    \begin{split}
       & 0\leq  2\lambda^{-1}\langle z_H, \mathcal{D}(\Delta_i)_H\rangle+r_i\lvert w_H+(\Delta_i)_H\rvert^2\\
       &\qquad\qquad =\lambda^{-1}r_i\big[w_T+(\Delta_i)_T-\langle w_H,\mathcal{D} w_H\rangle-2\langle w_H,\mathcal{D} (\Delta_i)_H\rangle\\
        &\qquad\qquad\qquad\qquad-\langle(\Delta_i)_H,\mathcal{D} (\Delta_i)_H\rangle\big]+r_i\lvert w_H+(\Delta_i)_H\rvert^2.
        \label{eq:quad5}
    \end{split}
\end{equation}

We divide by $r_i$ both sides of \eqref{eq:quad5}, take the limit as $i\to \infty$, and infer that
\begin{equation}
    \begin{split}
       & 0\leq \lambda^{-1}\big[w_T-\langle w_H,\mathcal{D} w_H\rangle]+\lvert w_H\rvert^2,
        \label{eq:quad5bis}
    \end{split}
\end{equation}
which turns into $w_T\leq -\langle w_H,\mathcal{D} w_H\rangle+\lambda\lvert w_H\rvert^2$. This constitutes a non-trivial bound on $w_T$ that contrasts with the arbitrariness of our choice of $w$. Hence, \eqref{eq:quad5bis} contradicts the assumption $z\not \in \Sigma(0,\mathcal{D},-1)$, which finishes the proof.
\end{proof}

\vv

In the two following propositions we use the results of Appendix \ref{TYLR}.
\begin{proposition}\label{prop:inftyhorgood}
Suppose that $\nu\in \mathcal U(n+1)$ is such that  $\supp(\nu)\subseteq \mathbb{K}(0,\mathcal{D},-1)$. Then $\supp(\nu)=\mathbb{K}(0,\mathcal{D},-1)$ and $\mathrm{dim}(\mathrm{Ker}(\mathcal{D}))\leq n-2$.
\end{proposition}

\begin{proof} Analogously to the proof of Proposition \ref{prop:touchhor}, we can assume that $\mathrm{Ker}(\mathcal{D})\neq \{	0\}$.

Let $C$ be a connected component of $\mathbb{K}(0,\mathcal{D},-1)\setminus \Sigma(0,\mathcal{D},-1)$. Since $\supp(\nu)\cap C$ is relatively closed in $C$, if we can prove that it is also relatively open, the connectedness of $C$ would imply that either $\supp(\nu)\cap C=\emptyset$ or $\supp(\nu)\cap C=C$. So, let us assume that $\supp(\nu)\cap C\neq \emptyset$ and, arguing by contradiction, that there exist $z\in \supp(\nu)\cap C$ and a sequence of points $y(i)\in C\setminus \supp(\nu)$ such that $y(i)\to z$.
Then, let $z(i)\in \supp(\nu)$ be such that
$$\lvert z(i)_H-y(i)_H\rvert=\dist_{\mathrm{eu}}(y(i)_H,\pi_H(\supp(\nu))),$$
and note that $z(i)\to z$. By Proposition \ref{prop:touchhor} we have $z(i)\in \Sigma(0,\mathcal{D},-1)$, so $z$ itself must be in $\Sigma(0,\mathcal{D},-1)$ since $\Sigma(0,\mathcal{D},-1)$ is a closed set. However this contradicts the choice of $z$.

If the topological dimension of $\mathrm{Ker}(\mathcal{D})$ is smaller than $n-2$, then $\mathbb{K}(0,\mathcal{D},-1)\setminus \Sigma(0,\mathcal{D},-1)$ is a connected set since it is the image under the continuous map $u\mapsto(u,\langle u,\mathcal{D}u\rangle)$ of the connected set $\R^n\setminus \mathrm{Ker}(\mathcal{D})$. Hence, in this case, we have $\supp(\nu)=\mathbb{K}(0,\mathcal{D},-1)$.

In order to conclude the proof, it is enough to show that $\mathrm{Ker}(\mathcal{D})=n-1$ is not possible. We argue by contradiction and assume that $\mathrm{Ker}(\mathcal{D})=n-1$. This implies that $\mathbb{K}(0,\mathcal{D},-1)\setminus \Sigma(0,\mathcal{D},-1)$ has two connected components and at least one of them, which we denote as $C$, must be contained in $\supp(\nu)$, otherwise we would have $\nu=0$. Hence we can find a unitary eigenvector $e$ relative to the only non-null eigenvalue of $\mathcal{D}$ such that $\mathcal{X}(e)\coloneqq (e,\langle e,\mathcal{D}e\rangle)\in C$. In addition, since $C$ is relatively open in $\mathbb{K}(0,\mathcal{D},-1)$, there exists $r>0$ such that $$B(\mathcal{X}(e),r)\cap \supp(\nu)=B(\mathcal{X}(e),r)\cap C.$$

So, thanks to Theorem \ref{strutturaunifmeasvssurfacemeasure}, Proposition \ref{unif:expansion}, and denoting by $\lambda$ the only non-null eigenvalue of $\mathcal{D}$ and $\mathfrak{n}(e)\coloneqq \mathcal De/\lvert \mathcal De\rvert=\mathrm{sgn}(\lambda) e$, we have
\begin{equation*}
\begin{split}
      0=&\frac{\lambda^2-2\langle \mathfrak{n}(e),\mathcal{D}^2\mathfrak{n}(e)\rangle+\langle \mathfrak{n}(e),\mathcal{D}\mathfrak{n}(e)\rangle^2}{4(n-1)}-\frac{1}{4}-\frac{(\lambda-\langle \mathfrak{n}(e),\mathcal{D}\mathfrak{n}(e)\rangle)^2}{8(n-1)}\\
      =&\frac{\lambda^2-2\lambda^2+\lambda^2}{4(n-1)}-\frac{1}{4}-\frac{(\lambda-\lambda)^2}{8(n-1)}=-\frac{1}{4},
    \label{eq:id:id1}
\end{split}
\end{equation*}
which is a contradiction and finishes the proof.
\end{proof}

\vv

\begin{proposition}\label{prop:noT}
There exists no $\mu\in\mathcal{U}(n+1)$ whose support is contained in a quadric $\mathbb{K}(b,\mathcal{Q},\tau)$ with $\tau\neq 0$.
\end{proposition}

\begin{proof}
We recall that the set $\Tan_{n+1}(\mu,\infty)$ is  non-empty thanks to Proposition \ref{uniformup}. Via the same argument used in the proof of \cite[Proposition 5.2]{MerloG1cod}, one can show that any $\nu\in\Tan_{n+1}(\mu,\infty)$ is a uniform measure whose support is contained in the quadric $\mathbb{K}(0,\mathcal{Q}/\tau,-1)$. Thanks to Proposition \ref{prop:inftyhorgood} we know that $\mathrm{Ker}(\mathcal{Q})$ has dimension smaller than $n-2$ and that
$$\supp(\nu)=\mathbb{K}(0,\mathcal{Q}/\tau,-1).$$

Let $\mathcal{D}\coloneqq \mathcal{Q}/\tau$. By Theorem \ref{strutturaunifmeasvssurfacemeasure} and Proposition \ref{unif:expansion} we have
\begin{equation}
    0=\frac{\text{Tr}(\mathcal{D}^2)-2\langle \mathfrak{n}(x),\mathcal{D}^2\mathfrak{n}(x)\rangle+\langle \mathfrak{n}(x),\mathcal{D}\mathfrak{n}(x)\rangle^2}{4(n-1)}-\frac{1}{4}-\frac{(\text{Tr}(\mathcal{D})-\langle \mathfrak{n}(x),\mathcal{D}\mathfrak{n}(x)\rangle)^2}{8(n-1)},
    \label{eq:id:id_bis}
\end{equation}
where $\mathfrak{n}(x)\coloneqq\mathcal{D}x/\lvert\mathcal{D}x\rvert$ for any $x\not \in \mathrm{Ker}(\mathcal{D})$. If $e$ is a unitary eigenvector relative to a non-null eigenvalue $\lambda$ of $\mathcal{D}$, we have that $\mathfrak{n}(e)=\mathrm{sgn}(\lambda)e$ and thus \eqref{eq:id:id_bis} reads
\begin{equation}
    0=\frac{\text{Tr}(\mathcal{D}^2)-\lambda^2}{4(n-1)}-\frac{1}{4}-\frac{(\text{Tr}(\mathcal{D})-\lambda)^2}{8(n-1)}.
    \label{eq:id:id2}
\end{equation}

The identity \eqref{eq:id:id2} implies that the non-null eigenvalues $\lambda$ satisfy a quadratic equation and thus $\mathcal{Q}/\tau$ has at most two distinct eigenvalues.
Denote by $\lambda_1$ and $\lambda_2$ the solutions to \eqref{eq:id:id2}, which satisfy
\begin{equation}
    \begin{cases}
    \lambda_1+\lambda_2=\frac{2\mathrm{Tr}(\mathcal{D})}{3},\\
    \lambda_1\lambda_2=\frac{2(n-1)+\mathrm{Tr}(\mathcal{D})^2-2\mathrm{Tr}(\mathcal{D}^2)}{3}.
    \end{cases}
    \label{eq:id:id:4}
\end{equation}

The spectral theorem implies that any $x\in \R^n\setminus \mathrm{Ker}(\mathcal{D})$ can be uniquely written as $x=v_0+v_1+v_2$, where $v_0,v_1,v_2$ are eigenvectors of $0$, $\lambda_1$, and $\lambda_2$ respectively. Observe that, for such $x$, the vector $\mathfrak{n}(x)$ becomes
\[
	\mathfrak{n}(x)=\frac{\lambda_1 v_1+\lambda_2 v_2}{\lvert \lambda_1 v_1+\lambda_2 v_2\rvert}=\frac{\lambda_1 v_1+\lambda_2 v_2}{\sqrt{\lambda_1^2\lvert v_1\rvert+\lambda_2^2\lvert v_2\rvert}}.
\]
Moreover
\begin{equation}
	\begin{split}
		\langle \mathfrak{n}(x),\mathcal{D}^2\mathfrak{n}(x)\rangle=&\frac{\lambda_1^4\lvert v_1\rvert^2+\lambda_2^4\lvert v_2\rvert^2}{\lambda_1^2\lvert v_1\rvert^2+\lambda_2^2\lvert v_2\rvert^2},\\
		\langle\mathfrak{n}(x),\mathcal{D}\mathfrak{n}(x)\rangle=&\frac{\lambda_1^3\lvert v_1\rvert^2+\lambda_2^3\lvert v_2\rvert^2}{\lambda_1^2\lvert v_1\rvert^2+\lambda_2^2\lvert v_2\rvert^2}.
		\label{valuesnandD}
	\end{split}
\end{equation}
Therefore, thanks to \eqref{eq:id:id:4}, the identity \eqref{eq:id:id_bis} also reads
\begin{equation}\label{simpid}
\begin{split} 
    0=&-4\langle \mathfrak{n}(x),\mathcal{D}^2\mathfrak{n}(x)\rangle+\langle \mathfrak{n}(x),\mathcal{D}\mathfrak{n}(x)\rangle^2+2\mathrm{Tr}(\mathcal{D})\langle\mathfrak{n}(x),\mathcal{D}\mathfrak{n}(x)\rangle\\
&\qquad\qquad\qquad+(2\mathrm{Tr}(\mathcal{D}^2)-2(n-1)-\mathrm{Tr}(\mathcal{D})^2)\\
=&-4\langle \mathfrak{n}(x),\mathcal{D}^2\mathfrak{n}(x)\rangle+\langle \mathfrak{n}(x),\mathcal{D}\mathfrak{n}(x)\rangle^2+3(\lambda_1+\lambda_2)\langle\mathfrak{n}(x),\mathcal{D}\mathfrak{n}(x)\rangle-3\lambda_1\lambda_2
\\
    \overset{\eqref{valuesnandD}}{=}&-4\frac{\lambda_1^4\lvert v_1\rvert^2+\lambda_2^4\lvert v_2\rvert^2}{\lambda_1^2\lvert v_1\rvert^2+\lambda_2^2\lvert v_2\rvert^2}+\frac{(\lambda_1^3\lvert v_1\rvert^2+\lambda_2^3\lvert v_2\rvert^2)^2}{(\lambda_1^2\lvert v_1\rvert^2+\lambda_2^2\lvert v_2\rvert^2)^2}
    +3(\lambda_1+\lambda_2)\frac{\lambda_1^3\lvert v_1\rvert^2+\lambda_2^3\lvert v_2\rvert^2}{\lambda_1^2\lvert v_1\rvert^2+\lambda_2^2\lvert v_2\rvert^2}-3\lambda_1\lambda_2.
\end{split}
\end{equation}

Few omitted algebraic computations show that \eqref{simpid} boils down to
\begin{equation}
0=-\frac{\lambda_1^2\lambda_2^2\lvert v_1\rvert^2\lvert v_2\rvert^2(\lambda_1-\lambda_2)^2}{(\lambda_1^2\lvert v_1\rvert^2+\lambda_2^2\lvert v_2\rvert^2)^2}.
    \label{eq:risolv}
\end{equation}
Since \eqref{eq:risolv} holds for any arbitrary couple $(v_1,v_2)\neq (0,0)$ and since $\lambda_1,\lambda_2\neq 0$, we conclude that $\lambda_1=\lambda_2$. We denote as $n-k$ the dimension of $\mathrm{Ker}(\mathcal{D})$ and we plug the information $\lambda\coloneqq\lambda_1=\lambda_2$ into \eqref{eq:id:id:4} which allows us to conclude that $k=3$, $\mathrm{Tr}(\mathcal{D}^2)=k\lambda^2$, and
{\[
	\lambda^2=\frac{2(n-1)+(k\lambda)^2-2k\lambda^2}{3}=\frac{2(n-1)}{3}+\lambda^2.
\]}
However, the above identity can hold only if $n=1$, which is excluded by hypothesis. This concludes the proof of the proposition.
\end{proof}
\vv

\subsection{Structure of uniform measures with \texorpdfstring{$\tau=0$}{t=0}.} As in the previous subsection, we assume that $n\geq 2$. 

\begin{proposition}\label{proppi.verti}
Let $\nu\in\mathcal U(n+1)$ be such that  $\supp(\nu)\subseteq \mathbb{K}(b,\mathcal{Q},0)$. If  $y\in\mathbb{K}(b,\mathcal{Q},0)\setminus \supp(\nu)$ and $z\in\supp(\nu)$ is such that
\begin{equation}\label{eq:zdistpt1}
    \lvert z-y\rvert=\dist_{\mathrm{eu}}(y,\supp(\nu)),
\end{equation}
then $z\in\Sigma(b,\mathcal{Q},0)$.
\end{proposition}

\begin{proof}
Assume by contradiction that $z\not\in \Sigma(b,\mathcal{D},0)$. Then, by Propositions \ref{propspt1} and \ref{prop.tangreg} we infer that
for any $w\in (b+2\mathcal{Q}z_H)^\perp$ and any infinitesimal sequence $\{r_i\}_{i\in\N}$ there exists a sequence $\{ z_i\}_{i\in\N}\subset\supp(\mu)$ such that
\[
	\delta_{1/r_i}(z_i-z)\to w.
\]
In particular, there exists a sequence $\{\Delta_i\}\subseteq \mathbb{P}^n$ such that $\lVert \Delta_i\rVert$ is infinitesimal and
\begin{itemize}
    \item[\hypertarget{(ia)}{(i)}] $(z_i)_H\coloneqq z_H+r_iw_H+r_i(\Delta_i)_H$.
    \item[\hypertarget{(iia)}{(ii)}] $(z_i)_T\coloneqq z_T+r_i^2w_T+r_i^2(\Delta_i)_T$.
\end{itemize}
Since $z_i\in \supp(\nu)$ then the choices of $y$ and $z$ as in \eqref{eq:zdistpt1} imply that
\begin{equation}\label{eq:prop75a1}
    \lvert z_i-y\rvert^2\geq \lvert z-y\rvert^2.
\end{equation}
Therefore, plugging \hyperlink{(ia)}{(i)} and \hyperlink{(iia)}{(ii)} inside \eqref{eq:prop75a1} we have
\begin{equation}
\begin{split}
    2r_i\langle z_H-y_H, w_H+&(\Delta_i)_H\rangle+r_i^2\lvert w_H+(\Delta_i)_H \rvert^2\\
    +&2r_i^2(w_T+(\Delta_i)_T)(z_T-y_T)+r_i^4(w_T+(\Delta_i)_T)^2\geq 0.
    \label{eq.balltang}
\end{split}
\end{equation}

We take the limit in \eqref{eq.balltang} as $i\to\infty$, and obtain that $\langle z_H-y_H, w_H\rangle\geq 0$. The arbitrariness of $w\in (b+2\mathcal{Q}z_H)^\perp$ implies that either $z_H=y_H$ or $z_H-y_H=\lambda(b+2\mathcal{Q}z_H)$ for some $\lambda\neq 0$ and thus $\langle z_H-y_H, w_H\rangle=0$.

Let us first assume that $z_H=y_H$. In this case we have $y_T\neq z_T$ and \eqref{eq.balltang} implies, for $i\to \infty$, that
$$\lvert w_H\rvert^2+2w_T(z_T-y_T)\geq 0,$$
which is an inequality that cannot be satisfied for any $w\in (b+2\mathcal{Q}z_H)^\perp$. Therefore the assumption $z_H=y_H$ cannot hold and this case is excluded.

Thus, we are left with the only other possibility that there exists $\lambda\neq 0$ such that $z_H-y_H=\lambda (b+2\mathcal{Q} z_H)$, which yields that
\begin{equation}\label{eq:idp752a}
    \langle z_H-y_H, z_H\rangle=\lambda\langle b+2\mathcal{Q} z_H, z_H\rangle=0,
\end{equation}
where the last equality above holds because $z\in \mathbb{K}(b,\mathcal{Q},0)$.
The identity \eqref{eq:idp752a} finally implies that $z_H=y_H$ and in particular $0=\lambda (b+2\mathcal{Q} z_H)$, which contradicts the facts that $\lambda\neq 0$ and $\mathcal{Q} z_H\neq 0$.
\end{proof}

\vv

\begin{proposition}\label{prop:invvert}
Let $\nu\in \mathcal U(n+1)$ be such that  $\supp(\nu)\subseteq \mathbb{K}(b,\mathcal{Q},0)$. If $C$ is a connected component of $\mathbb{K}(b,\mathcal{Q},0)\setminus \Sigma(b,\mathcal{Q},0)$, then either $C\subseteq\supp(\nu)$ or $\supp(\nu)\cap C=\emptyset$. In particular $\supp(\nu)$ is vertically invariant.
\end{proposition}

\begin{proof}
If $\mathcal{Q}$ has rank $1$ and $b=0$, then $\mathbb{K}(0,\mathcal{Q},0)\setminus \Sigma(0,\mathcal{Q},0)=\emptyset$ and there is nothing to prove. Note however that in this case $\mathbb{K}(0,\mathcal{Q},0)\in\Gr(n+1)$, so by Proposition \ref{verticalsamoa} the measure $\nu$ must be flat and thus $\mathbb{K}(0,\mathcal{Q},0)=\supp(\nu)$.

Let $a\in \R^n$ be such that $\mathcal Q=a\otimes a$. If $b$ and $a$ are linearly independent, then $\Sigma(b,\mathcal{Q},0)=\emptyset$. If, on the other hand, $b=\lambda a$ for some $\lambda>0$ then $$\lambda a+2a\langle a,z_H\rangle=0,\qquad \text{ for }z\in \Sigma(b,\mathcal{Q},0),$$ 
which contradicts the fact that $z\in \mathbb{K}(\lambda a,a\otimes a,0)$.

Assume now that $\mathrm{rk}(\mathcal{Q})\geq 2$. The dimension of $\ker(\mathcal{Q})$ is smaller than $ n-2$ in $\R^n$, so by Proposition \ref{strutturaunifmeasvssurfacemeasure} and the definition of the surface measures $\sigma_{\supp(\nu)}$ and $\sigma_{\mathbb{K}(b,\mathcal{Q},0)}$ it holds
\[
    \nu(\Sigma(b,\mathcal{Q},0))=\sigma_{\supp(\nu)}(\Sigma(b,\mathcal{Q},0))=\sigma_{\mathbb{K}(b,\mathcal{Q},0)}(\Sigma(b,\mathcal{Q},0))=0.
\]

The discussion above implies that there exists a non-empty connected component of $\mathbb{K}(b,\mathcal{Q},0)\setminus \Sigma(b,\mathcal{Q},0)$, that we denote as $C$.
We now claim that $\supp(\nu)\cap C$ is relatively open inside $C$ which, since $\supp(\nu)\cap C$ is relatively closed inside $C$, allows us to infer that either $C\cap \supp(\nu)\neq \emptyset$ or $C\cap \supp(\nu)=C$.  

Suppose by contradiction that $C\cap \supp(\nu)$ is non-empty and not open. Then we can find $z\in \supp(\nu)\cap C $ and two sequences $\{y_i\}_{i\in\N}\subset \mathbb{K}(b,\mathcal{Q},0)\setminus \supp(\nu)$ and $\{z_i\}_{i\in\N}\subset \supp(\nu)$ such that the $z_i\to z$, $y_i\to z$, and
$\lvert y_i-z_i\rvert=\dist_{\eu}(y_i,\supp(\nu))$. Moreover,  since $\mathbb{K}(b,\mathcal{Q},0)\setminus \Sigma(b,\mathcal{Q},0)$ is an open set, the points $y_i,z_i$ cannot be contained in $\Sigma(b,\mathcal{Q},0)$ for $i$ big enough as they converge to $z$. This however is in contradiction with Proposition \ref{proppi.verti} and thus the first part of the proposition is proved.
The second part of the claim immediately follows by noticing that $\mathbb{K}(b,\mathcal{Q},0)$, $\Sigma(b,\mathcal{Q},0)$, and thus every connected component of $C$, are vertically invariant.
\end{proof}

\vv

\begin{proposition}\label{prop:reprez}
Let $\nu\in\mathcal U(n+1)$ be such that $\supp(\nu)\subseteq \mathbb{K}(b,\mathcal{Q},0)$. Then $\nu$ has the representation
\begin{equation}\label{eq:propreprez1}
    \nu(A)=c\int\mathcal{H}^{n-1}_{\eu}\llcorner \pi_H(\supp(\nu))\bigl(A\cap \{x_{n+1}=t\}\bigr)\,dt, \qquad \text{ for }A\subseteq \mathbb P^n,
\end{equation}
where $c$ is as in Proposition \ref{strutturaunifmeasvssurfacemeasure}.
In addition, $\nu$ is invariant under dilations and for any positive Borel function $f\colon \R^n\to [0,\infty]$ and $g\colon \R\to[0,\infty]$ we have
\begin{equation}
      \int f(\pi_H x)g(\pi_T x)\,d\nu(x)=c\int r^{n-2}\Big(\int f(rw)d  \sigma(w)\Big)dr\int g(t)\,dt
      \label{eq:numrepr}
\end{equation}
and
\begin{equation}
\int_{B_1(0)} f(z_H) \,d\nu(z)=c\int_0^1 r^{n-2}\sqrt{1-r^4}\Big(\int f(r u)d\sigma(u)\Big)\,dr,
      \label{eq:numrepr1}
\end{equation}
where $\sigma\coloneqq\mathcal{H}^{n-2}_{\eu}\llcorner \bigl[\pi_H(\supp(\nu))\cap \mathbb{S}^{n-1}\bigr]$. Finally, $\| \sigma\|=8\Gamma(\frac{n+5}{4})/(\sqrt{\pi}\Gamma(\frac{n-1}{4})c)$.
\end{proposition}

\begin{proof}
The identity \eqref{eq:propreprez1} is an immediate consequence of Propositions \ref{strutturaunifmeasvssurfacemeasure}, \ref{prop:invvert}, and the definition of the measure $\sigma_{\supp(\nu)}$.
In order to prove \eqref{eq:numrepr}, we first note that for any couple of Borel subsets $A_H\subseteq \R^n$ and $A_T\subseteq \R$ we have
$$\nu(A_H\times A_T)=c\mathcal{H}^{n-1}_{\eu}(A_H)\mathcal{L}^1(A_T).$$

The product formula above and monotone convergence theorem imply that
$$  \int f(\pi_H x)g(\pi_T x)\,d\nu(x)=c\int f(y)d\bigl[\mathcal{H}^{n-1}_{\eu}\llcorner \pi_H(\supp(\nu))\bigr](y)\int g(t)\,dt.$$

On the other hand, since $\mathbb{K}(0,\mathcal{Q},0)$ and $\Sigma(0,\mathcal{Q},0)$ are invariant under dilations, the set $\mathbb{K}(0,\mathcal{Q},0)\setminus\Sigma(0,\mathcal{Q},0)$ and each of its connected components are invariant under dilations, too. Thus, Proposition \ref{prop:invvert} implies that either $\nu$ is flat, and in this case there is nothing to prove, or $\supp(\nu)$ coincides with the closure of some connected components of $\mathbb{K}(0,\mathcal{Q},0)\setminus\Sigma(0,\mathcal{Q},0)$. In the latter case, $\pi_H(\supp(\nu))$ must be a cone in $\R^n$ and by the classical coarea formula\footnote{It is well known that coarea formula can be extend to general rectifiable sets. We refer, for instance, to \cite[Theorem 3.2.22]{Federer1996GeometricTheory}.} \cite[(10.6) Chapter 2]{Simon1983LecturesTheory} we have
$$  \int f(\pi_H x)g(\pi_T x)\,d\nu(x)=c\int r^{n-2}\Big(\int f(rw)\,d  \sigma(w)\Big)dr\int g(t)\,dt,$$
where $\sigma\coloneqq \mathcal{H}^{n-2}_{\eu}\llcorner \bigl[\pi_H(\supp(\nu))\cap \mathbb{S}^{n-1}\bigr]$, which concludes the proof of \eqref{eq:numrepr}.

The identity \eqref{eq:numrepr1} is proved similarly and the computation of $\| \sigma\|$ follows from \eqref{eq:numrepr1}, indeed
\[1=\nu(B(0,1))=c\| \sigma\|\int r^{n-2}\sqrt{1-r^4}\,dr=c\| \sigma\|\frac{\sqrt{\pi}\Gamma(\frac{n-1}{4})}{8\Gamma(\frac{n+5}{4})}.\qedhere\]
\end{proof}

\vv

\begin{proposition}\label{charunifquadrics}
Let $\nu\in\mathcal U(n+1)$ be such that $\supp(\nu)\subseteq \mathbb{K}(b,\mathcal{Q},0)$. Then either $\nu$ is flat or, up to isometries, there exists $c=c(n)>0$ such that
$$\nu=c\mathcal{H}^3\llcorner \{x_1^2+x^2_2+x_3^2=x^4\}\otimes \mathcal{L}^{n-4}\llcorner\mathrm{span}\{e_4,\ldots,e_n\}\otimes \mathcal{L}^1\llcorner \mathrm{span}\{e_{n+1}\}.$$
\end{proposition}

\begin{proof}
Corollary \ref{prop1} and Proposition \ref{prop:reprez} imply that for any $t>0$ and $x\in\supp(\nu)$ we have 
\begin{equation}
\begin{split}
    \frac{n+1}{4}\Gamma\Big(\frac{n+1}{4}\Big) t^{n+1}=& \int e^{-\frac{\lVert x-y\rVert^4}{t^4}}\,d\nu(y)\\
    =&\int e^{-\frac{\lvert x_H-y_H\rvert^4}{t^4}}\,d\bigl[\mathcal{H}^{n-1}_{\eu}\llcorner \pi_H(\supp(\nu))\bigr](y)\int e^{-\frac{s^2}{t^4}}\,ds\\
      =&\sqrt{\pi}t^2\int e^{-\frac{\lvert x_H-y_H\rvert^4}{t^4}}\,d\bigl[\mathcal{H}^{n-1}_{\eu}\llcorner \pi_H(\supp(\nu))\bigr](y).
\end{split}
    \nonumber
\end{equation}
Hence, \cite[Lemma 3.4]{MR3461027} implies that $\mathcal{H}^{n-1}_{\eu}\llcorner \pi_H(\supp(\nu))$ is a uniform measure in $(\R^n,\lvert\cdot\rvert)$. This, together with \cite[Main Theorem]{Kowalski1986Besicovitch-typeSubmanifolds}, concludes the proof. 
\end{proof}

\vv

\section{Proof of \texorpdfstring{$1$}{1}-codimensional Preiss's theorem}\label{section:proof_main_theorem}

\begin{proposition}\label{conti}
The functional $\mathscr{F}\colon\mathcal{U}(n+1)\to\R$ defined by
 \begin{equation}
        \mathscr{F}(\nu)\coloneqq\inf_{u\in\mathbb{S}^{n-1}} 
        \int \lvert z_H\rvert^4\langle z_H,u\rangle^2
					e^{-\lVert z\rVert^4}\, d\nu(z),
    \end{equation}
is continuous with respect to the weak convergence of measures.
\end{proposition}

\begin{proof}
Let $\{\mu_i\}_i$ be a sequence in $\mathcal U(n+1)$ and suppose that $\mu_i\rightharpoonup \mu$ for some $\mu\in \mathcal U(n+1)$. Let $\mathfrak{m}_i\in\mathbb{S}^{n-1}$ be such that
\[
	\mathscr{F}(\mu_i)=\int \lvert z_H\rvert^4\langle z_H,\mathfrak{m}_i\rangle^2
					e^{-\lVert z\rVert^4}\,d\mu_i(z).
\]
Up to passing to a subsequence, we suppose that $\mathfrak{m}_i$ converges to some $\mathfrak{m}\in\mathbb{S}^{n-1}$. Thus, the function $f_i(z)\coloneqq \lvert z_H\rvert^4\langle z_H,\mathfrak{m}_i\rangle^2e^{-\lVert z\rVert^4}$ converges uniformly to $f_{\mathfrak{m}}(z)\coloneqq\lvert z_H\rvert^4\langle z_H,\mathfrak{m}\rangle^2e^{-\lVert z\rVert^4}$. This implies that
$$\lim_{i\to\infty} \int \lvert z_H\rvert^4\langle z_H,\mathfrak{m}_i\rangle^2e^{-\lVert z\rVert^4}\,d\mu_i(z)= \int \lvert z_H\rvert^4\langle z_H,\mathfrak{m}\rangle^2e^{-\lVert z\rVert^4}\,d\mu(z),$$
from which we infer that $\liminf_{i\to \infty} \mathscr{F}(\mu_i)\geq \mathscr{F}(\mu)$. 

On the other hand, let $\overline{\mathfrak{m}}\in\mathbb{S}^{n-1}$ be such that $\mathscr{F}(\mu)=\int \lvert z_H\rvert^4\langle z_H,\overline{\mathfrak{m}}\rangle^2e^{-\lVert z\rVert^4} d\mu(z)$. Since $\mu_i\rightharpoonup \mu$, we deduce that
$$\lim_{i\to\infty} \int  \lvert z_H\rvert^4\langle z_H,\overline{\mathfrak{m}}\rangle^2e^{-\lambda\lVert z\rVert^4}\,d\mu_i(z)= \int \lvert z_H\rvert^4\langle z_H,\overline{\mathfrak{m}}\rangle^2e^{-\lambda\lVert z\rVert^4}\,d\mu(z).$$
This implies that $\limsup_{i\to\infty} \mathscr{F}(\mu_i)\leq \mathscr{F}(\mu)$, which concludes the proof.
\end{proof}
	
	\vv
	The proof of the following crucial result is based on that of \cite[Proposition  6.10]{DeLellis2008RectifiableMeasures}.
	
	\begin{theorem}\label{th:disconnection}
    Assume that $\phi$ is a Radon measure on $\mathbb P^n$ such that
    $$0<\Theta^{n+1}(\phi,x)\coloneqq \lim_{r\to 0}\frac{\phi(B(x,r))}{r^{n+1}}<\infty\qquad  \text{for $\phi$-almost every $x\in\mathbb{P}^n$}.$$
    We further suppose that  $\mu$ satisfies the property:
    \begin{itemize}
        \item[(\hypertarget{P}{\textbf{P}})] There exists a constant $\newep\label{ep:disc}>0$, depending only on $n$, such that if $\mu\in\mathcal{U}(n+1)$ and $\mathscr{F}(\nu)\leq\oldep{ep:disc}$
    for some $\nu\in \Tan_{n+1}(\mu,\infty)$ then $\mu\in\mathfrak{M}(n+1)$.
    \end{itemize}
    Then $\Tan_{n+1}(\phi,x)\subseteq \Theta^{n+1}(\phi,x)\mathfrak{M}(n+1)$ for $\phi$-almost every $x\in\mathbb P^n$.
\end{theorem}

	\begin{proof}
    We argue by contradiction and suppose that there exists $x\in\mathbb{P}^n$ such that:
\begin{itemize}
\item[(i)] $\Tan_{n+1}(\phi,x)\subseteq\Theta^{n+1}(\phi,x)\mathcal{U}(n+1)$.
\item[(ii)] There are $\zeta,\nu\in\Tan_{n+1}(\phi,x)$ such that $\nu$ is flat and $\zeta$ is not flat.
\item[(iii)] Proposition \ref{tuttitg}-(iv) holds at $x$.
\end{itemize}

Note that, thanks to Propositions \ref{propup}, \ref{tuttitg}-(iv), and \ref{existenceflatmeasure}, the above choice of $x$ does not imply a loss of generality.
We can also assume that $\Theta^{n+1}(\phi,x)=1$. Observe that no tangent to $\zeta$ at infinity can be flat otherwise property \hyperlink{P}{\textbf{P}} would imply that $\zeta$ is flat and violate (ii). Also, by the assumption on the functional $\mathscr{F}$, we have $\mathscr{F}(\chi)>\oldep{ep:disc}$ for any $\chi\in\Tan_{n+1}(\zeta,\infty)$. From now on we fix $\chi\in\Tan_{n+1}(\zeta,\infty)$ and we let $r_k\to0$ and $s_k\to 0$ be such that
\[
	\frac{T_{x,r_k}\phi}{r_k^{n+1}}\rightharpoonup \nu\qquad\text{and}\qquad\frac{T_{x,s_k}\phi}{s_k^{n+1}}\rightharpoonup \chi.
\]
Possibly passing to a subsequence, we also suppose that $s_k<r_k$. Let us define the function
\[
	f(r)\coloneqq\mathscr{F}\bigl(r^{-(n+1)}T_{x,r}\phi\bigr),\qquad r>0,
\]
and note that since $\mathscr{F}$ is continuous on $\mathcal{U}(n+1)$ with respect to the weak-$*$ convergence of measures, $f$ is continuous in $r$.
The flatness of $\nu$ implies
\[
	\lim_{r_k\to 0}f(r_k)=\mathscr{F}(\nu)=0,
\]
thus for $r_k$ small enough we have $f(r_k)<\oldep{ep:disc}$. On the other hand, since
\[
	\lim_{s_k\to 0}f(s_k)=\mathscr{F}(\chi)> \oldep{ep:disc},
\]
for sufficiently small $s_k$ we have $f(s_k)>\oldep{ep:disc}$. Fix $\sigma_k\in [s_k,r_k]$ such that $f(\sigma_k)=\oldep{ep:disc}$ and $f(r)\leq\oldep{ep:disc}$ for $r\in[\sigma_k,r_k]$. By compactness, possibly passing to a subsequence, $\sigma_k^{-(n+1)}T_{x,\sigma_k}\phi$ converges weakly-$*$ to a measure $\xi\in \mathcal{U}(n+1)$. Clearly by continuity
\begin{equation}\label{eq:Fxi_oldep}
	\mathscr{F}(\xi)=\lim_{\sigma_k\to 0} f(\sigma_k)=\oldep{ep:disc}>0.
\end{equation}
We claim that $r_k/\sigma_k\to \infty$;
 otherwise, if for some subsequence not relabeled we had that $r_k/\sigma_k$ converges to a constant $C\geq 1$, we would have
\[
	\nu=\lim_{k\to\infty}\frac{T_{x,r_k}\phi}{r_k^{n+1}}=\lim_{k\to \infty} \Bigl(\frac{\sigma_k}{r_k}\Bigr)^{n+1}T_{0,r_k/\sigma_k}\left(\frac{T_{x,\sigma_k}\phi}{\sigma_k^{n+1}}\right),
\]
which implies that ${C^{-n-1}}{\xi_{0,C}}=\nu$.
In particular $\xi$ would be flat, which is not possible by \eqref{eq:Fxi_oldep}.

Note that for any given $R>0$ it holds
\[
	(R\sigma_k)^{-(n+1)}T_{x,R\sigma_k}\phi\rightharpoonup R^{-(n+1)}T_{0,R}\xi
\]
which, by continuity of $\mathscr{F}$, implies that
\[
	\mathscr{F}\bigl(R^{-(n+1)}T_{0,R}\xi\bigr)=\lim_{k\to\infty}f(R\sigma_k).
\]
Moreover, since $r_k/\sigma_k\to\infty$ we conclude that for any $R>1$ we have that $R\sigma_k\in[\sigma_k,r_k]$ whenever $k$ is large enough. By our choice of $\sigma_k$ and $r_k$, this yields
\begin{equation}
    \mathscr{F}\bigl(R^{-(n+1)}T_{0,R}\xi\bigr)=\lim_{k\to\infty}f(R\sigma_k)\leq \oldep{ep:disc},
    \label{eq1010}
\end{equation}
for every $R\geq 1$. Let $\psi\in\Tan_{n+1}(\xi,\infty)$. Thanks to \eqref{eq1010} and the continuity of $\mathscr{F}$, we infer that
$$\mathscr{F}(\psi)=\lim_{R\to\infty}\mathscr{F}\bigl(R^{-(n+1)}T_{0,R}\xi\bigr)\leq \oldep{ep:disc},$$
and \hyperlink{P}{\textbf{P}} implies that $\xi\in\mathfrak{M}(n+1)$, which contradicts the fact that $\xi$ is \emph{not} flat by \eqref{eq:Fxi_oldep}.
	\end{proof}

We can finally complete the proof of Preiss's Theorem in the parabolic space.

\begin{proof}[Proof of Theorem \ref{theorem:main_theorem_Preiss}]
We prove that (iii) implies (i). The case $n=1$ has already been discussed in Proposition \ref{preiss1}, so we assume $n>1$ and verify that the hypotheses of Theorem \ref{th:disconnection} are satisfied.

Let us check that \hyperlink{P}{\textbf{P}} holds. Defined $\mathscr{F}(\nu)\coloneqq\inf_{u\in\R^n} 
        \int \lvert z_H\rvert^4\langle z_H,u\rangle^2
					e^{-\lVert z\rVert^4}\, d\nu(z)$, we know by Proposition \ref{disconnectdegenerated} that if there exists $\nu\in\Tan_{n+1}(\mu,\infty)$ such that
					$$\int \lvert z_H\rvert^4\langle z_H,u\rangle^2
					e^{-\lVert z\rVert^4}\, d\nu(z)<\oldep{eps:1}/2,$$
					then $\mu\in \mathcal{U}(n+1)\setminus\mathcal{D}\mathcal{U}(n+1)$. However, thanks to Propositions \ref{prop:noT} and \ref{charunifquadrics} we infer that there exists $\tilde\varepsilon$ such that if $\mathscr{F}(\nu)<\tilde\varepsilon$ for some $\nu\in\Tan(\mu,\infty)$, then $\mu\in\mathfrak{M}(n+1)$. If we set $\oldep{ep:disc}\coloneqq\min\{\oldep{eps:1},\tilde{\varepsilon}\}$ the property \hyperlink{P}{\textbf{P}} holds and Theorem \ref{th:disconnection} applies, which concludes the proof.
\end{proof}

\vv

\section{Weak constant density implies bilateral weak geometric lemma}\label{section:BWGL}

Throughout this section we understand again that $\mathbb P^n$ is endowed with the Koranyi norm $d$.
\vv

The precise statements of the WCD and BWGL conditions require some preliminary definitions.
We first recall that a Radon measure $\mu$ on $(\mathbb P^n, d)$ is called $(n+1)$-Ahlfors regular, or simply $(n+1)$-Ahlfors-regular, if there exists $C>0$ such that
\[
    C^{-1}r^{n+1}\leq \mu(B(x,r))\leq C r^{n+1}
\]
for all $x\in \supp(\mu)$ and $0<r<\diam (\supp(\mu)).$ We also refer to $C$ as the Ahlfors-regularity constant of $\mu$.
The Euclidean analogue of the following definition was introduced by David and Semmes in \cite[Definition I.2.55]{DavidSemmes}.

\begin{definition}[WCD]\label{def-WCD}
    Let $\mu$ be an $(n+1)$-Ahlfors-regular measure on $(\mathbb P^n,d)$ with Ahlfors-regularity constant $C>0$. For $\varepsilon>0$ we denote as $G(C,\varepsilon)$ the subset of those $(x,r)\in \supp(\mu)\times (0,+\infty)$ for which there exists a Borel measure $\sigma=\sigma_{x,r}$ on $\mathbb P^n$ such that $\supp(\sigma)=\supp(\mu),$ $\sigma$ is $(n+1)$-Ahlfors-regular with constant $C$, and
    \[
        \bigl|\sigma(B(y,t))-t^{n+1}\bigr|\leq \varepsilon\, r^{n+1}
    \]
    for all $y\in \supp(\mu)\cap B(x,r)$ and all $0<t<r.$
    
    We say that $\mu$ satisfies the \textit{weak constant density condition} (abb. WCD) if there is $C'>0$ such that  $G(C',\varepsilon)^c\coloneqq \bigl(\supp(\mu)\times (0,+\infty)\bigr)\setminus G(C',\varepsilon)$ is a Carleson set for every $\varepsilon>0,$ namely
    \[
        \int_0^R\int_{B(x,R)}\chi_{G(C',\varepsilon)^c}(x,r)\, d\mu(x)\, \frac{dr}{r}\leq C(\varepsilon)R^{n+1}
    \]
    for all $x\in \supp(\mu)$ and all $R>0$.
\end{definition}

\vv

In order to quantify the flatness of a measure, we use Jones' $\beta$-coefficients and their bilateral version.

\begin{definition}\label{def:beta_bil_beta}
    Let $\mu$ be a Radon measure on $\mathbb P^n$ and let $B=B(x,r)$ be a ball in $\mathbb P^n$. We define
    \[
        \beta_\mu(B)\coloneqq \inf_{V\in \Gr(n+1)}\Bigl(\sup_{x\in \supp(\mu)\cap B} \frac{\dist(x,V)}{r}\Bigr)
    \]
    and 
    \[
        b\beta_\mu(B)\coloneqq \inf_{V\in \Gr(n+1)}\Bigl(\sup_{x\in \supp(\mu)\cap B} \frac{\dist(x,V)}{r} +  \sup_{x\in V\cap B} \frac{\dist(x,\supp(\mu))}{r}\Bigr).
    \]
\end{definition}

\vv

\begin{lemma}\label{lemma:beta_sequences}
   Let $\{\mu_j\}_j$ be a sequence in $\mathcal U(n+1)$ which converges weakly to a Radon measure $\mu$ on $\mathbb P^n$, and suppose that $\supp(\mu_j)\cap B\neq \varnothing$ for all $j$ and $\supp(\mu)\cap B\neq \varnothing$. Then
   \[
        \frac{1}{2}\limsup_{j\to \infty}\beta_{\mu_j}(\tfrac{1}{2}B)\leq \beta_\mu(B)\leq 2\liminf_{j\to \infty}\beta_{\mu_j}( 2B)
   \]
   and
   \[
        \frac{1}{2}\limsup_{j\to \infty}b\beta_{\mu_j}(\tfrac{1}{2}B)\leq b\beta_\mu(B)\leq 2\liminf_{j\to \infty}b\beta_{\mu_j}( 2B).
   \]
\end{lemma}
\begin{proof}
    See \cite[Lemma 2.2]{Tolsa_uniform_measures} and the references therein.
\end{proof}

\vv

Given an $(n+1)$-Ahlfors-regular measure $\mu$ on $\mathbb P^n$, as a particular case of \cite{christ} we can construct a family $\mathcal D_\mu$ of Borel subsets of $\supp(\mu)$, the so-called dyadic cubes adapted to $\mu$ (or simply $\mu$-cubes). More specifically $\mathcal D_\mu\coloneqq \bigcup_{j\in \mathbb Z} \mathcal D_\mu^j$, where $\mathcal D^j_\mu$ satisfies the following properties:
\begin{enumerate}
    \item $\supp(\mu)=\bigcup_{Q\in \mathcal D^j_\mu}Q$ for all $j\in \mathbb Z$.
    \item If $Q,Q'\in \mathcal D^j_\mu$, then either $Q\cap Q'=\emptyset$ or $Q=Q'$.
    \item If $Q\in \mathcal D^j_\mu$ and $Q'\in \mathcal D^k_\mu$ with $k\leq j$, then either $Q\subset Q'$ or $Q\cap Q'=\emptyset.$
    \item For all $j\in \mathbb Z$ and $Q\in \mathcal D^j_\mu$ we have
    \[
        2^{-j}\lesssim \diam(Q)\leq 2^{-j}\qquad \text{ and }\qquad \mu(Q)\approx 2^{-j(n+1)}.
    \]
    \item If $Q\in \mathcal D^j_\mu$, there exists $z_Q\in Q$ such that $\dist(z_Q, \supp(\mu)\setminus Q)\gtrsim 2^{-j}.$
\end{enumerate}
\vv

For $Q\in \mathcal D^j_\mu$, the point $z_Q$ is often referred to as its center. We also define $\mathcal D_\mu(Q)\coloneqq \{P\in\mathcal D_\mu: P\subset Q\}$ and $\mathrm{Ch}(Q)\coloneqq \{P\in\mathcal D^{j+1}_\mu: P\subset Q\}$. Let us further denote as $\ell(Q)\coloneqq 2^{-j}$ its ``side-length'' and $B_Q\coloneqq B(z_Q,3\ell(Q))$, so that the $\beta$-numbers associated with $Q$ can be naturally defined as
\[
    \beta_\mu(Q)\coloneqq \beta_\mu(B_Q), \qquad\text{ and }\qquad  b\beta_\mu(Q)\coloneqq \beta_\mu(B_Q).
\]
Observe that $Q$ may belong to $\mathcal D^k_\mu\cap \mathcal D^j_\mu$ for $k\neq j$ so, in order for $\ell(Q)$ to be well-defined, we should identify $Q$ with the couple $(Q,j)$. As customary, this identification will be omitted.

\vv

A family $\mathcal F\subset \mathcal D_\mu$ is called \textit{Carleson family} if there exists $c>0$ such that 
\[
    \sum_{Q\in \mathcal F: Q\subset R}\mu(Q)\leq c\mu(R), \qquad \text{ for every } R\in \mathcal D_\mu.
\]

\begin{definition}[BWGL]\label{def:BWGL}
    We say that an $(n+1)$-Ahlfors-regular measure $\mu$ on $\mathbb P^n$ satisfies the \textit{bilateral weak geometric lemma} (abb. BWGL) if, for each $\eta>0$, the family
    \[
       \mathcal B_\eta\coloneqq\{Q\in \mathcal D_\mu: b\beta_\mu(Q)>\eta\}
    \]
    is a Carleson family.
\end{definition}

The purpose of this section is to prove Theorem \ref{theorem:WCD_implies_BWGL_2}, which we state again below for the reader's convenience.

\begin{theorem}\label{theorem:WCD_implies_BWGL}
    Let $\mu$ be an $(n+1)$-Ahlfors regular measure on $\mathbb P^n$, and assume that it satisfies the weak constant density condition. Then $\mu$ satisfies the bilateral weak geometric lemma.
\end{theorem}

\vv

Let us now introduce the operator which is involved in the touching-point argument.
For $\mu\in \mathcal U(n+1)$, $0<r<s$, and $x\in\mathbb{P}^n$ we define
\begin{equation}\label{eq:definition_Rrs}
    \mathcal{R}_{r,s}\mu(x)\coloneqq \int_{r< \lvert x-y \rvert\leq s} \frac{\lvert x_H-y_H\rvert^2 (x_H-y_H)}{\lVert x-y\rVert^{n+4}}\,d\mu(y).
\end{equation}

\vv

\begin{proposition}\label{prop:propaux1_s}

Let $\mu\in\mathcal{U}(n+1)$, $z\in\supp(\mu)$, and $r>0$. There exists $c=c(n)>0$ such that for all $s>r$ and
all $x\in B(z,r)\cap\supp(\mu)$, we have
$$\Bigl\lvert \Bigl\langle\frac{x_H-z_H}{r},\mathcal{R}_{r,s}\mu(z)\Bigr\rangle\Bigr\rvert\leq c.$$
\end{proposition}

\begin{proof}
Without loss of generality, let us assume that $z=0.$
We then fix $0<r<s$ and let $\vphi\colon [0,\infty)\to\R$ be a non-negative $C^\infty$ function such that
$$\vphi(t) =\begin{cases}
0 & \mbox{if $0\leq t\leq r/2$ or $t\geq 2s$,}\\
t^{-(n+1)}& \mbox{if $r\leq t\leq s$.}
\end{cases}$$

We also ask $\vphi$ to satisfy
\begin{equation}\label{eq:def_aux_funct_tpa}
    |\vphi(t)|\leq t^{-(n+1)}\qquad\mbox{and}\qquad
|\vphi'(t)|\leq c'\min\{r^{-(n+2)}, t^{-(n+2)}\} \qquad\mbox{for all $t\geq 0$,}
\end{equation}
for some $c'>0$ and we define
$$\rho(u) \coloneqq -\int_u^{\infty}\vphi(t)\,dt,\qquad u\in\mathbb R,$$
so that $\rho'(t)=\vphi(t)$. Finally, for $y\in\R^n$ set $\Phi(y) \coloneqq \rho(\lVert y\rVert)$ and notice that $\Phi$ is a radial $C^\infty$ function which is supported on $B(0,2s)$.

For any $x,y\in\mathbb{P}^n$ we let $\gamma_{x,y}(\lambda)\coloneqq \delta_\lambda(x)-y$ for any $\lambda\in [0,1]$, so that Taylor's expansion of $\Phi(\gamma_{x,y}(\lambda))$ yields
\begin{equation}
    \begin{split}
        \Phi(x-y)-\Phi(-y)=&\Phi(\gamma_{x,y}(1))-\Phi(\gamma_{x,y}(0))\\
        =&\frac{d}{d\lambda}\bigl(\Phi(\gamma_{x,y}(\lambda))\bigr)\big\vert_{\lambda=0}+\frac{1}{2}\frac{d^2}{d\lambda^2}\bigl(\Phi(\gamma_{x,y}(\lambda))\bigr)\big\rvert_{\lambda=\sigma},
        \label{eq:taylorphi}
    \end{split}
\end{equation}
for some $\sigma\in [0,1]$ which depends on both $x$ and $y$. The chain rule implies that
\begin{align*}
    \frac{d}{d\lambda}\bigl(\Phi(\gamma_{x,y}(\lambda))\bigr)\big\vert_{\lambda=0}=&D\Phi(-y)[(x_H,0)]
\eqqcolon D_H\Phi(-y)[x_H],\\
\frac{d^2}{d\lambda^2}\bigl(\Phi(\gamma_{x,y}(\lambda))\bigr)\big\rvert_{\lambda=\sigma}=&2\partial_T\Phi(\gamma_{x,y}(\sigma))x_T+D^2\Phi(\gamma_{x,y}(\sigma))\bigl[(x_H,2\sigma x_T),(x_H,2\sigma x_T)\bigr].
\end{align*}

The identities above make \eqref{eq:taylorphi} boil down to
\begin{equation}\label{eq:sec101}
\begin{split}
     \Phi(x-y)-\Phi(-y)=&D_H\Phi(-y)[x_H]\\
    &+\partial_T\Phi(\gamma_{x,y}(\sigma))x_T+\frac{1}{2}D^2\Phi(\gamma_{x,y}(\sigma))\bigl[(x_H,2\sigma x_T),(x_H,2\sigma x_T)\bigr].
\end{split}
\end{equation}
If we assume $x\in \supp(\mu)$, define $\tilde x \coloneqq (x_H,2\sigma x_T)$, and integrate \eqref{eq:sec101}, we get
\begin{equation}
    \begin{split}
        0=&\int (\Phi(x-y)-\Phi(-y))\,d\mu(y)
        =\int D_H\phi(-y)[x_H]\,d\mu(y)\\
        &\qquad +\int \Big(\partial_T\Phi(\gamma_{x,y}(\sigma))x_T+\frac{1}{2}D^2\Phi(\gamma_{x,y}(\sigma))[\tilde x,\tilde x]\Big)\,d\mu(y).
        \label{estimate3}
    \end{split}
\end{equation}
Let us note that by definition of $\Phi$ we have
\begin{equation}\label{eq:D_H_Phi}
    D_H\Phi(z)=\rho^\prime(\lVert z\rVert)\frac{\lvert z_H\rvert^2z_H}{\lVert z\rVert^3}=\vphi(\lVert z\rVert)\frac{\lvert z_H\rvert^2z_H}{\lVert z\rVert^3},
\end{equation}
and thus:
\begin{align}
    D_H\Phi(\lVert z\rVert)=\frac{\lvert z_H\rvert^2z_H}{\lVert z\rVert^{n+4}}\qquad &\text{for } r\leq \lVert z\rVert\leq s,\label{eq:DHphi1}\\
    \lvert D_H\Phi(\lVert z\rVert)\rvert \leq 2^{n+1}r^{-(n+1)}\qquad &\text{for }\lVert z\rVert\leq r,\label{eq:DHphi2}\\
    \lvert D_H\Phi(\lVert z\rVert)\rvert \leq \lVert z\rVert^{-(n+1)}\qquad &\text{for }\lVert z\rVert\geq s.\label{eq:DHphi3}
\end{align}
Thanks to \eqref{eq:DHphi1} and \eqref{eq:definition_Rrs} we deduce that
\begin{equation}
    \int D_H\Phi (-y)\,d\mu(z)=\int_{\lVert z\rVert\leq r} D_H\Phi (-y)\,d\mu(z)-\mathcal{R}_{r,s}\mu(0)+\int_{\lVert z\rVert\geq s} D_H\Phi (-y)\,d\mu(z).
    \label{estimate2}
\end{equation}
Moreover, \eqref{eq:DHphi2}, \eqref{eq:DHphi3}, and the uniformity of $\mu$ imply
\begin{equation}
\begin{split}
      \Bigl\lvert \int_{\lVert z\rVert\leq r} D_H\Phi (-y)\,d\mu(z)+\int_{\lVert z\rVert\geq s} D_H\Phi (-y)\,d\mu(z)\Bigr\rvert\leq& 2^{n+1}\frac{\mu(B(0,r))}{r^{n+1}}+\frac{\mu(B(0,2s))}{s^{n+1}}\\
      =&2^{n+2}.
      \label{estimate1}
\end{split}
\end{equation}

We gather \eqref{estimate3}, \eqref{estimate2},  \eqref{estimate1}, and infer that
\begin{equation}
\begin{split}
     \lvert \langle x_H, \mathcal{R}_{r,s}\mu(x)\rangle\rvert &\leq 2^{n+2}\lvert x_H\rvert+\Bigl\lvert\int \partial_T\Phi(\gamma_{x,y}(\sigma))x_T\,d\mu(y)\Bigr\rvert 
     +\frac{1}{2}\Bigl\lvert\int D^2\Phi(\gamma_{x,y}(\sigma))\bigl[\tilde x,\tilde x\bigr]\,d\mu(y)\Bigr\rvert\\
    &\eqqcolon 2^{n+2}\lvert x_H\rvert + \mathfrak I_1 + \mathfrak I_2.
    \label{equation4}
\end{split}
\end{equation}

Before estimating $\mathfrak I_1$ and $\mathfrak I_2$, we observe that the assumption $\lVert x\rVert\leq r$ and triangle inequality imply that
if $\lVert y\rVert\leq 2r$ then $\lVert\gamma_{x,y}(\sigma)\rVert\leq 3r$. If, on the contrary, $\lVert y\rVert> 2r$ then $\lVert\gamma_{x,y}(\sigma)\rVert>\lVert y\rVert$/2.
Thus, since  $\supp(\varphi)\subseteq B(0,2s)\setminus B(0,r/2)$ and
\begin{equation}\label{eq:D_T_Phi}
    \partial_T\Phi(z)=\vphi(\lVert z\rVert)\frac{ z_T}{2\lVert z\rVert^3},
\end{equation}
we obtain that 
\[
    \lvert\partial_T\Phi(\gamma_{x,y}(\sigma))\rvert\leq 2^{n+2}r^{-(n+2)}\qquad \text{ for }\lVert y\rVert\leq 2r
\]
and 
\[
    \lvert\partial_T\Phi(\gamma_{x,y}(\sigma))\rvert\leq\lVert y\rVert^{-(n+2)}\qquad \text{ for }\lVert y\rVert> 2r.
\]

This readily shows that
\begin{equation}
    \begin{split}
        \mathfrak I_1\leq&\Big\lvert\int_{\lVert y\rVert\leq 2r} \partial_T\Phi(\gamma_{x,y}(\sigma))\, d\mu(y)\Big\rvert |x_T|+\Big\lvert\int_{\lVert y\rVert> 2r} \partial_T\Phi(\gamma_{x,y}(\sigma)) d\mu(y)\Big\rvert|x_T|\\
        \leq &\frac{2^{n+2}\mu(B(0,2r))}{r^{n+2}}\|x\|^2+2^{n+2}\|x\|^2\int_{\lVert y\rVert> 2r}\frac{1}{\lVert y\rVert^{n+2}}\,d\mu(y)\leq 2^{2(n+3)}\frac{\|x\|^2}{r},
        \label{eq:stima_I_1_fin}
    \end{split}
\end{equation}
where the estimate of the last integral can be obtain via a standard decomposition in dyadic annuli of the domain of integration.

We are left with the estimate of $\mathfrak I_2$. We first recall that $\tilde x \coloneqq (x_H,2\sigma x_T)$ and note that
\begin{equation}
\begin{split}
      D^2\Phi(z)[\tilde x,\tilde x ]=\bigl\langle x_H,D_H^2\Phi(z) x_H\bigr\rangle+&4\sigma x_T\langle x_H,\partial_T D_H\Phi(z)\rangle
      +4\sigma^2 x_T^2\partial_T^2\Phi(z).
      \label{equazione5}
\end{split}
\end{equation}
A standard differentiation of \eqref{eq:D_H_Phi} and \eqref{eq:D_T_Phi}, whose details we omit, show that
\begin{align*}
  \bigl\langle x_H, D_H^2\Phi(z)x_H\bigr\rangle=\frac{\varphi'(\lVert z\rVert)\lvert z_H\rvert^4\langle z_H,x_H\rangle^2}{\lVert z\rVert^6}+&\frac{\varphi(\lVert z\rVert)\bigl(2\langle z_H,x_H\rangle^2+\lvert z_H\rvert^2\lvert x_H\rvert^2\bigr)}{\lVert z\rVert^3}\\
  -&\frac{3\varphi(\lVert z\rVert)\lvert z_H\rvert^4\langle z_H,x_H\rangle^2}{\lVert z\rVert^7},\\
  \langle \partial_TD_H\varphi(z),x_H\rangle=\frac{\lvert z_H\rvert^2\langle z_H,x_H\rangle z_T}{2\lVert z\rVert^6}\Big( \varphi'&(\lVert z\rVert)-3\frac{\varphi(\lVert z\rVert)}{\lVert z\rVert}\Big),
\end{align*}
and
\[
  \partial_T^2\Phi(z)=\frac{\varphi'(\lVert z\rVert)z_T^2}{4\lVert z\rVert^6}+\frac{\varphi(\lVert z\rVert)}{2\lVert z\rVert^3}-3\frac{\varphi(\lVert z\rVert)z_T^2}{4\lVert z\rVert^7}.
\]

Thanks to the above identities and for $c'>0$ as in \eqref{eq:def_aux_funct_tpa}, few omitted standard computations yield
\begin{equation}
\begin{split}
    \Bigl\lvert\int_{\lVert y\rVert\leq 2r} D^2\Phi(\gamma_{x,y}(\sigma))[\tilde x,\tilde x]\,d\mu(y)\Bigr\rvert&\leq 2^{n+10}(c'+1)\frac{\mu(B(0,2r))}{r^{n+2}}\lVert x\rVert^2\leq 2^{2(n+10)}(c'+1)\frac{\lVert x\rVert^2}{r}
    \label{stima1}
\end{split}
\end{equation}
and
\begin{equation}
\begin{split}
  & \Bigl\lvert\int_{\lVert y\rVert\geq 2r} D^2\Phi(\gamma_{x,y}(\sigma))[\tilde x,\tilde x]\,d\mu(y)\Bigr\rvert\\
    &\qquad\leq2^{2(n+10)}(c'+1)\Big(\int_{\lVert y\rVert\geq 2r}\lVert y\rVert^{-(n+2)} \,d\mu(y)\lvert x_H\rvert^2+\int_{\lVert y\rVert\geq 2r}\lVert y\rVert^{-(n+3)} d\mu(y)\lvert x_H\rvert\lvert x_T\rvert\\
    &\qquad\qquad\qquad\qquad+\int_{\lVert y\rVert\geq 2r}\lVert y\rVert^{-(n+3)} \,d\mu(y)\lvert x_T\rvert^2\Big)\\
    &\qquad \leq 2^{2(n+11)}(c'+1)\Big(\frac{\lVert x\rVert^2}{2r}+\frac{\lVert x\rVert^3}{8r^2}+\frac{\lVert x\rVert^4}{24r^3}\Big)\leq 2^{2(n+11)}(c'+1)\frac{\lVert x\rVert^2}{r},
    \label{stima2}
\end{split}
\end{equation}
where the last inequality holds because $\lVert x\rVert\leq r$. Hence, \eqref{stima1} and \eqref{stima2} imply
\begin{equation}\label{eq:stima_I_2_fin}
    \mathfrak{I}_2\leq 2^{2(n+12)}(c'+1)\frac{\lVert x\rVert^2}{r}.
\end{equation}

In conclusion, we gather \eqref{equation4}, \eqref{eq:stima_I_1_fin}, \eqref{eq:stima_I_2_fin}, and finally obtain that
\begin{equation}
    \lvert \langle x_H, \mathcal{R}_{r,s}\mu(0)\rangle\rvert\leq 2^{n+2}\lvert x_H\rvert + 2^{2(n+3)}\frac{\|x\|^2}{r} +2^{2(n+12)}(c'+1)\frac{\lVert x\rVert^2}{r}\leq c \,r,
\end{equation}
for $c\coloneqq 2^{2(n+13)} (1+c')$, which completes the proof of the proposition.
\end{proof}

\vv

\begin{lemma}\label{lemmaangolitangenti}
Let $\mu\in \mathcal U(n+1)$. There are two constants $0<\mathfrak{c}_n<1/4$ and $\alpha_n>0$ depending only on $n$ such that for every $y\in\supp(\nu)$ and any $0<r\leq 1$ there exist $z=z(y)\in B(y,r/4)$ and $\rho=\rho(y)>\mathfrak{c}_nr$ such that:
\begin{itemize}
    \item[(i)] $U(z,\rho)\cap \supp(\mu)=\emptyset$ and there exists $w\in\mathbb{P}^n$ such that $w\in\partial B(z,\rho)\cap \supp(\mu)$.
    \item[(ii)]$\lvert z_H-w_H\rvert\geq \alpha_n\lVert z-w\rVert$.
\end{itemize}
\end{lemma}

\begin{proof}
 Up to a translation and a dilation we can assume without loss of generality that $y=0$ and $r=1$. 
Note that if we take $0<\vartheta<1/2$ and let $\mathcal{D}$ be a $(\vartheta/4)$-dense set of points in $B(0,1/4)$, then
\[
    \mathcal L^{n+1}(B(0,1))\Bigl(\frac{1}{4}\Bigr)^{n+2}\leq\mathcal{L}^{n+1}\Bigl(\bigcup_{p\in\mathcal{D}}B(p,\vartheta/2)\Bigr)\leq \mathrm{Card}(\mathcal{D})\mathcal L^{n+1}(B(0,1))\Bigl(\frac{\vartheta}{2}\Bigr)^{n+2},
\]
which implies that $\mathrm{Card}(\mathcal{D})\geq(2\vartheta)^{-(n+2)}
$. 

Let us prove (i). We argue by contradiction and assume that for any $y\in B(0,1/4)$ we have $U(y,\vartheta/8)\cap \supp(\mu)\neq \emptyset$. Then if $\widetilde{\mathcal{D}}$ is a $(\vartheta/4)$-dense set in $B(0,1/4)$ of elements in $\supp(\mu)$ such that $B(p,\vartheta/8)\cap B(p',\vartheta/8)=\emptyset$ for every $p,p'\in\widetilde{\mathcal D}$ with $p\neq p'$, it holds
\begin{equation}
    \begin{split}
        16^{-(n+2)}\vartheta^{-1}\leq& (2\vartheta)^{-(n+2)}(\vartheta/8)^{n+1} \leq \mathrm{Card}\bigl(\widetilde{\mathcal{D}}\bigr)(\vartheta/8)^{n+1}=\sum_{p\in\widetilde{\mathcal D}}\mu\bigl(B(p,\vartheta/8)\bigr)\\
        \leq& \mu\Bigl(\bigcup_{p\in\widetilde{\mathcal{D}}}B(p,\vartheta/2)\Bigr)\leq \mu\Bigl(B\Bigl(0,\frac{1}{4}+\vartheta\Bigr)\Bigr)=\Bigl(\frac{1}{4}+\vartheta\Bigr)^{n+1}\leq 1,
        \label{eq.contradictiontheta}
    \end{split}
\end{equation}
where we also use that $\mu\in \mathcal U(n+1).$
It is readily seen that if $\vartheta< 16^{-(n+2)}$ then \eqref{eq.contradictiontheta} cannot be satisfied. This implies that there exists an open ball $U^\prime$ of radius bigger than $32^{-(n+5)}$ centered at some point of $B(0,1/4)$ such that $\supp(\mu)\cap U^\prime=\emptyset$ and $\partial(U^\prime)\cap \supp(\mu)=\emptyset$. Observe also that $r(U')< 1/4$, because 
otherwise $U'$ would contain $0$, which belongs to $\supp(\mu)$. Thus,
$U'\subseteq B(0,1/2)$ and (i) holds with $\mathfrak{c}_n\coloneqq 32^{-(n+5)}$.

Let us move to the proof of (ii). Let $\tilde B=B(\omega,r(\tilde B))$ and $z\in\supp(\mu)\cap \partial \tilde B$. First, we prove that $z_H-\omega_H\neq 0$. Indeed if, by contradiction, we had $z_H-\omega_H= 0$, then (i) implies that the tangent uniform measure $\nu$ of $\mu$ at $z$ would satisfy
\begin{equation}\label{eq:nusubhspace1}
    \supp(\nu)\subseteq\{\xi\in \mathbb P^n:\xi_T(z_T-\omega_T)\leq 0\}.
\end{equation}

Moreover, by Proposition \ref{bupunifarecones} we know that $\nu$ is also invariant under dilations and by Proposition \ref{CONO} we conclude that either $\nu$ is contained in a quadric since the vertical barycenter $\mathcal{T}_\nu$ is non-zero, or 
\begin{equation}\label{eq:suppnusubplane}
    \supp(\nu)\subseteq \mathbb R^n\times\{0\}.
\end{equation}

However, the alternative \eqref{eq:suppnusubplane} is impossible as it would force $\supp(\nu)$ to have dimension $n$. Thus, Proposition \ref{charunifquadrics} implies that $\nu$ is invariant under vertical translations, which contradicts \eqref{eq:nusubhspace1}.

We are left with the proof of the existence of the constant $\alpha_n$. Let us first assume that $r=1$ and, arguing by contradiction, suppose that there are two sequences $\mu_i\in\mathcal{U}(n+1)$ and $y_i\in \supp(\mu_i)$ such that we can find $z_i\in B(y_i,1/4)$, $\rho_i>\mathfrak{c}_n$, and $w_i\in\supp(\mu_i)$ which satisfy:
\begin{itemize}
    \item[($\alpha$)] $U(z_i,\rho_i)\cap \supp(\mu_i)=\emptyset$ and $w_i\in\partial B(z_i,\rho_i)$.
    \item[($\beta$)] $0<\lvert (z_i)_H-(w_i)_H\rvert\leq i^{-1}\rho_i$.
\end{itemize}

Up to translating $y_i$ at $0$, thanks to Proposition \ref{replica}, and possibly passing to non-relabeled subsequences  we assume that the $\mu_i$ converge in the weak-* topology to some $\eta\in\mathcal{U}(n+1)$, that $\rho_i$ converge to some $\rho>\mathfrak{c}_n$, and that the sequences $y_i,w_i, z_i$ converge to some $y,w\in\supp(\eta)$, and $z\in B(y,1/4)$ respectively. Note that the condition $(\beta)$ implies that $z_H=w_H$ and it is also not hard to see that $U(z,\rho)\cap \supp(\mu)=\emptyset$ and that $w\in \partial B(z,\rho)$. This however contradicts the first part of the proof of (ii). The result for a general $r$ follows by rescaling.
\end{proof}

\vv

\begin{lemma}\label{lemaux1}
Let $\mu\in \mathcal U(n+1)$ and $B$ be a ball centered at  $\supp(\mu)$. 
Suppose also that, for some $\ve>0$ and $0<\kappa<2$, we have
\begin{equation}\label{eq:beta_big}
    \beta_\mu(B(x,r))\geq \ve, \qquad \mbox{for all $x\in B\cap \supp(\mu)$ \,and\, $\kappa \,r(B)\leq r\leq 2r(B)$.}
\end{equation}

For any $M>0$, if $\kappa=\kappa(M,\ve)>0$ is small enough,
there exist $r\in[\kappa\, r(B),\, r(B)]$ and
 $x,z_0\in B\cap\supp(\mu)$ with $|x-z_0|< \kappa\,r(B)$
such that
\begin{equation}\label{eqfo95}
\Bigl|\Bigl\langle\frac{x_H-z_{0,H}}{\kappa\,r(B)},\mathcal{R}_{\kappa r(B),r}\mu(z_0)\Bigr\rangle\Bigr|\geq M.
\end{equation}
\end{lemma}

\begin{proof}Let $\mathfrak{x}\in\supp(\mu)$, $\varrho>0$, and $B\coloneqq B(\mathfrak{x},\varrho)$.
Proposition \ref{lemmaangolitangenti} implies that there exist $z\in B(\mathfrak{x},\varrho/4)$ and  $\rho'>\mathfrak{c}_n\varrho$ such that:
\begin{itemize}    
    \item[($\alpha^\prime$)] $U(z,\rho')\cap \supp(\mu)=\emptyset$ and there exists $w\in\partial B(z,\rho')\cap \supp(\mu)\neq \emptyset$.
    \item[($\beta^\prime$)]$\lvert z_H-w_H\rvert\geq \alpha_n\lVert z-w\rVert\geq \alpha_n\mathfrak{c}_n\varrho$.
\end{itemize}

Let $B'\coloneqq B(z,\lVert z-w\rVert)$. Item ($\beta^\prime$) implies that there exists an hyperplane $L\in\Gr(n+1)$ which is the parabolic blow up of $\partial B'$ at $z$. All the arguments below are invariant under isometries, so we assume without loss of generality that $L=e_1^\perp$, $w=0$, and that 
the set  $U=\{y\in\R^d:\,y_1\geq 0\}$ is the blow up of $B^\prime$ at $0$. Let us further denote $D\coloneqq\mathbb{P}^{n+1}\setminus U$.

Hence, note that in this case $z=(-\lvert z_H\rvert e_1, z_T)$ and that $\rho'\coloneqq r(B')= \lVert z\rVert$, which implies that 
\[
    \mathfrak{c}_nr(B)=\mathfrak{c}_n\varrho<\rho'= r(B')\leq \frac{r(B)}{4}=\frac{\varrho}{4}.
\]

For each $j\geq 0$ and for $0<\kappa\ll 1$ to be chosen later, we denote
\[
    B_j\coloneqq B\bigl(0,(2\varepsilon^{-1}\bigr)^j\kappa \varrho).
\]
By short geometric computations, for any $y\in D\cap B_j\setminus B'$ it holds
\begin{equation}
\begin{split}
   \alpha_n^3|y_1|\overset{(\beta')}{\leq}&\frac{\lvert z_H\rvert^3}{\lVert z\rVert^3}|y_1|\leq \frac{8\lVert y\rVert^2}{\lVert z\rVert}\Big(1+\frac{\lVert y\rVert}{\lVert z\rVert}\Big)^2 .
\end{split}
\label{eqdo24}
\end{equation}
On the other hand, since $\beta_\mu(B_j)\geq\ve$, we infer that
there exists $\zeta\in B_j\cap\supp(\mu)$ such that
\begin{equation}\label{eqj22}
\dist(\zeta,L)\geq \ve\,r(B_j).
\end{equation}
An elementary calculation shows that if $\kappa<\kappa_j\coloneqq \min\{\mathfrak c_n(2\varepsilon^{-1})^{-j}, \mathfrak c_n\alpha_n^3\varepsilon(2\varepsilon^{-1})^{-j}/100\}$, then necessarily $\zeta$ belongs to $U\cap B_j\setminus B'$ otherwise \eqref{eqdo24} and \eqref{eqj22} together would result in a contradiction.

It is immediate to see that if $\kappa$ is chosen as above and $j\geq 1$, then 
$$B\bigl(\zeta,\varepsilon r(B_j)/2\bigr)\cap \bigl(B_{j-1}\cup B(L,\varepsilon r(B_j)/2)\bigr)=\emptyset,$$
and 
$$B\bigl(\zeta, \varepsilon r(B_j)/2\bigr)\subseteq U\cap B_{j+2},$$
so
\begin{equation}\label{eq:Bssubjb1}
  B\bigl(\zeta, \varepsilon r(B_j)/2\bigr)\subseteq  U\cap B_{j+2}\setminus \bigl(B_{j-1}\cup B(L,\varepsilon r(B_j)/2)\bigr).  
\end{equation}
The inclusion \eqref{eq:Bssubjb1} implies that
\begin{equation}\label{eq:Bssubjb2}
    \mu\bigl(U\cap B_{j+2}\setminus \bigl(B_{j-1}\cup B(L,\varepsilon r(B_j)/2)\bigr)\bigr)\overset{}{\geq} \mu\bigl(B(\zeta,\varepsilon r(B_j)/4)\bigr)=(\varepsilon /4)^{n+1}r(B_j)^{n+1}.
\end{equation}
Since $y_1\geq0$ for all $y\in U$, for any $j\geq 1$ we deduce from \eqref{eq:Bssubjb2} that
\begin{equation}\label{eqdo98}
\begin{split}
    \int_{U\cap B_{j+2}\setminus B_{j-1}} \frac{\lvert y_H\rvert^2 y_1}{\lVert y\rVert^{n+4}}\,d\mu(y) &\geq 
\mu\bigl( U\cap B_{j+2}\setminus \bigl(B_{j-1}\cup  B(L,\varepsilon r(B_j)/2)\bigr)\bigr) \,
\frac{\ve^3 \,r(B_j)^3}{8\,r(B_{j-1})^{n+4}}\\
& \geq \Bigl(\frac{\varepsilon}{4}\Bigr)^{n+1}\frac{\varepsilon^3r(B_j)^{n+4}}{8r(B_{j-1})^{n+4}}=\frac{\varepsilon^{2n+8}}{2^{3n+6}}.
\end{split}
\end{equation}
Also, by \eqref{eqdo24} and the uniformity of $\mu$,
\begin{equation}\label{eqdo99}
\begin{split}
\Bigl|\int_{D\cap B_{j}\setminus B_{j-1}} \frac{\lvert y_H\rvert^2 y_1}{\lVert y\rVert^{n+4}}\,d\mu(y)\Bigr| \leq& 
8\mu\bigl(
D\cap B_j\setminus B_{j-1}\bigr) \,
\frac{r(B_{j})^4(1+r(B_j)/r(B'))^2}{r(B')\,r(B_{j-1})^{n+4}}\\
\leq & (2\varepsilon^{-1})^{n+10}\frac{r(B_{j})(1+r(B_j)/r(B'))^2}{r(B')}.
\end{split}
\end{equation}
Fix $N\in \N$, choose \(\kappa<\kappa_{N}\), and denote by 
$\mathcal{R}_{\kappa\,r(B),r_N}^1\mu(0)$ the first coordinate of $\mathcal{R}_{\kappa\,r(B),r_N}\mu(0)$. We apply triangle inequality and write
\begin{align*}
\mathcal{R}_{\kappa\,r(B),r_N}^1\mu(0) & = \sum_{j=1}^N \int_{y\in B_j\setminus B_{j-1}}
\frac{\lvert y_H\rvert^2 y_1}{\lVert y\rVert^{n+4}}\,d\mu(y) \\
& \geq \sum_{j=1}^{N} \int_{U\cap B_{j}\setminus B_{j-1}} \frac{\lvert y_H\rvert^2 y_1}{\lVert y\rVert^{n+4}}\,d\mu(y) 
-  \sum_{j=1}^N \Bigl|\int_{D\cap B_{j}\setminus B_{j-1}} \frac{\lvert y_H\rvert^2 y_1}{\lVert y\rVert^{n+4}}\,d\mu(y)\Bigr|.
\end{align*}
Notice that, by \eqref{eqdo98},
\begin{equation}\label{eq:bound_U_tpa}
    \begin{split}
\sum_{j=1}^{N} \int_{U\cap B_{j}\setminus B_{j-1}} \frac{\lvert y_H\rvert^2 y_1}{\lVert y\rVert^{n+4}}\,d\mu(y)\geq& \frac13
\sum_{j=1}^{N-1} \int_{U\cap B_{j+2}\setminus B_{j-1}} \frac{\lvert y_H\rvert^2 y_1}{\lVert y\rVert^{n+4}}\,d\mu(y)
\geq \frac{N-1}{3}\frac{\varepsilon^{2n+8}}{2^{3n+6}}.
    \end{split}
\end{equation}
On the other hand, from \eqref{eqdo99} and the choice of $r(B_j)$ we derive
\begin{equation}\label{eq:bound_D_tpa}
    \begin{split}
        \sum_{j=1}^N \Bigl|\int_{D\cap B_j\setminus B_{j-1}}\frac{\lvert y_H\rvert^2 y_1}{\lVert y\rVert^{n+4}}\,d\mu(y)\Bigr|
\leq &(2\varepsilon^{-1})^{n+10}\sum_{j=1}^{N-1} \frac{r(B_{j})(1+r(B_j)/r(B'))^2}{r(B')}\\
\leq &4\kappa(2\varepsilon^{-1})^{n+10}\sum_{j=1}^{N-1}(2\varepsilon^{-1})^{j}\bigl(1+4(2\varepsilon^{-1})^{j}\kappa\bigr)^2\\
\leq& 4\kappa(2\varepsilon^{-1})^{n+10+N}\bigl(1+4(2\varepsilon^{-1})^{N-1}\kappa\bigr)^2.
    \end{split}
\end{equation}
The bounds \eqref{eq:bound_U_tpa} and \eqref{eq:bound_D_tpa} imply that
\begin{equation}
    \begin{split}
        \langle e_1 , \mathcal{R}_{\kappa\,r(B),r_N}\mu(0)\rangle =&
\mathcal{R}_{\kappa\,r(B),r_N}^1\mu(0)\\ \geq&\frac{N-1}{3}\frac{\varepsilon^{2n+8}}{2^{3n+6}}-4\kappa(2\varepsilon^{-1})^{n+10+N}\bigl(1+4(2\varepsilon^{-1})^{N-1}\kappa\bigr)^2.
    \end{split}
\end{equation}
Since $\beta_\mu\bigl(B(0,\kappa\,r(B))\bigr)\geq\ve$, there are $x_{1},\ldots,
x_{n}\in\supp(\mu)\cap B(0,\kappa\,r(B))$ such that $\pi_H(x_1),\ldots,\pi_H(x_n)$ span $\R^n$ and 
\begin{equation}
\frac{\mathrm{dist}\bigl(x_j,\mathrm{span}(\{\pi_H(x_i):i\neq j\})\times \mathcal{V}\bigr)}{\kappa r(B)}\geq \varepsilon    
\label{madonnat}
\end{equation}
Indeed, if this was not the case, there would be a plane $\tilde{W}$ for which $\mathrm{dist}(\supp(\mu)\cap B(0,r),\tilde{W}\cap B(0,r))< \varepsilon r$, which would contradict the fact that $\beta_\mu(B(0,\kappa r(B)))\geq \varepsilon$.
Note that since $x_1,\ldots,x_n$ are in generic position thanks to \eqref{madonnat}, there are $a_1,\ldots,a_n\in\R$ such that 
\begin{equation}\label{eq:e1_sum_basis}
   e_1 = \sum_{i=1}^n a_i\,\frac{\pi_H(x_i)}{\kappa \,r(B)}. 
\end{equation}
Let $j=1,\ldots, n$ and let $w_j$ be one of the two unitary vectors in $\mathbb R^n$ such that $(\mathrm{span}(\{\pi_H(x_i):i\neq j\}))^\perp=\mathrm{span}(w_j)$. Such choice of $w_j$, together with \eqref{eq:e1_sum_basis} implies that
$$1\geq \lvert\langle e_1,w_j\rangle\rvert=\Big\lvert \Big\langle\sum_{i=1}^na_i\frac{\pi_H(x_i)}{\kappa r(B)},w_j \Big\rangle\Big\rvert=\frac{\lvert a_j\rvert \lvert \langle w_j,\pi_H(x_j)\rangle\rvert}{\kappa r(B)}\overset{\eqref{madonnat}}{\geq} \lvert a_j\rvert\varepsilon.$$
Therefore, $|a_j|\leq \ve^{-1}$ for every $i=1,\ldots,n$. Observe also that if we further impose that
\[
    \kappa < \min\Bigl\{\kappa_{N},\frac{1}{4}\Bigl(\frac{\varepsilon}{2}\Bigr)^{N-1}, \frac{1}{100}\frac{\varepsilon^{3n+18+N}}{2^{4n+16+N}}\Bigr\}\eqqcolon \tilde \kappa_N,
\]
then it holds
\begin{equation}\label{eq:choice_coeff_2}
   4\kappa(2\varepsilon^{-1})^{n+10+N}\bigl(1+4(2\varepsilon^{-1})^{N-1}\kappa\bigr)^2\leq \frac{1}{3}\frac{\varepsilon^{2n+8}}{2^{3n+6}}.
\end{equation}
 This finally implies that 
\begin{equation}
    \begin{split}
     &\sum_{i=1}^na_i\Bigl|\Bigl\langle\frac{\pi_H(x_{i})}{\kappa \,r(B)}, \mathcal{R}_{\kappa\,r(B),r}\mu(0)\Bigr\rangle\Bigr| \geq\left|\langle e_1, \mathcal{R}_{\kappa\,r(B),r}\mu(0)\rangle \right| \\
       & \qquad \geq \frac{N-1}{3}\frac{\varepsilon^{2n+8}}{2^{3n+6}}-4\kappa(2\varepsilon^{-1})^{n+10+N}\bigl(1+4(2\varepsilon^{-1})^{N-1}\kappa\bigr)^2\overset{\eqref{eq:choice_coeff_2}}{\geq} \frac{N-2}{3}\frac{\varepsilon^{2n+8}}{2^{3n+6}},
    \end{split}
\end{equation}
thanks to the choice of $\kappa<\tilde \kappa_N$. Thus, there exists $i\in\{1,\ldots,n\}$ such that  $x_{i}\in\supp(\mu)\cap B(0,\kappa\,r(B))$ and
$$\Bigl|\Bigl\langle\frac{x_{i}}{\kappa \,r(B)}, \mathcal{R}_{\kappa\,r(B),r}\mu(0)\Bigr\rangle\Bigr| \geq 
\frac{N-2}{3n}\frac{\varepsilon^{2n+9}}{2^{3n+6}}.$$
If $N$ is taken big enough, then \eqref{eqfo95} follows.
\end{proof}

\vv

\begin{corollary}\label{corollary:flat_ball_big}
Let $\mu\in \mathcal U(n+1)$. For any $\varepsilon>0$ there exists $\tau=\tau(n,\varepsilon)>0$ such that, if $B$ is a ball centered at $\supp(\mu)$, there exists another ball $B'$ centered at $B\cap \supp(\mu)$ such that \(\beta_\mu(B')\leq \varepsilon\) and $r(B')\in [\tau r(B),2r(B)].$
\end{corollary}
\begin{proof}
    Let us assume that $B$ is centered at $z\in \supp(\mu).$ Let $c=c(n)>0$ be as in Proposition \ref{prop:propaux1_s}, so that for $0<\kappa <1$, every $r\in [\kappa r(B), r(B)]$, and every $x\in B(z,\kappa r(B))$ we have that
    \begin{equation}\label{eq:105_1}
        \Bigl|\Bigl\langle \frac{x_H-z_H}{\kappa r(B)}, \mathcal R_{\kappa r(B), r}\mu(z)\Bigr\rangle\Bigr|\leq c.
    \end{equation}
    In particular, for $M=2c$ and $\kappa=\kappa(M, \varepsilon)$ as in  Lemma \ref{lemaux1}, the hypothesis \eqref{eq:beta_big} is \textit{not} satisfied, otherwise \eqref{eq:105_1} would contradict \eqref{eqfo95}. Thus, there exists $B'$ centered at $B\cap \supp(\mu)$ such that \(\beta_\mu(B')\leq \varepsilon\) and $r(B')\in [\kappa r(B), 2r(B)]$, and the corollary holds for $\tau =\kappa$.
\end{proof}

\vv

The proofs of Lemmas \ref{lemma:lem_bilateral_beta_1} and \ref{lemma:lem_bilateral_beta_2} follow the scheme of those of \cite[Lemma 3.5, Lemma 3.6]{Tolsa_uniform_measures}. Since Tolsa used the deep results of the solution of the Euclidean density problem in \cite{Preiss1987GeometryDensities}, we provide all the details.

\begin{lemma}\label{lemma:lem_bilateral_beta_1}
    Let $\mu\in \mathcal U(n+1).$ For any $\varepsilon>0$ there exists $\delta=\delta(\varepsilon,n)>0$ such that for $x\in \supp(\mu)$ and $r>0$, if $\beta_\mu(B(x,\delta^{-1}r))\leq \delta^2$, then $b\beta_\mu(B(x,r))\leq \varepsilon.$
\end{lemma}
\begin{proof}
    We argue by contradiction and assume that the conclusion of the lemma does not hold. Hence, there exist $\varepsilon>0$ and a sequence $\{\mu_j\}_{j\in \mathbb N}\subseteq \mathcal U(n+1)$ such that $\beta_{\mu_j}(jB_j)\leq j^{-2}$ and $b\beta_{\mu_j}(B_j)\geq \varepsilon.$ Without loss of generality, we further assume $B_j=B(0,1)$ for all $j\in \mathbb N$, observe that $0\in \supp(\mu_j)$ for all $j\in \mathbb N$ and, possibly passing to a subsequence, we have that $\mu_j$ converges weakly to some $\nu\in \mathcal U(n+1).$
    For $1\leq k\leq j$, the assumptions on $\mu_j$ and the basic properties of $\beta$-numbers yield
    \[
        \beta_{\mu_j}(B(0,k))\leq \frac{j}{k}\beta_{\mu_j}(B(0,j))\leq \frac{1}{jk}\to 0 \qquad \text{ as }j\to \infty.
    \]
    Thus, by Lemma \ref{lemma:beta_sequences} it follows that $\beta_\nu (B(0,k/2))=0$ for all $k\geq 1$, which implies that $\supp(\nu)\subseteq L$ for some $L\in \Gr(n+1).$ Then, by Proposition \ref{verticalsamoa}, we have that $\nu\in \mathfrak M(n+1).$
    
    However, the assumption $b\beta_{\mu_j}(B(0,1))\geq \varepsilon$ and Lemma \ref{lemma:beta_sequences} also imply that $b\beta(B(0,2))\geq \varepsilon/2$, which contradicts the flatness of $\nu$ and finishes the proof.
\end{proof}

\begin{lemma}\label{lemmaF->beta}
There exist $\Lambda>0$ and $\newep\label{ep:betainfinito}>0$ depending only on $n$ such that, if  $\mu\in\mathcal{U}(n+1)$ and for some $\nu\in\Tan_{n+1}(\mu,\infty)$ we have $\beta_\nu(B(0,\Lambda))<\oldep{ep:betainfinito}$, then $\mu$ is flat.  
\end{lemma}

\begin{proof}
In Section \ref{section:proof_main_theorem} we showed that there exists $\oldep{ep:disc}=\oldep{ep:disc}(n)>0$ such that if $\mu\in\mathcal{U}(n+1)$ and for some $\nu\in\Tan_{n+1}(\mu,\infty)$ it holds \begin{equation}
    \inf_{u\in\mathbb{S}^{n-1}} 
        \int \lvert z_H\rvert^4\langle z_H,u\rangle^2
					e^{-\lVert z\rVert^4}\, d\nu(z)\leq\oldep{ep:disc},
	\label{eq:piccoloallorapiatto}
\end{equation}
then $\mu\in\mathfrak{M}(n+1)$.

Let us note that for any $\Lambda>0$ and $u\in \mathbb S^{n-1}$ we have
\begin{equation*}
    \begin{split}
        &\int \lvert z_H\rvert^4\langle z_H,u\rangle^2 e^{-\lVert z\rVert^4}\, d\nu(z)\\
        &\qquad\qquad\leq \int_{B(0,\Lambda)} \lvert z_H\rvert^4\langle z_H,u\rangle^2 e^{-\lVert z\rVert^4}\, d\nu(z)+ \int_{\mathbb R^{n+1}\setminus B(0,\Lambda)} \lvert z_H\rvert^4\langle z_H,u\rangle^2 e^{-\lVert z\rVert^4}\, d\nu(z)\\
        &\qquad\qquad\leq \Lambda^4\int_{B(0,\Lambda)}\langle z_H,u\rangle^2 \, d\nu(z) + e^{-\frac{\Lambda^4}{2}}\int_{\mathbb R^{n+1}\setminus B(0,\Lambda)}\|z\|^6 e^{-\frac{\|z\|^4}{2}}\, d\nu(z)\eqqcolon I_u + II.
    \end{split}
\end{equation*}

By the definition of $\beta_\nu$-numbers and since $\nu$ is a uniform measure, is easy to see that
\begin{equation}\label{eq:estimIu_unif1}
    \inf_{u\in \mathbb S^{n-1}}I_u \leq \Lambda^{n+5} \beta_{\nu}(B(0,\Lambda))^2.
\end{equation}
Moreover, by Corollary \ref{prop1} we have
\begin{equation}\label{eq:estimII_unif2}
    II\leq  e^{-\frac{\Lambda^4}{2}}\int \|z\|^6 e^{-\frac{\|z\|^4}{2}}\, d\nu(z)= (n+1)2^{\frac{n-1}{4}}\Gamma\Bigl(\frac{n+7}{4}\Bigr)e^{-\frac{\Lambda^4}{2}}\eqqcolon \tilde c_n\, e^{-\frac{\Lambda^4}{2}}.
\end{equation}

Thus, \eqref{eq:estimIu_unif1} and \eqref{eq:estimII_unif2} imply that for $\Lambda= -2\log\big(\oldep{ep:disc}/(2\tilde c_n)\big)$ and 
$$\beta_\nu(B(0,\Lambda))\leq \frac{\oldep{ep:disc}}{-4\log\big(\frac{\oldep{ep:disc}}{2\tilde c_n}\big)},$$
the bound \eqref{eq:piccoloallorapiatto}
holds, which implies that $\mu\in \mathfrak M(n+1)$.
\end{proof}

\vv

\begin{lemma}\label{lemma:lem_bilateral_beta_2}
    Let $\mu\in \mathcal U(n+1)$ and $\varepsilon>0.$ There exist $\delta_0=\delta_0(n)>0$ and an integer $N=N(\varepsilon,n)>0$ such that if $B$ is a ball centered at $\supp(\mu)$ such that
    \[
        \beta_\mu(2^kB)\leq \delta_0\qquad \text{ for }1\leq k\leq N,
    \]
    then $b\beta_\mu(B)\leq \varepsilon.$
\end{lemma}
\begin{proof}
    By Lemma \ref{lemma:lem_bilateral_beta_1}, possibly by adjusting the values of $\varepsilon$ and $N$, it is enough to prove that $\beta_\mu(B)\leq \varepsilon.$ We argue by contradiction: let us assume that there exist a sequence of $\{\mu_j\}_{j\geq 1}\subset \mathcal U(n+1)$ and balls $B_j(x_j,r_j)$ centered in $\supp(\mu_j)$ such that for any $j$ it holds
    \[
        \beta_{\mu_j}(2^jB_j)\leq \delta_0 \qquad \text{ for } 0\leq k\leq j \qquad \text{ but }\qquad \beta_{\mu_j}(B_j)\geq \varepsilon.
    \]
    
    For $j\geq 1$ we also define the measure $\tilde\mu_j\coloneqq \mu(B_j)^{-1}T_{x_j, r_j}\mu_j.$ Observe that, for every $j\geq 1$, we have that $T_{x_j,r_j}(B_j)=B(0,1)$, $\tilde \mu_j\in \mathcal U(n+1),$ and $0\in \supp(\tilde \mu_j).$ Moreover, possibly by passing to a subsequence, $\mu_j$ converges weakly to some $\nu\in \mathcal U(n+1)$. By Lemma \ref{lemma:beta_sequences}, we have that
    \begin{equation}\label{eq:beta_nu_geq_ve}
        \beta_\nu(B(0,2))\geq \frac{1}{2}\limsup_{j\to \infty}\beta_{\tilde \mu_j}(B(0,1))=\frac{1}{2}\limsup_{j\to \infty }\beta_{\mu_j}(B_j)\geq \frac{1}{2}\varepsilon,
    \end{equation}
    which in particular implies that $\nu \not \in \mathfrak M(n+1).$ Moreover, for $k\in \mathbb N$ to be chosen, 
    \[
        \beta_\nu(B(0,2^{k+\ell}))\leq 2 \liminf_{j\to \infty}\beta_{\tilde \mu_j}(B(0,2^{k+\ell+1}))=2 \liminf_{j\to \infty}\beta_{\tilde \mu_j}(2^{k+\ell+1}B_j) \leq 2\delta_0
    \]
    so that Lemma \ref{lemma:beta_sequences} implies that there exists a tangent measure $\lambda$ to $\nu$ at $\infty$ for which
    \begin{equation}\label{eq:bound_beta_above}
        \beta_{\lambda}(B(0,2^k))\leq 2\liminf_{j\to \infty}\beta_\nu\bigl(B(0,2^{k+j})\bigr)\leq 4\delta_0.
    \end{equation}
    
    Let $\Lambda$ and $\oldep{ep:betainfinito}$ be as in Lemma \ref{lemmaF->beta}. We fix $k$ such that $2^{k-1}<\Lambda\leq 2^k,$ where $\Lambda$ and chose $\delta_0\leq \oldep{ep:betainfinito}/4$, so that \eqref{eq:bound_beta_above} gives \(\beta_\lambda (B(0,\Lambda))\leq \oldep{ep:betainfinito}\), which in turn implies that $\nu\in \mathfrak M(n+1)$ by Lemma \ref{lemmaF->beta}. This contradicts \eqref{eq:beta_nu_geq_ve} and finishes the proof of the lemma.
\end{proof}

\vv

Let $B$ be a (parabolic) ball in $\mathbb P^n$ and $\nu_1, \nu_2$ be two Radon measures on $\mathbb P^n$ such that $B\cap \supp(\nu_1)\neq \emptyset$ and  $B\cap \supp(\nu_2)\neq \emptyset$. We define
\[
    d_B(\nu_1,\nu_2)\coloneqq \sup_{x\in B\cap \supp(\nu_1)}\dist(x,\supp\nu_2) + \sup_{x\in B\cap \supp(\nu_2)}\dist(x,\supp\nu_1).
\]

For an $(n+1)$-Ahlfors-regular measure $\mu$ on $\mathbb P^n$ and $\varepsilon>0$ we denote by $\mathcal N_0(\varepsilon)$ the family of balls $B\subset \mathbb P^n$ for which there exists $\nu\in \mathcal U(n+1)$ such that 
\[
    d_B(\mu,\nu)\leq \varepsilon r(B),
\]
and we also define
\[
    \mathcal N(\varepsilon)\coloneqq \{Q\in\mathcal D_\mu: B_Q\in \mathcal N_0(\varepsilon)\}.  
\]

\vv

The following result was originally stated in Euclidean spaces, but the proof repeats verbatim in $\mathbb P^n.$

\begin{proposition}[see \cite{DavidSemmes}, Chapter III.5]\label{prop:proposition_DS_III_5}
    Let $\mu$ be an $(n+1)$-Ahlfors-regular measure on $\mathbb P^n$.
    If $\mu$ satisfies the WCD, then $\mathcal D_\mu\setminus \mathcal N(\varepsilon)$ is a Carleson family for every $\varepsilon>0.$
\end{proposition}

\vv
We can finally prove the main result of this section.

\begin{proof}[Proof of Theorem \ref{theorem:WCD_implies_BWGL}]
Since we showed Corollary \ref{corollary:flat_ball_big}, the theorem can be proved via the same stopping time argument of \cite[Theorem 1.1]{Tolsa_uniform_measures}, which we sketch here for the reader's convenience.

A collection $\mathcal T\subset \mathcal D_\mu$ is said to be a \textit{tree} if the following properties hold:
\begin{itemize}
    \item There exists $Q(\mathcal T)\in \mathcal T$, maximal with respect to the inclusion of sets, such that $Q'\subset Q(\mathcal T)$ for all $Q'\in \mathcal T.$
    \item If $Q,Q'\in \mathcal T$ and $Q\subset Q'$, then $R\in \mathcal T$ for all $R\in \mathcal D_\mu$ such that $Q\subset R\subset Q'$.
    \item If $Q\in \mathcal T$, 
    then either all $\mu$-cubes in $\mathrm{Ch}(Q)$
    belong to $\mathcal T$ or none of them do.
\end{itemize}
Given a tree $\mathcal T\subset \mathcal D_\mu$, the $\mu$-cube $Q(\mathcal T)$ is also called root. Finally, we denote as $\textrm{Stop}(\mathcal T)$ the set of $\mu$-cubes whose children do not belong to $\mathcal T.$

Fix $R\in \mathcal D_\mu$. Now we partition $\mathcal D_\mu(R) \cap \mathcal N(\varepsilon)$ into trees $\mathcal T_i$, $i\in I\subset \mathbb N$, inductively according to the following stopping time argument:
\begin{itemize}
    \item Let $Q_1$ be an element of $\mathcal D_\mu\cap \mathcal N(\varepsilon)$ with maximal side-lenght.
    \item $\mathcal T_1$ is defined such that $Q(\mathcal T_1)=Q_1$ and, if $Q\in \mathcal T_1$ and $Q'\in \mathcal N(\varepsilon)$ for all $Q'\in \textrm{Ch}(Q)$, then $\mathrm{Ch}(Q)\subset \mathcal T_1$.
    \item Assume $\mathcal T_i$ are defined for $i=1,\ldots, k-1,$ and choose $Q_k$ to be a $\mu$-cube of maximal side length which belongs to $(\mathcal D_\mu(R)\cap \mathcal N(\varepsilon))\setminus \bigcup_{1\leq i\leq (k-1)}\mathcal T_i\neq \emptyset$. The tree $\mathcal T_k$ is defined such that $Q(\mathcal T_k)=Q_k$ and, if $Q\in \mathcal T_k$ and $Q'\in \mathcal N(\varepsilon)$ for all $Q'\in \textrm{Ch}(Q)$, then $\mathrm{Ch}(Q)\subset \mathcal T_k$. For $i\in I$, we denote as $\textrm{ps}(Q(\mathcal T_i))$ a parent or sibling of $Q(\mathcal T_i)$ that does not belong to $\mathcal N(\varepsilon)\cap \mathcal D_\mu(R).$
\end{itemize}


We remark that $\mathrm{Ch}(Q)$ has a bounded multiplicity for all $Q\in \mathcal D_\mu$.
This, together with the Carleson property of $\mathcal D_\mu\setminus \mathcal N(\varepsilon)$ of Proposition \ref{prop:proposition_DS_III_5} gives that there exist $c_0, c'_0>0$ such that
\begin{equation}\label{eq:pack_1}
    \begin{split}
            \sum_{i\in I}\mu(Q(\mathcal T_i))&\leq \sum_{i\in I}\mu(\textrm{ps}(Q_i))=\mu\bigl(\textrm{ps}(Q(\mathcal T_1))\bigr) + \sum_{i\in I, i\neq 1}\mu\bigl(\textrm{ps}(Q(\mathcal T_i))\bigr)\\
            &\leq c_0 \,\mu(R) + c_0\,\sum_{Q\in \mathcal D_\mu(R)\setminus \mathcal N(\varepsilon)}\mu(Q)\leq c'_0\,\mu(R).
    \end{split}
\end{equation}

We also claim that there exists $c_1>0$ such that
\begin{equation}\label{eq:claim_bwgl}
    \sum_{Q\in \mathcal T_i\cap \mathcal B_\eta}\mu(Q)\leq c_1\, \mu(Q(\mathcal T_i))\qquad \text{ for all }i\in I.
\end{equation}

We observe that \eqref{eq:claim_bwgl} allows us to finish the proof of the theorem. Indeed, if $c_2>0$ denotes the Carleson constant that we obtain from Proposition \ref{prop:proposition_DS_III_5}, it implies
\begin{equation*}
    \begin{split}
        \sum_{Q\in \mathcal D_\mu(R)\cap \mathcal B_\eta}\mu(Q)& \leq \sum_{Q\in \mathcal D_\mu(R)\setminus \mathcal N(\varepsilon)} \mu(Q) + \sum_{Q\in \mathcal D_\mu(R)\cap \mathcal N(\varepsilon)\cap \mathcal B_\eta} \mu(Q)\\
        &\leq c_2\, \mu(R) + \sum_{i\in I}\sum_{Q\in \mathcal T_i\cap \mathcal B_\eta}\mu(Q)\overset{\eqref{eq:claim_bwgl}}{\leq } c_2\, \mu(R) + c_1\sum_{i\in I}\mu(Q(\mathcal T_i))\overset{\eqref{eq:pack_1}}{\leq } c\, \mu(R),
    \end{split}
\end{equation*}
 for $c\coloneqq c_2 + c_1\,c'_0.$

We are left with showing \eqref{eq:claim_bwgl}.
By \cite[Lemma 4.4]{Tolsa_uniform_measures}, whose proof repeats verbatim in $\mathbb P^n$, there exist $\varepsilon_1>0$ and $\delta_1>0$ both small enough depending on $\eta$ such that, if $\varepsilon<\varepsilon_1$ and $Q,P\in \mathcal D_\mu$ with $P\subset Q$ are such that
\[
    S\in \mathcal N(\varepsilon)\qquad \text{ for all } S\in \mathcal D_\mu \text{ with }P\subset S\subset Q
\]
and $b\beta_\mu(Q)\leq \delta_1,$ then $b\beta_\mu(P)\leq \eta.$

Let $i\in I$, $\delta_1>0$ as above, and denote as $\mathcal F_i$ the collection of $Q\in \mathcal T_i$ maximal with respect to the inclusion such that $b\beta_\mu(Q)\leq \delta_1$.
In particular, if $P\subset Q$ for some $P\in \mathcal F_i$, then $b\beta_\mu(P)\leq \eta.$ So, if we define
\[
    \mathcal H_i\coloneqq \{Q\in \mathcal T_i: Q\not\subset P\text{ in any }P\in \mathcal F_i\}\supset \mathcal T_i\cap \mathcal B_\eta,
\]
it holds
\begin{equation}\label{eq:auxbwgl2}
    \sum_{Q\in \mathcal T_i\cap \mathcal B_\eta}\mu(Q)=\sum_{Q\in \mathcal H_i}\mu(Q).
\end{equation}

For $Q\in \mathcal H_i$ we pick $P\in \mathcal F_i\cup \mathrm{Stop}(\mathcal T_i)$ as a $\mu$-cube contained in $Q$ with maximal side-lenght, and we denote $f(Q)\coloneqq P$. By \cite[Lemma 4.2]{Tolsa_uniform_measures} we have $\ell(P)\geq \tau\ell (Q)$ provided that $\varepsilon$ is small enough. For any $\mu$-cube $P$, the number of $\mu$-cubes $Q$ such that $f(P)=Q$ is bounded from above by a constant depending on $n,\tau$, and the Ahlfors-regularity of $\mu$.

Hence,  using the fact that both $\mathcal F_i$ and $\mathrm{Stop}(\mathcal T_i)$ are a family of pairwise disjoint $\mu$-cubes, there are $c'_2, c_2''>0$ depending on the constants above such that it holds
\begin{equation}\label{eq:auxbwgl3}
    \begin{split}
        \sum_{Q\in \mathcal H_i}\mu(Q)&\leq c'_2 \sum_{Q\in \mathcal H_i} \mu(f(Q))\leq c_2'' \sum_{P\in \mathcal F_i\cup \mathrm{Stop}(\mathcal T_i)}\mu(P)\\
        &\leq c''_2\sum_{P\in \mathcal F_i}\mu(P)+ c''_2\sum_{P\in \mathrm{Stop}(\mathcal T_i)}\mu(P)\leq c_2\,\mu(Q(\mathcal T_i)),
    \end{split}
\end{equation}
where $c_2\coloneqq 2\, c_2''.$
Thus, the claim \eqref{eq:claim_bwgl} follows from \eqref{eq:auxbwgl2} and \eqref{eq:auxbwgl3}.
\end{proof}

\vv

\appendix

\section{Proof of Theorem \ref{th:failurepreiss}}\label{section:counterexample_PDT}

Let $f\colon \R\to\R$ be a $1/2$-H\"older function with unitary H\"older constant, namely
\begin{equation}\label{eq:halfholdf}
	|f(s)-f(t)|\leq |s-t|^{1/2}, \qquad \text{ for }s,t\in \R.
\end{equation}
We define $\Gamma\coloneqq \{(f(t),t):t\in \R\}\subset \mathbb{P}^1$ and denote by $\mathcal C^2$ the $2$-dimensional centered Hausforff measure on $\mathbb{P}^1$. For $x\in\mathbb{P}^1$ and $r>0$, let $\mathfrak B_\infty(x,r)\coloneqq \{y\in\mathbb{P}^1: \|x-y\|_\infty\leq r\}$ denote the ball associated with the $\|\cdot\|_\infty$-norm.
The condition \eqref{eq:halfholdf} implies that, for $s\in \R$,
\begin{equation*}
	\begin{split}
		\Gamma\cap \mathfrak B_\infty\bigl((f(s),s),r\bigr)&=\bigl\{(f(\rho),\rho):\|(f(\rho),\rho)-(f(s),s)\|_\infty \leq r\bigr\}= \bigl\{(f(\rho),\rho):|\rho-s|\leq r^2\bigr\}.
	\end{split}
\end{equation*}

\begin{lemma}\label{lemma:unif_measure_box_norm}
	Let $\Gamma$ be as above. We have
	\[
		\mathcal C^2\llcorner \Gamma(\mathfrak B_\infty(x,r))=2r^2, \qquad x\in \Gamma, r>0.
	\]
	In particular, the measure $\mathcal C^2\llcorner \Gamma$ is a uniform measure on $\mathbb{P}^1$.
\end{lemma}
\begin{proof}
	Let $x\in \Gamma$, $r>0$, and consider $E\subseteq \mathfrak B_\infty(x,r)\cap \Gamma$. Let us denote as $ S(y,\rho)\coloneqq \{z\in \mathbb{P}^1:|z_T-y_T|\leq \rho^2\}$, for $y\in \mathbb{P}^1$ and $\rho>0$, the infinite horizontal strip. Since $E\subseteq \mathfrak B_\infty(x,r)\cap \Gamma$, we have that
	\begin{equation}\label{eq:BcapE_ScapE}
		 \mathfrak B_\infty(y,\rho)\cap E = S(y, \rho)\cap E\qquad \text{ for all }y\in E, \, \rho>0.
	\end{equation}
	Thus, by Definition \ref{Hausdro}, it holds
	\begin{equation*}
		\begin{split}
			\mathcal C^2\llcorner \Gamma(\mathfrak B_\infty(x,r))&=\sup_{E\subseteq \mathfrak B_\infty(x,r)\cap \Gamma}\sup_{\delta>0}\,\inf\Bigl\{\sum_{j=1}^\infty r_j^2: E\subseteq \bigcup_{j=1}^\infty \mathfrak B_\infty(x_j,r_j) \text{ with }x_j\in E, r_j\leq \delta\Bigr\}\\
			& \overset{\eqref{eq:BcapE_ScapE}}{=}\sup_{E\subseteq \mathfrak B_\infty(x,r)\cap \Gamma}\sup_{\delta>0}\,\inf\Bigl\{\sum_{j=1}^\infty r_j^2: E\subseteq \bigcup_{j=1}^\infty S(x_j,r_j) \text{ with }x_j\in E, r_j\leq \delta\Bigr\}\\
			& = \sup_{E\subseteq \mathfrak B_\infty(x,r)\cap \Gamma}\mathcal L^1(\pi_V(E))=\mathcal L^1\bigl(\pi_V(\mathfrak B_\infty(x,r))\bigr)=2r^2.\qedhere
		\end{split}
	\end{equation*}
\end{proof}
\vv

It is known that there exist $1/2$-H\"older Lipschitz graphs that do not admit any flat parabolic blowup. In order to see this we refer for instance to \cite{JNV21}, where the authors construct a non differentiable intrinsic Lipschitz map in $\mathbb H^1\times \R$. Hence, if $\Gamma$ denotes the graph associated to such a map, a combination of Lemma \ref{lemma:unif_measure_box_norm} and Mattila's Theorem \ref{theorem:Mattila_parabolic_rectifiability} shows that $\mathcal C^2\llcorner \Gamma$ is a uniform measure with respect to $\|\cdot\|_\infty$ which is \textit{not} $\mathscr{P}_2$-rectifiable.

Note also that if $\mathrm{gr}(\Gamma)$ is endowed with the metric induced by $\lVert\cdot\rVert_\infty$, the pointed measured Gromov-Hausdorff tangents of $\bigl(\mathrm{gr}(\Gamma),\lVert\cdot\rVert_\infty,\mathcal{H}^2_{\lVert\cdot\rVert_\infty}\llcorner \Gamma\bigr)$ (see \cite{pmghtangents}) are unique $\mathcal{H}^2_{\lVert\cdot\rVert_\infty}$-almost everywhere. Moreover, if $\mathrm{gr}(\Gamma)$ is endowed with the metric induced by the Koranyi distance, then the pointed measured Gromov-Hausdorff tangents of the measured metric space $\bigl(\mathrm{gr}(\Gamma),\lVert\cdot\rVert,\mathcal{H}^2_{\lVert\cdot\rVert}\llcorner \Gamma\bigr)$ are not unique at $\mathcal{H}^2_{\lVert\cdot\rVert}\llcorner \Gamma$-almost every $x\in\Gamma$, although the measures $\mathcal{H}^2_{\lVert\cdot\rVert}\llcorner \Gamma$ and $\mathcal{H}^2_{\lVert\cdot\rVert_\infty}\llcorner \Gamma$ are mutually absolutely continuous. 

\vvv

\section{Taylor expansion of area on quadratic \texorpdfstring{$t$}{Lg}-cones}
\label{TYLR}
Before giving a short account on the content of this appendix, let us introduce some notation. \textbf{Throughout this appendix we always suppose that} $\mathcal{D}\in\mathrm{Sym}(n)\setminus \{0\}$ \textbf{and let} $f\colon \R^{n}\to\R$ be the quadratic polynomial defined as
$$f(y)\coloneqq \langle y,\mathcal{D} y\rangle.$$
\textbf{Furthermore, we fix } $x\in\R^{n}\setminus \mathrm{Ker}(\mathcal D)$, \textbf{ let} $\mathcal{X}\coloneqq (x,f(x))$, and denote $\gr(f)\coloneqq \{(x,f(x)):x\in \R^n\}.$

\medskip

The main goal of this section is to determine an asymptotic expansion of $\mathcal{H}^{n+1}(B(\mathcal{X},r)\cap \gr(f))$ for $r$ small. More precisely, written
\begin{equation}\label{eq:expansion_area_formula}
    \mathcal{H}^{n+1}(B(\mathcal{X},r)\cap \gr(f))=\mathfrak{c}(\mathcal{X})r^{n+1}+\zeta(\mathcal{X})r^{n+2}+\mathfrak{e}(\mathcal{X})r^{n+3}+O(r^{n+4}),
\end{equation}
we want to find an expression for the coefficients $\mathfrak{c},\zeta,\mathfrak{e}$ in terms of $x$, $\mathcal{D}$ and $n$.
The coefficient $\mathfrak{c}$ is quite easy to study and we show that it is a constant depending only on $n$. On the other hand the coefficients $\zeta$ and $\mathfrak{e}$ need much more work and they play a fundamental role in the study of the geometric properties of $1$-codimensional uniform measures carried on in Section \ref{eq:sec_non_deg_unif_meas}.

\begin{definition}
Let $\mathcal{D}$, $\mathcal{X}$ and $f$ be as above. We denote as $$\mathfrak{n}\coloneqq \frac{ \mathcal{D}x}{\lvert \mathcal{D}x\rvert},$$
the \emph{horizontal normal} at $\mathcal{X}$ to $\text{gr}(f)$ and we let $c\coloneqq 2\lvert \mathcal{D}x\rvert$.
\end{definition}

The following proposition gives a first characterization of the shape of the intersection between $B(\mathcal{X},r)$ and $\text{gr}(f)$. In particular we construct a function $G$ at the point $x$ whose sublevel sets are the horizontal projection of $B(\mathcal{X},r)\cap\text{gr}(f)$.

\begin{proposition}\label{pll}
In the notations above, defined $G(w)\coloneqq \lvert w\rvert^4+\lvert c\langle \mathfrak{n},w\rangle+\langle w,\mathcal{D}w\rangle \rvert^2$, we have
\begin{equation}
   \pi_H(B(\mathcal{X},r)\cap\gr(f))=x+\{w\in\R^{n}:G(w)\leq r^4\}.
   \nonumber
\end{equation}
\end{proposition}

\begin{proof}
The definitions of $\mathcal{X}$ and of the Koranyi norm imply
\begin{equation}
\begin{split}
    B(\mathcal{X},r)\coloneqq&\{z\in\R^{n+1}:\lvert z_H-x\rvert^4+\lvert z_T-f(x) \rvert^2\leq r^4\}.
\end{split}
    \nonumber
\end{equation}
Therefore,
\begin{equation}\label{eq:Bcapgrf1}
\begin{split}
B(\mathcal{X},r)\cap\gr(f)=&\bigl\{(y,f(y))\in\R^{n+1}:\lvert x-y\rvert^4+\lvert -f(x)+f(y)\rvert^2\leq r^4\bigr\}\\
=&\mathcal{X}+\bigl\{(w,f(w)+2\langle w, \mathcal{D} x\rangle)\in\R^{n+1}:\lvert w\rvert^4+\lvert -f(x)+f(x+w)\rvert^2\leq r^4\bigr\},
\end{split}
\end{equation}
where in the last line we have performed the change of variable $y=x+w$. By definition of $f$ and the symmetry of $\mathcal D$, we have
\begin{equation}\label{eq:Bcapgrf2}
\begin{split}
    -f(x)+f(x+w)&=-\langle x,\mathcal{D}x\rangle+\langle x+w, \mathcal{D}(x+w)\rangle\\
    &=2\langle x,\mathcal{D} w\rangle+\langle w,\mathcal{D}w\rangle=c\langle \mathfrak{n},w\rangle+\langle w,\mathcal{D}w\rangle,
\end{split}
\end{equation}
where the latter equality follows by the definitions of $\mathfrak n$ and $c$.
In particular, \eqref{eq:Bcapgrf1} and \eqref{eq:Bcapgrf2} imply that
    \begin{equation*}
\begin{split}
\pi_H(B(\mathcal{X},r)\cap\gr(f))=
x+\bigl\{w\in\R^{n}:\lvert w\rvert^4+\lvert c\langle \mathfrak{n},w\rangle+\langle w,\mathcal{D}w\rangle\rvert^2\leq r^4\bigr\},
\end{split}
\end{equation*}
which concludes the proof by definition of $G$.
\end{proof}

\vv

We now introduce a special set of polar coordinates, which are very useful in the study of $B(\mathcal{X},r)\cap \gr(f)$ when $r$ is small.

\begin{proposition}\label{coordinate}
For any $w\in \R^{n}\setminus \bigl( x+\mathrm{span}( \mathfrak{n})\bigr)$ there exists a unique triple $(\rho, \vartheta,v)\in \mathscr{C}\coloneqq (0,\infty)\times [-\frac{\pi}{2},\frac{\pi}{2})\times\bigl(\mathbb{S}^{n-1}\cap \mathfrak{n}^\perp\bigr)$ such that
\begin{equation}
    w=x+\frac{\sin \vartheta}{c} \rho^2 \mathfrak{n}+\cos\vartheta\rho v\eqqcolon x+\mathcal{P}(\rho, \vartheta,v).
    \label{N:4}
\end{equation}
\end{proposition}

\begin{proof}
The proof can be obtained following verbatim \cite[Proposition B.2]{MerloG1cod}.
\end{proof}

\vv

\begin{proposition}[Representation formula for the area]
In the notations fixed above, we have
\[
	\sigma_{\gr(f)}(B(\mathcal{X},r))=\int_{\pi_H(B(\mathcal{X},r)\cap \mathrm{gr}(f))}\lvert \grad f(x)\rvert \,dx.
\]
\end{proposition}

\begin{proof}
Coarea formula implies that
$$\int_{\pi_H(B(\mathcal{X},r)\cap \mathrm{gr}(f))}\lvert \grad f(x)\rvert \,dx=\int \mathcal{H}^{n-1}_{\mathrm{eu}}\bigl(B(\mathcal{X},r)\cap \{x:f(x)=t\}\bigr)\,dt=\sigma_{\gr(f)}(B(\mathcal{X},r)),$$
which concludes the proof.
\end{proof}
\vv

\textbf{For simplicity of notation, we define}
\begin{equation}
  \alpha_\mathfrak{n}\coloneqq \left\langle \mathfrak{n},\mathcal{D}\mathfrak{n}\right\rangle ,\qquad\beta_\mathfrak{n}(v)\coloneqq \left\langle v,\mathcal{D}\mathfrak{n}\right\rangle,\qquad\gamma(v)\coloneqq \left\langle v,\mathcal{D}v\right\rangle \qquad \text{for any }v\in\mathbb{S}^{n-1}. 
  \label{numbero27}
\end{equation}
In the following proposition we give an explicit expression of $G$ in the coordinates $\mathcal{P}(\vartheta,\rho,v)$ introduced in Proposition \ref{coordinate}.

\begin{proposition}\label{strutto}
Let $\mathscr{C}$ and  $\mathcal{P}(\rho, \vartheta,v)$ be as in Proposition \ref{coordinate}. Let us define the function $H\colon\mathscr{C}\to\R$ as
\begin{equation}
    H(\rho, \vartheta,v)\coloneqq G(\mathcal{P}(\rho,\vartheta,v)),
    \label{numerooo3}
\end{equation}
where $G$ was introduced in Proposition \ref{pll}.
Then, $H$ can be expressed as
\[
	H(\rho, \vartheta,v)= A(\vartheta,v)\rho^4+\frac{\overline{B}(\vartheta, v)}{c}\rho^5+\frac{\overline{C}(\vartheta, v)}{c^2}\rho^6+\frac{\overline{D}(\vartheta, v)}{c^3}\rho^7+\frac{\overline{E}(\vartheta)}{c^4}\rho^8,
\]
where, as $(\vartheta,v)$ varies in $[-\pi/2,\pi/2)\times\bigl(\mathbb{S}^{n-1}\cap \mathfrak{n}^\perp\bigr)$, we define:
\begin{itemize}
\item[(i)] $A(\vartheta,v)\coloneqq (\cos^4\vartheta+(\cos^2\vartheta\gamma(v)+\sin\vartheta)^2)$.
\item[(ii)] $\overline{B}(\vartheta,v)\coloneqq cB(\vartheta,v)\coloneqq 4\sin\vartheta\cos\vartheta\beta_\mathfrak{n}(v)(\cos^2\vartheta\gamma(v)+\sin\vartheta)$.
\item[(iii)] $\overline{C}(\vartheta,v)\coloneqq c^2C(\vartheta,v)\coloneqq \sin^2\vartheta\bigl(\cos^2\vartheta(2+4\beta_\mathfrak{n}(v)^2+2\gamma(v)\alpha_\mathfrak{n})+2\sin\vartheta\alpha_\mathfrak{n}\bigr)$.
\item[(iv)] $\overline{D}(\vartheta,v)\coloneqq c^3D(\vartheta,v)\coloneqq 4\alpha_\mathfrak{n}\beta_\mathfrak{n}(v)\sin\vartheta^3\cos\vartheta$.
\item[(v)] $\overline{E}(\vartheta)\coloneqq c^4E(\vartheta)\coloneqq (1+\alpha_\mathfrak{n}^2)\sin^4\vartheta$.
\end{itemize}
\end{proposition}

\begin{proof}
The proof can be obtained following verbatim \cite[Proposition B.3]{MerloG1cod}.
\end{proof}
\vv

We now summarize some algebraic properties of the functions $A,\ldots,E$ introduced in Proposition \ref{strutto}.

\begin{lemma}\label{symA}
Let $(\vartheta,v)\in[-\pi/2,\pi/2)\times\bigl(\mathbb{S}^{n}\cap \mathfrak{n}^\perp\bigr)$. Then:
\begin{itemize}
\item[(i)] $A(\vartheta,v)=A(\vartheta,-v)$ and $C(\vartheta,v)=C(\vartheta,-v)$.
\item[(ii)] $B(\vartheta,v)=-B(\vartheta,-v)$ and $D(\vartheta,v)=-D(\vartheta,-v)$.
\item[(iii)] $D(\vartheta,v)=-D(-\vartheta,v)$.
\item[(iv)] $E$ does not depend on $v$.
\item[(v)] $A$ is bounded away from $0$, i.e. \[\omega\coloneqq \min_{(\vartheta,v)\in[-\pi/2,\pi/2)\times\mathbb{S}^{n}\cap \mathfrak{n}^\perp}A(\vartheta,v)>0.\]
\end{itemize}
\end{lemma}

\begin{proof}
The proof can be obtained following verbatim \cite[Proposition B.4]{MerloG1cod}.
\end{proof}


The following proposition allows us to determine, up to a certain degree of precision, the shape of the set $\pi_H(B(\mathcal{X},r)\cap\text{gr}(f))$ when $r$ is small.

\begin{proposition}\label{propexp}
There exists an $0<\mathfrak{r}_1(\mathcal{X})=\mathfrak{r}_1<1$ such that for any $0<r<\mathfrak{r}_1$, if $\rho(r)$ is a solution to the equation
\begin{equation}
    H(\rho(r),\vartheta,v)=r^4,
    \label{eq152}
\end{equation}
then
\begin{equation}
    \rho(r)=P_{\vartheta,v}(r)+O(r^4)\coloneqq \frac{r}{A^\frac{1}{4}}-\frac{Br^2}{4A^\frac{3}{2}}+\left(\frac{7}{32}\frac{B^2}{A^\frac{11}{4}}-\frac{C}{4A^\frac{7}{4}}\right)r^3+O(r^4),
    \label{N:3}
	\end{equation}
and the remainder $O(r^4)$ is independent on $v$ and on $\vartheta$.
\end{proposition}

\begin{proof}
The proof can be obtained following verbatim \cite[Proposition B.5]{MerloG1cod}.
\end{proof}

\vv

\textbf{Let us denote by $\mathfrak{r}_2$ the supremum of those positive numbers for which} $P_{\vartheta,v}(r)\geq \omega^{1/4} r/2$.
 
\begin{corollary}\label{Co1}
For any $0<r<\min\{1,\mathfrak{r}_2\}$ and any $\delta\in (-\omega^{1/4}/2,\omega^{1/4}/2)$, define the set
$$\mathcal{B}_{r,\delta}\coloneqq \mathcal{P}\bigl( \bigl\{(\rho,\vartheta,v)\in\mathscr{C}: \rho\leq P_{\vartheta,v}(r)+\delta r^3\bigr\}\bigr),$$
where $P_{\vartheta,v}$ was defined in \eqref{N:3}.
Then, there exists an $\epsilon_0(\mathcal{X})=\epsilon_0>0$ such that for any $0<\epsilon<\epsilon_0$, there is an $0<\mathfrak{r}_3(\epsilon)=\mathfrak{r}_3$ such that for any $0<r<\mathfrak{r}_3$, we have
\begin{equation}
x+\mathcal{B}_{r,-\epsilon}\subseteq\pi_H( B_{r}(\mathcal{X})\cap\gr(f))\subseteq x+\mathcal{B}_{r,\epsilon}.
    \nonumber
\end{equation}
\end{corollary}

\begin{proof}
Proposition \ref{pll}, the definition of $\mathcal{P}$ (see \eqref{N:4}) and of $H$ (see \eqref{numerooo3}) imply
\begin{equation}
    \pi_H(B(\mathcal{X},r)\cap\text{gr}(f))=x+\mathcal{P}(\{(\rho,\vartheta, v)\in\mathscr{C}:H(\rho,\vartheta, v)\leq r^4\}).
    \label{numbero42}
\end{equation}
The function $\rho\mapsto H(\rho,\vartheta, v)$ is a polynomial of degree $8$ in $\rho$, thus the equation
\begin{equation}
    H(\rho,\vartheta, v)=r^4
    \label{eq172}
\end{equation}
has at most $8$ solutions in $\rho$ for any fixed $\vartheta$ and $v$. As $H(0,\vartheta,v)=0$ and $\lim_{\rho\to +\infty}H(\rho,\vartheta, v)=+\infty$, the equation \eqref{eq172} has at least one positive solution. Assume $0<\rho_1<\ldots<\rho_k$, where $k\in\{1,\ldots,8\}$, is the number of \emph{positive} distinct solutions of \eqref{eq172}. If $\rho>\rho_k$ then $H(\rho,\vartheta,v)>r^4$ and on the other hand, since $H(0,\vartheta,v)=0$, if $0\leq\rho<\rho_1$ then $H(\rho,\vartheta, v)<r^4$. Hence
\begin{equation}
    \begin{split}
        \mathcal{P}\left( \left\{(\rho,\vartheta,v)\in\mathscr{C}: \rho\leq \rho_1\right\}\right)&\subseteq\mathcal{P}(\{(\rho,\vartheta, v)\in\mathscr{C}:H(\rho,\vartheta, v)\leq r^4\})\\
        &\subseteq \mathcal{P}\left( \left\{(\rho,\vartheta,v)\in\mathscr{C}: \rho\leq \rho_k\right\}\right).
        \nonumber
    \end{split}
\end{equation}
Proposition \ref{propexp} concludes the proof since $\rho_1$ and $\rho_k$ coincide up to an error of order $r^4$.
\end{proof}

\vv
The following technical lemma will be needed in the computations of Proposition \ref{TEXP}, and it is a Taylor expansion formula for the sub-Riemmanian area element at a non-characteristic point of a horizontal quadric.

\begin{lemma}\label{sviluppodens}
For any $(\rho,\vartheta, v)\in\mathscr{C}$ we have
\begin{equation}
2\bigl\lvert \mathcal{D}\bigl[x+\mathcal{P}(\rho,\vartheta, v)\bigr]\bigr\rvert=c+\mathcal{A}(\vartheta,v)\rho+\mathcal{B}(\vartheta,v)\rho^2+R_5(\rho),
\nonumber
\end{equation}
where:
\begin{itemize}
\item[(i)] $\mathcal{A}(\vartheta,v)\coloneqq 2\cos\vartheta  \beta_{\mathfrak{n}}(v)$.
\item[(ii)] $\overline{\mathcal{B}}(\vartheta,v)=c\mathcal{B}(\vartheta,v)\coloneqq 2(\sin\vartheta\alpha_\mathfrak{n}+\cos^2\vartheta\lvert P_{\mathfrak{n}^\perp}(\mathcal{D}v)\rvert^2)$,
and $P_\mathfrak{n}$ is the orthogonal projection in $\R^{n}$ on $\mathfrak{n}^\perp$.
\item[(iii)] $\lvert R_5(\rho)\rvert\leq \mathfrak{c}_5\rho^3$ for any $0<\rho<\mathfrak{r}_4$ and for some constant $\mathfrak{c}_5>0$, independent on $\vartheta$ and $v$.
\end{itemize}
\end{lemma}

\begin{proof}
First of all we, observe that
\begin{equation}
\begin{split}
     2\mathcal{D}[x+\mathcal{P}(\rho,\vartheta, v)]
    =2\mathcal{D}\left[x+\frac{\sin \vartheta}{c} \rho^2 \mathfrak{n}+\cos\vartheta\rho v\right]=c\mathfrak{n}+\frac{2\sin \vartheta}{c} \rho^2\mathcal{D} \mathfrak{n}+2\cos\vartheta\rho  \mathcal{D}v.
\end{split}
\nonumber
\end{equation}
Then, an elementary calculation yields
\begin{equation}
    \begin{split}
       4\bigl\lvert  \mathcal{D}[x+\mathcal{P}(\rho,\vartheta, v)]\bigr\rvert^2=c^2&+4c\cos\vartheta\rho  \langle \mathcal{D}v,\mathfrak{n}\rangle+4\sin \vartheta \rho^2\langle \mathcal{D} \mathfrak{n},\mathfrak{n}\rangle+4\cos^2\vartheta\rho^2\lvert \mathcal{D}v\rvert^2\\
        &+\frac{8\sin \vartheta\cos\vartheta}{c}\rho^3\langle \mathcal{D}\mathfrak{n},\mathcal{D}v\rangle+\frac{4\sin^2 \vartheta}{c^2} \rho^4\lvert \mathcal{D} \mathfrak{n}\rvert^2\\
        =c^2&+4c\cos\vartheta  \langle \mathcal{D}v,n\rangle\rho+\Big[4\sin \vartheta \langle \mathcal{D} \mathfrak{n},\mathfrak{n}\rangle+4\cos^2\vartheta\lvert \mathcal{D}v\rvert^2\Big]\rho^2\\
        &+\frac{8\sin \vartheta\cos\vartheta}{c}\rho^3\langle \mathcal{D}\mathfrak{n},\mathcal{D}v\rangle+\frac{4\sin^2 \vartheta}{c^2} \rho^4\lvert \mathcal{D} \mathfrak{n}\rvert^2.
    \end{split}
\label{numerooo6}
\end{equation}
Another standard computation shows that, if $\rho$ is small enough, then
\begin{equation}
\begin{split}
&4\lvert \mathcal{D}[x+\mathcal{P}(\rho,\vartheta, v)]\rvert^2=c+2\cos\vartheta  \langle \mathcal{D}v,\mathfrak n\rangle\rho \\ &\qquad\qquad+\frac{-\big[4c\cos\vartheta  \langle \mathcal{D}v,n\rangle\big]^2+4c^2\big[4\sin \vartheta \langle \mathcal{D} \mathfrak{n},\mathfrak{n}\rangle+4\cos^2\vartheta\lvert \mathcal{D}v\rvert^2\big]}{8c^3}\rho^2+R_5(\rho),
\end{split}
\end{equation}
which boils down to
\begin{equation}
\begin{split}
     2\lvert  \mathcal{D}[x+\mathcal{P}(\rho,\vartheta, v)]\rvert=c+&2\cos\vartheta  \beta_{\mathfrak{n}}(v)\rho 
     +2\frac{\sin\vartheta\alpha_\mathfrak{n}+\cos^2\vartheta\lvert P_{\mathfrak{n}^\perp}(\mathcal{D}v)\rvert^2}{c}\rho^2+R_5(\rho)
\end{split}
\end{equation}
where $P_{\mathfrak{n}^\perp}$ is the orthogonal projection in $\R^n$ onto $\mathfrak{n}^\perp$.
\end{proof}

\vv

\begin{remark}\label{rk22}
The functions $\mathcal{A}(\cdot,\cdot)$ and $\mathcal{B}(\cdot,\cdot)$ defined in the statement of Proposition \ref{sviluppodens} have the following symmetries.
For any $(\vartheta,v)\in[-\pi/2,\pi/2]\times\bigl(\mathbb{S}^{n-1}\cap \mathfrak{n}^\perp\bigr)$, we have that:
\begin{itemize}
\item[(i)] $\mathcal{A}(\vartheta,v)=-\mathcal{A}(\vartheta,-v)$.
\item[(ii)] $\mathcal{B}(\vartheta,v)=\mathcal{B}(\vartheta,-v)$.
\end{itemize}
\end{remark}

\begin{proposition}\label{rapr1}
For any $0<r<\mathrm{dist}(x,\Sigma(f))$, we have
$$\sigma_{\mathrm{gr}(f)}(B(\mathcal{X},r))= \int_{\mathbb{S}^{n-1}\cap\mathfrak{n}^\perp}\int_{-\frac{\pi}{2}}^{\frac{\pi}{2}}\int_{\{H(\rho,\vartheta,v)\leq r^4\}}\Xi(\rho,\vartheta, v) \, d\rho d\vartheta d\omega(v),$$
where:
\begin{itemize}
\item[(i)] $\Xi(\rho,\vartheta, v)\coloneqq c^{-1}\rho^{n}\cos^{n-2}\vartheta(1+\sin^2\vartheta)\cdot 2\lvert \mathcal{D}[x+\mathcal{P}(\rho,\vartheta, v)]\rvert^2$.
\item[(ii)] $\omega\coloneqq \mathcal{H}^{n-2}_{\eu}\llcorner \mathbb{S}^{n-1}\cap \mathfrak{n}^\perp$.
\end{itemize}
\end{proposition}

\begin{proof}
The proof can be obtained following verbatim \cite[Proposition B.8]{MerloG1cod}.
\end{proof}

\vv

The following two lemmas are needed to compute some integrals in Propositions \ref{TEXP} and \ref{TEXP2}.

\begin{lemma}\label{conto1}
For any $k\in\N$ and any $\alpha>(k+1)/2$ we have
\begin{equation}
    \int_{-\infty}^\infty \frac{x^k}{(1+x^2)^\alpha}\,dx=\begin{cases}
    0 &\text{if }k\text{ is odd,}\\
   \frac{\Gamma\big(\frac{k+1}{2}\big)\Gamma\big(\alpha-\frac{k+1}{2}\big)}{\Gamma(\alpha)} &\text{if }k\text{ is even.}
    \end{cases}
    \nonumber
\end{equation}
\end{lemma}

\begin{proof}
The proof can be obtained following verbatim \cite[Proposition B.9]{MerloG1cod}.
\end{proof}

\vv
\begin{lemma}\label{conto2}
Suppose that $f\colon\R\to\R$ is a measurable function such that $f(x)/{(1+x^2)^\alpha}\in L^1(\R)$ and let $\mathfrak{d}(\vartheta)\coloneqq \cos^{n-2}\vartheta(\cos^2\vartheta+2\sin^2\vartheta)$. Then it holds that
$$\int_{-\frac{\pi}{2}}^\frac{\pi}{2}\mathfrak{d}(\vartheta)\frac{\cos^{4\alpha-n-1}\vartheta f\left(\frac{\sin\vartheta}{\cos^2\vartheta}\right)}{A(\vartheta,v)^\alpha}\,d\vartheta=\int_{-\infty}^\infty\frac{f(x)}{\big(1+\big(x+\gamma(v)\big)^2\big)^\alpha}\,dx,$$
where $\gamma(v)$ was defined in \eqref{numbero27} and where $A(\vartheta,v)$ was introduced in Proposition \ref{strutto}.
\end{lemma}

\begin{proof}
The proof can be obtained following verbatim \cite[Proposition B.10]{MerloG1cod}.
\end{proof}

\vv


Proposition \ref{TEXP} gives a first description of the structure of the coefficients $\mathfrak{c}$,  $\zeta$ and $\mathfrak{e}$.

\begin{proposition}\label{TEXP}
For any $\epsilon>0$ there exists $\mathfrak{r}_5=\mathfrak{r}_5(\epsilon)>0$ such that for any $0<r<\mathfrak{r}_5$ it holds
\begin{equation}
   \sigma_{\mathrm{gr}(f)}(B(\mathcal{X},r))=\mathfrak{c}_nr^{n+1}+\mathfrak{e}(\mathcal{X})r^{n+3}+\epsilon R_6(r),
    \label{N:5}
\end{equation}
where for $\mathbb{S}(\mathfrak{n})\coloneqq \mathbb{S}^{n-1}\cap\mathfrak{n}^\perp$, we have:
\begin{itemize}
\item[(i)]The coefficient $\mathfrak{c}(\mathcal{X})$ is independent on $\mathcal{X}$ and
$$\mathfrak{c}(\mathcal{X})=\mathfrak{c}_n\coloneqq \frac{\sqrt{\pi}\Gamma\left(\frac{n-1}{4}\right)\sigma(\mathbb{S}(\mathfrak{n}))}{(n+1)\Gamma\left(\frac{n+1}{4}\right)}.$$
\item[(ii)]The coefficient $\mathfrak{e}(\mathcal{X})$ has the expression
$$\mathfrak{e}(\mathcal{X})=\int_{\mathbb{S}(\mathfrak{n})}\int_{-\frac{\pi}{2}}^{\frac{\pi}{2}}\frac{\mathfrak{d}(\vartheta)\Big(\frac{7+n}{32}\frac{\overline{B}^2}{A^2}-\frac{\overline{C}}{4A}-\frac{\mathcal{A}\overline{B}}{4A}+\frac{\overline{\mathcal{B}}}{(n+3)}\Big)}{c^2 A^{\frac{n+3}{4}}}\,d\vartheta d\omega,$$
where the coefficients $A,\overline{B},\overline{C}$ were introduced in Proposition \ref{strutto} and $\mathcal{A}$ and $\overline{\mathcal{B}}$ in Proposition \ref{sviluppodens}.
\item[(iii)] $\lvert R_6(r)\rvert\leq \mathfrak{C}_{11}(n) r^{n+3}$ for any $0<r<\mathfrak{r}_5$ and for some constant $\mathfrak{C}_{11}(n)=\mathfrak{C}_{11}(n,\mathcal{X})$ depending only on $\mathcal{X}$.
\end{itemize}
\end{proposition}

\begin{proof}
The proof can be obtained following verbatim \cite[Proposition B.11]{MerloG1cod}.
\end{proof}

\begin{definition}\label{Cn}
Throughout the rest of this appendix we denote
$$\mathcal{C}_n\coloneqq \frac{\sqrt{\pi}\Gamma\left(\frac{n+1}{4}\right)}{\frac{n+3}{4}\Gamma\left(\frac{n+3}{4}\right)}.$$
\end{definition}

The previous propositions gives a first characterization of the coefficients of the Taylor expansion of the perimeter of quadratic surfaces. The coefficient relative to $r^{n+1}$ has been proved to be a constant depending only on $n$ and the one of $r^{n+2}$ has been proved null. We now investigate more carefully the structure of the coefficient relative to $r^{n+3}$.

\begin{proposition}\label{TEXP2}
In the notation of the previous propositions we have
\begin{equation}
  \begin{split}
    \frac{c^2\mathfrak{e}(\mathcal{X})}{\mathcal{C}_n\sigma(\mathbb{S}(\mathfrak{n}))}=&\frac{\text{Tr}(\mathcal{D}^2)-2\langle \mathfrak{n},\mathcal{D}^2\mathfrak{n}\rangle+\langle \mathfrak{n},\mathcal{D}\mathfrak{n}\rangle^2}{4(n-1)}-\frac{1}{4}-\frac{(\text{Tr}(\mathcal{D})-\langle \mathfrak{n},\mathcal{D}\mathfrak{n}\rangle)^2}{8(n-1)}.
    \nonumber
\end{split} 
\end{equation}
\end{proposition}
%
\begin{proof}
Proposition \ref{TEXP}-(ii) yields
\begin{equation}
\begin{split}
c^2\mathfrak{e}(\mathcal{X})=&\int_{\mathbb{S}(\mathfrak{n})}\Bigg(\underbrace{\frac{n+7}{32}\int_{-\frac{\pi}{2}}^{\frac{\pi}{2}}\frac{\mathfrak{d}(\vartheta)\overline{B}^2}{A^\frac{n+11}{4}}\,d\vartheta }_\text{(I)}-\underbrace{\int_{-\frac{\pi}{2}}^{\frac{\pi}{2}}\frac{\mathfrak{d}(\vartheta)\overline{C}}{4A^\frac{n+7}{4}}\,d\vartheta}_\text{(II)}\\
&\qquad\qquad\qquad\qquad\qquad-\underbrace{\int_{-\frac{\pi}{2}}^{\frac{\pi}{2}}\frac{\mathfrak{d}(\vartheta)\mathcal{A}\overline{B}}{4A^\frac{n+7}{4}}\,d\vartheta}_\text{(III)} +\underbrace{\int_{-\frac{\pi}{2}}^{\frac{\pi}{2}} \frac{\mathfrak{d}(\vartheta)\overline{\mathcal{B}}}{(n+3)A^\frac{n+3}{4}}\,d\vartheta}_\text{(IV)}\Bigg)\, d\omega.
    \label{bigexp}
\end{split}
\end{equation}
We now study each one of the terms (I)$,\ldots,$(IV) separately. Let us start with the integral (I). Since \[\overline{B}^2=\cos^{10}\vartheta16\beta_\mathfrak{n}(v)^2\Bigl(\frac{\sin\vartheta}{\cos^2\vartheta}\Bigr)^2\Bigl(\frac{\sin\vartheta}{\cos^2\vartheta}+\gamma(v)\Bigr)^2,\] Lemma \ref{conto2}, a standard change of variables and Lemma \ref{conto1} imply
\begin{equation}
\begin{split}
&\frac{n+7}{32}\int_{-\frac{\pi}{2}}^{\frac{\pi}{2}}\frac{\mathfrak{d}(\vartheta)\overline{B}^2}{A^{\frac{n+11}{4}}}d\vartheta
=\frac{n+7}{32}\int_{-\infty}^\infty \frac{16\beta_\mathfrak{n}(v)^2x^2(\gamma(v)+x)^2}{(1+(\gamma(v)+x)^2)^{\frac{n+11}{4}}}\, dx\\
&\qquad=\frac{(n+7)\beta_\mathfrak{n}(v)^2}{2}\int_{-\infty}^{\infty}\frac{x^2(x-\gamma(v))^2}{(1+x^2)^{\frac{n+11}{4}}}dx\\
&\qquad=\frac{(n+7)\beta_\mathfrak{n}(v)^2}{2}\left(\int_{-\infty}^{\infty}\frac{x^4}{(1+x^2)^{\frac{n+11}{4}}}dx+\gamma(v)^2\int_{-\infty}^{\infty}\frac{x^2}{(1+x^2)^{\frac{n+11}{4}}}dx\right)\\
&\qquad =2\mathcal{C}_n\beta_\mathfrak{n}(v)^2\left(\frac{3}{4}+\frac{(n+1)\gamma(v)^2}{8}\right),
\label{numbero(I)}
\end{split}
\end{equation}
where the constant $\mathcal{C}_n$ was introduced in Definition \ref{Cn}.

We turn now our attention to (II). Since \[\overline{C}=\cos^6\vartheta(\frac{\sin\vartheta}{\cos^2\vartheta})^2\Big((2+4\beta_\mathfrak{n}(v)^2+2\gamma(v)\alpha_\mathfrak{n})+2\alpha_\mathfrak{n}\frac{\sin\vartheta}{\cos^2\vartheta}\Big),\] Lemmas \ref{conto1} and \ref{conto2} imply
\begin{equation}
    \begin{split}
        \int_{-\frac{\pi}{2}}^{\frac{\pi}{2}}&\frac{\mathfrak{d}(\vartheta)\overline{C}}{4A^\frac{n+7}{4}}\,d\vartheta
         =\frac{1}{2}\int_{-\infty}^{\infty}\frac{x^2((1+2\beta_\mathfrak{n}(v)^2+\gamma(v)\alpha_\mathfrak{n})+\alpha_\mathfrak{n}x)}{\left(1+(x+\gamma(v))^2\right)^{\frac{n+7}{4}}}\,dx\\
         =&\frac{1}{2}\int_{-\infty}^{\infty}\frac{(x-\gamma(v))^2((1+2\beta_\mathfrak{n}(v)^2)+\alpha_\mathfrak{n}x)}{\left(1+x^2\right)^{\frac{n+7}{4}}}\,dx\\
         =&\frac{(1+2\beta_\mathfrak{n}(v)^2-2\alpha_\mathfrak{n}\gamma(v))}{2}\int_{-\infty}^{\infty}\frac{x^2}{\left(1+x^2\right)^{\frac{n+7}{4}}}\,dx+\frac{(1+2\beta_\mathfrak{n}(v)^2)\gamma(v)^2}{2}\int_{-\infty}^\infty\frac{dx}{\left(1+x^2\right)^{\frac{n+7}{4}}}\\
         =&\mathcal{C}_n\frac{(n+1)(1+2\beta_\mathfrak{n}^2(v))\gamma(v)^2+2+4\beta_\mathfrak{n}(v)^2-4\gamma(v)\alpha_\mathfrak{n}}{8}.
    \end{split}
    \label{numbero(II)}
\end{equation}
    Since $\mathcal{A}\overline{B}=8\cos^6\vartheta\beta_\mathfrak{n}(v)^2\frac{\sin\vartheta}{\cos^2\vartheta}\Big(\gamma(v)+\frac{\sin\vartheta}{\cos^2\vartheta}\Big)$, Lemmas \ref{conto1} and \ref{conto2} imply
\begin{equation}
    \begin{split}
        \int_{-\frac{\pi}{2}}^{\frac{\pi}{2}}\frac{\mathfrak{d}(\vartheta)\mathcal{A}\overline{B}}{4A^{\frac{n+7}{4}}}\,dx
        =&2\beta_\mathfrak{n}(v)^2\int_{-\infty}^{\infty}\frac{x(\gamma(v)+x)}{(1+(\gamma(v)+x)^2)^{\frac{n+7}{4}}}\,d\vartheta\\
        =&2\beta_\mathfrak{n}(v)^2\int_{-\infty}^{\infty}\frac{x(x-\gamma(v))}{(1+x^2)^{\frac{n+7}{4}}}\,dx
        =\mathcal{C}_n\beta_\mathfrak{n}(v)^2,
    \end{split}
    \label{numbero(III)}
\end{equation}
which concludes the discussion of the integral (III). Finally, we are left with the discussion of (IV). Thanks to the fact that $\overline{\mathcal{B}}=2\cos^2\vartheta\big(\alpha_\mathfrak{n}\frac{\sin\vartheta}{\cos^2\vartheta}+\lvert P_{\mathfrak{n}^\perp}(\mathcal{D}v)\rvert^2\big)$ Lemmas \ref{conto1} and \ref{conto2}, imply that
\begin{equation}
    \begin{split}
    \int_{-\frac{\pi}{2}}^{\frac{\pi}{2}}\frac{\mathfrak{d}(\vartheta)}{A^{\frac{n+3}{4}}}\frac{\overline{\mathcal{B}}}{(n+3)}\,d\vartheta
    =&\frac{2}{n+3}\int_{-\infty}^\infty \frac{\alpha_\mathfrak{n}x+\lvert P_\mathfrak{n}\mathcal{D}v\rvert^2}{(1+(x+\gamma(v)))^\frac{n+3}{4}}\,dx
    =\mathcal{C}_n\frac{-\alpha_\mathfrak{n} \gamma(v)+\lvert P_\mathfrak{n}\mathcal{D}v\rvert^2}{2}.
    \end{split}
    \label{numbero(IV)}
\end{equation}
We plug the identities \eqref{numbero(I)}, \eqref{numbero(II)}, \eqref{numbero(III)}, \eqref{numbero(IV)} into \eqref{bigexp}, and we get
\begin{equation}
    \begin{split}
    \frac{c^2\mathfrak{e}(\mathcal{X})}{\mathcal{C}_n\omega(\mathbb{S}(\mathfrak{n}))}=&\fint_{\mathbb{S}(\mathfrak{n})} \bigg[\beta_\mathfrak{n}(v)^2\left(\frac{3}{2}+\frac{n+1}{4}\gamma(v)^2\right)\bigg]\,d\omega(v)\\
    &-\fint_{\mathbb{S}(\mathfrak{n})} \bigg[\frac{(n+1)(1+2\beta_\mathfrak{n}^2(v))\gamma(v)^2+2+4\beta_\mathfrak{n}(v)^2-4\gamma(v)\alpha_\mathfrak{n}}{8}\bigg]\,d\omega(v)\\
    &+\fint_{\mathbb{S}(\mathfrak{n})}\left[-\beta_\mathfrak{n}(v)^2 +\frac{-\alpha_\mathfrak{n} \gamma(v)+\lvert P_\mathfrak{n}(\mathcal{D}v)\rvert^2}{2}\right] \,d\omega(v)\\
    =&-\frac{n+1}{8}\underbrace{\fint_{\mathbb{S}(\mathfrak{n})}\gamma(v)^2 \,d\omega}_\text{(V)}-\frac{1}{4}+\underbrace{\fint_{\mathbb{S}(\mathfrak{n})}\frac{\lvert P_{\mathfrak{n}}(\mathcal{D}v)\rvert^2}{2} \,d\omega}_\text{(VI)}.
        \label{equizias}
    \end{split}
\end{equation}

In order to make the expression for $\mathfrak{e}(\mathcal{X})$ explicit, we need to compute the integrals (V) and (VI). To do so, we let $\mathcal{E}\coloneqq \{m_1,\ldots,m_{n}\}$ be an orthonormal basis  of $\R^{n}$ such that $m_1=\mathfrak{n}$.
Moreover, the points $v\in\mathbb{S}(\mathfrak{n})$ can be written as $v=\sum_{i=2}^{n} v_i m_i$, where $v_i\coloneqq \langle v,m_i\rangle$. This is due to the fact that $v\in \mathfrak{n}^\perp$ by definition of $\mathbb{S}(\mathfrak{n})$. With these notations, the integral (V) becomes
\begin{equation}
    \begin{split}
        \fint_{\mathbb{S}(\mathfrak{n})}&\gamma(v)^2\,  d\sigma(v)=\fint_{\mathbb{S}(\mathfrak{n})}\left(\sum_{i,j=2}^{n}\langle m_i,\mathcal{D}m_j\rangle v_iv_j\right)^2\, d\sigma(v)\\
        =&\sum_{i,j,k,\ell=2}^{n}\langle m_i,\mathcal{D}m_j\rangle\langle m_k,\mathcal{D}m_\ell\rangle \fint_{\mathbb{S}(\mathfrak{n})}v_iv_jv_kv_\ell\, d\sigma(v)\\
        =&\sum_{i=2}^{n}\langle m_i,\mathcal{D}m_i\rangle^2\fint_{\mathbb{S}(\mathfrak{n})}v_i^4\,d\sigma(v)+\sum_{\substack{2\leq i,j\leq n\\ i\neq j}}\langle m_i,\mathcal{D}m_i\rangle\langle m_j,\mathcal{D}m_j\rangle\fint_{\mathbb{S}(\mathfrak{n})}v_i^2v_j^2\, d\sigma(v)\\
        &\qquad\qquad\qquad\qquad\qquad\qquad\qquad\qquad\qquad+2\sum_{\substack{2\leq i,j\leq n\\ i\neq j}}\langle m_i,\mathcal{D}m_j\rangle^2\fint_{\mathbb{S}(\mathfrak{n})}v_i^2v_j^2\, d\sigma(v),
        \nonumber
    \end{split}
\end{equation}
where the last equality comes from the fact that integrals of odd functions on $\mathbb{S}(\mathfrak{n})$ are null.
By direct computation or using formulas stated at the beginning of \cite[Section 2c]{Kowalski1986Besicovitch-typeSubmanifolds}, we have that
\begin{equation}
    \begin{split}
        \fint_{\mathbb{S}(\mathfrak{n})}\gamma(v)^2 d\sigma(v)=\frac{3}{4n^2-1}\sum_{i=2}^{n}\langle m_i,\mathcal{D}m_i\rangle^2
        +&\frac{1}{4n^2-1}\sum_{i\neq k=2}^{n}\langle m_i,\mathcal{D}m_i\rangle\langle m_k,\mathcal{D}m_k\rangle\\
        +&\frac{2}{4n^2-1}\sum_{i\neq j=2}^{n}\langle m_i,\mathcal{D}m_j\rangle^2\\
        =\frac{2}{4n^2-1}\sum_{i,j=2}^{n}\langle m_i,\mathcal{D}m_j\rangle^2+&\frac{1}{4n^2-1}\left(\sum_{i=2}^{n}\langle m_i,\mathcal{D}m_i\rangle\right)^2.
        \label{numerooo15}
    \end{split}
    \end{equation}
Since the matrix $\mathcal{D}$ is symmetric, the well-known expression $\text{Tr}(\mathcal{D}^2)=\sum_{i,j=1}^{n}\langle m_i,\mathcal{D}m_j\rangle^2$, implies
\begin{equation}
    \sum_{i,j=2}^{n}\langle m_i,\mathcal{D}m_j\rangle^2=\mathrm{Tr}(\mathcal{D}^2)-2\sum_{i=2}^{n}\langle m_i,\mathcal{D}\mathfrak{n}\rangle^2-\langle\mathfrak{n},\mathcal{D}\mathfrak{n}\rangle^2.
    \label{numerooo11}
\end{equation}
Furthermore, since $\lvert \mathcal{D}\mathfrak{n}\rvert^2=\langle \mathfrak{n},\mathcal{D}^2\mathfrak{n}\rangle$, we also have
\begin{equation}
    \langle \mathfrak{n},\mathcal{D}^2\mathfrak{n}\rangle=\sum_{i=1}^{n} \langle m_i,\mathcal{D}\mathfrak{n}\rangle^2=\sum_{i=2}^{n} \langle m_i,\mathcal{D}\mathfrak{n}\rangle^2+ \langle \mathfrak{n},\mathcal{D}\mathfrak{n}\rangle^2.
    \label{numerooo12}
\end{equation}
We put \eqref{numerooo11} and \eqref{numerooo12} together, and infer that
\begin{equation}
	\begin{split}
    	\sum_{i,j=2}^{n}\langle m_i,\mathcal{D}m_j\rangle^2&=\mathrm{Tr}(\mathcal{D}^2)-2\bigl(\langle \mathfrak{n},\mathcal{D}^2\mathfrak{n}\rangle-\langle \mathfrak{n},\mathcal{D}\mathfrak{n}\rangle^2\bigr)-\langle\mathfrak{n},\mathcal{D}\mathfrak{n}\rangle^2\\
    	&=\mathrm{Tr}(\mathcal{D}^2)-2\langle \mathfrak{n},\mathcal{D}^2\mathfrak{n}\rangle+\langle\mathfrak{n},\mathcal{D}\mathfrak{n}\rangle^2.
    \end{split}
    \label{numerooo14}
\end{equation}
Finally, we plug \eqref{numerooo14} into \eqref{numerooo15} and conclude that
    \begin{equation}
    \begin{split}
        \fint_{\mathbb{S}(\mathfrak{n})}\gamma(v)^2 \,d\sigma(v)
        =&\frac{2\text{Tr}(\mathcal{D}^2)-4\langle \mathfrak{n},\mathcal{D}^2\mathfrak{n}\rangle+2\langle \mathfrak{n},\mathcal{D}\mathfrak{n}\rangle^2+\left(\text{Tr}(\mathcal{D})-\langle \mathfrak{n},\mathcal{D}\mathfrak{n}\rangle\right)^2}{(n-1)(n+1)}.
        \label{eq:exprV_a}
    \end{split}
    \end{equation}

We are left to study the integral (VI) in \eqref{equizias}. Since $P_\mathfrak{n}$ is the orthogonal projection on $\mathfrak{n}^\perp$, we have that
\begin{equation}
\begin{split}
    \lvert P_{\mathfrak{n}}(\mathcal{D}v)\rvert^2=&\sum_{i=2}^{n} \langle m_i, \mathcal{D}v\rangle^2=\sum_{i,j,k=2}^{n}v_j v_k\langle m_i, \mathcal{D}m_j\rangle\langle m_i, \mathcal{D}m_k\rangle.
    \nonumber
\end{split}
\end{equation}
Furthermore, taking the integral over the sphere we have
\begin{equation}
\begin{split}
    \fint_{\mathbb{S}(\mathfrak{n})}\lvert P_{\mathfrak{n}}(\mathcal{D}v)\rvert^2&d\sigma(v)=\sum_{i,j,k=2}^{n}\langle m_i, \mathcal{D}m_j\rangle\langle m_i, \mathcal{D}m_k\rangle\fint_{\mathbb{S}(\mathfrak{n})}v_j v_k\,d\sigma(v)\\
    =&\sum_{i,j=2}^{n}\langle m_i, \mathcal{D}m_j\rangle^2\fint_{\mathbb{S}(\mathfrak{n})}v_j^2\,d\sigma(v)=\frac{1}{n-1}\sum_{i,j=2}^{n}\langle m_i, \mathcal{D}m_j\rangle^2\\
    \overset{\eqref{numerooo14}}{=}&\frac{\mathrm{Tr}(\mathcal{D}^2)-2\langle \mathfrak{n},\mathcal{D}^2\mathfrak{n}\rangle+\langle\mathfrak{n},\mathcal{D}\mathfrak{n}\rangle^2}{n-1}.
    \label{eq:exprVI_a}
    \end{split}
\end{equation}
Finally we gather \eqref{equizias}, \eqref{eq:exprV_a} and \eqref{eq:exprVI_a}, and we obtain
\begin{equation}
    \begin{split}
\frac{c^2\mathfrak{e}(\mathcal{X})}{\mathcal{C}_n\sigma(\mathbb{S}(\mathfrak{n}))}&=
-\frac{n+1}{8}\bigg(\frac{2\text{Tr}(\mathcal{D}^2)-4\langle \mathfrak{n},\mathcal{D}^2\mathfrak{n}\rangle+2\langle \mathfrak{n},\mathcal{D}\mathfrak{n}\rangle^2+\left(\text{Tr}(\mathcal{D})-\langle \mathfrak{n},\mathcal{D}\mathfrak{n}\rangle\right)^2}{(n-1)(n+1)}\bigg)\\
        &\qquad\qquad-\frac{1}{4}+\frac{1}{2}\bigg(\frac{\text{Tr}(\mathcal{D}^2)-2\langle \mathfrak{n},\mathcal{D}^2\mathfrak{n}\rangle+\langle \mathfrak{n},\mathcal{D}\mathfrak{n}\rangle^2}{n-1}\bigg)\\
        &=\frac{\text{Tr}(\mathcal{D}^2)-2\langle \mathfrak{n},\mathcal{D}^2\mathfrak{n}\rangle+\langle \mathfrak{n},\mathcal{D}\mathfrak{n}\rangle^2}{4(n-1)}-\frac{1}{4}-\frac{(\text{Tr}(\mathcal{D})-\langle \mathfrak{n},\mathcal{D}\mathfrak{n}\rangle)^2}{8(n-1)}.
        \nonumber
    \end{split}
\end{equation}
where the last identity is obtained from the previous ones with few algebraic computations.
\end{proof}

\vv

We finally gather Proposition \ref{TEXP2} and the other results of this appendix as in \cite[Theorem B.13]{MerloG1cod}, and obtain the following result:
\begin{proposition}\label{unif:expansion}
Let $\mu\in \mathcal U(n+1)$ be such that $\supp(\mu)\subseteq\mathbb K(0,\mathcal D, -1)$. For any $p\in \mathbb{P}^n\setminus \{u\in\mathbb{K}(0,\mathcal{D},-1):u_H\in\mathrm{Ker}(\mathcal{D})\}$ we have
\[
	\frac{\text{Tr}(\mathcal{D}^2)-2\langle \mathfrak{n},\mathcal{D}^2\mathfrak{n}\rangle+\langle \mathfrak{n},\mathcal{D}\mathfrak{n}\rangle^2}{4(n-1)}-\frac{1}{4}-\frac{(\text{Tr}(\mathcal{D})-\langle \mathfrak{n},\mathcal{D}\mathfrak{n}\rangle)^2}{8(n-1)}=0,
\]
where $\mathfrak n \coloneqq \mathfrak n(p)=\mathcal D p/|\mathcal D(p)|$.
\end{proposition}

\printbibliography

\end{document}